\documentclass[11pt,reqno]{amsart} 
\usepackage[utf8]{inputenc}
\usepackage{mathptmx}
\usepackage{amsmath}
\usepackage{amscd}
\usepackage{amssymb, bm}
\usepackage{amsthm}
\usepackage{xspace}
\usepackage[all,tips]{xy}
\usepackage[dvips]{graphicx}
\usepackage{verbatim}
\usepackage{syntonly}
\usepackage{hyperref}
\usepackage{graphics, fullpage,color, epsfig,url}
\usepackage{indentfirst}
\usepackage{esint}
\usepackage{enumitem}
\usepackage[dvipsnames]{xcolor}
\usepackage{tensor}
\usepackage{amsrefs, eucal}
\usepackage{bbm}
\usepackage{ulem}

\theoremstyle{plain}
\newtheorem{thm}{Theorem}[section]
\newtheorem{lem}[thm]{Lemma}
\newtheorem{prop}[thm]{Proposition}
\newtheorem{defn}[thm]{Definition}
\newtheorem{cor}[thm]{Corollary}

\newtheorem{ex}{Example}

\newtheorem*{defin}{Definition}
\theoremstyle{remark}
\newtheorem{rem}[thm]{Remark}

\newenvironment{pf}
{\begin{proof}} {\end{proof}}

\newcommand{\disp}{\displaystyle}
\DeclareMathOperator{\dist}{dist}

\DeclareMathOperator{\di}{div}

\newcommand{\eps}{\varepsilon}
\newcommand{\vp}{\varphi}

\newcommand{\al}{\alpha}
\newcommand{\be}{\beta}
\newcommand{\ga}{\gamma}
\newcommand{\de}{\delta}

\newcommand{\Ga}{\Gamma}
\newcommand{\te}{\theta}
\newcommand{\la}{\lambda}
\newcommand{\La}{\Lambda}
\newcommand{\om}{\omega}
\newcommand{\Om}{\Omega}
\newcommand{\si}{\sigma}

\newcommand{\iny}{\infty}
\newcommand{\tr}{\text{tr}}
\newcommand{\del}{ \partial}
\newcommand{\su}{\subset}
\newcommand{\LP}{\Delta}
\newcommand{\gr}{\nabla}

\newcommand{\norm}[1]{\left\| #1\right\|}
\newcommand{\innp}[1]{\left< #1 \right>}
\newcommand{\abs}[1]{\left\vert#1\right\vert}
\newcommand{\set}[1]{\left\{#1\right\}}
\newcommand{\brac}[1]{\left[#1\right]}
\newcommand{\pr}[1]{\left( #1 \right) }

\newcommand{\N}{\ensuremath{\mathbb{N}}}

\newcommand{\R}{\ensuremath{\mathbb{R}}}
\newcommand{\Z}{\ensuremath{\mathbb{Z}}}

\begin{document}

\title{Variable-coefficient parabolic theory as a \\ high-dimensional limit of elliptic theory}

\author{Blair Davey}
\address[B. Davey]{Department of Mathematical Sciences, Montana State University, Bozeman, MT 59717}
\email{\textcolor{blue}{\href{mailto:}{blairdavey@montana.edu}}}
\thanks{B. D. has been partially supported by the NSF LEAPS-MPS grant DMS-2137743, and a Simons Foundation Collaboration Grant 430198.}

\author{Mariana Smit Vega Garcia}
\address[M. Smit Vega Garcia]{Department of Mathematics, Western Washington University, Bellingham, WA 98225}
\email{\textcolor{blue}{\href{mailto:}{smitvem@wwu.edu}}}
\thanks{M.S.V.G has been partially supported by the NSF grant DMS-2054282.}

\maketitle

\begin{abstract}

This paper continues the study initiated in \cite{Dav18}, where a high-dimensional limiting technique was developed and used to prove certain parabolic theorems from their elliptic counterparts. 
In this article, we extend these ideas to the variable-coefficient setting. 
This generalized technique is demonstrated through new proofs of three important theorems for variable-coefficient heat operators, one of which establishes a result that is, to the best of our knowledge, also new.
Specifically, we give new proofs of $L^2 \to L^2$ Carleman estimates and the monotonicity of  Almgren-type frequency functions, and we prove a new monotonicity of Alt-Caffarelli-Friedman-type functions.
The proofs in this article rely only on their related elliptic theorems and a limiting argument. 
That is, each parabolic theorem is proved by taking a high-dimensional limit of a related elliptic result.
\end{abstract}

\section{Introduction}
\label{S:intro}

In this paper, we explore the connections between the elliptic and parabolic theory of partial differential equations, generalizing the work done by the first-named author for constant-coefficient equations in \cite{Dav18}. 
Specifically, we generalize the ideas from \cite{Dav18} and establish a technique that can be used to prove variable-coefficient parabolic theorems from their appropriate elliptic counterparts. 
The key idea is that certain parabolic estimates may be obtained by taking high-dimensional limits of their corresponding elliptic results. 
We obtain information about solutions to $\text{div}(A\nabla u)+\partial_t u=0$ on the parabolic side by analyzing the behavior of solutions to non-homogeneous equations of the form $\text{div}(\kappa \nabla v)=\kappa \ell$ on the elliptic side.
Here, $A$ has a specific structure, $v$ and $\kappa$ are defined in terms of $u$ and $A$ (see Lemmas \ref{ChainRuleLem1} and \ref{ChainRuleLem}), respectively, and $\ell$ depends on both $u$ and $A$. 
Rewriting our elliptic equation as $\disp \kappa^{-1}\text{div}(\kappa\nabla v)=\ell$, we notice that the associated operator is a special type of Witten Laplacian, or weighted Laplacian (see 2.4 in \cite{G}, and also \cite{G}, \cite{FLL}, \cite{CESS}, \cite{L}, \cite{CES}, \cite{M}, \cite{LPN}, and \cite{LX}). 
From this perspective, the ideas in this article show how to obtain results for variable-coefficient parabolic operators from those for the Witten Laplacian.

Perelman first considered parabolic theory as a high-dimensional limit of elliptic theory in \cite{P}. 
This general principle was discussed in the blog of Tao \cite{Tao}, modified in the coursenotes of Sverak \cite{S}, then developed and applied in \cite{Dav18}. 
In our setting, we follow the ideas from \cite{Dav18}; namely, we use classical probabilistic formulae, essentially going back to Wiener \cite{W}, with a slight modification used by Sverak in \cite{S}. 
However, to account for the presence of variable-coefficients, we have modified (and complicated) the change of variables formula from \cite{Dav18}.
Once the general framework has been established, we demonstrate the utility of this technique by establishing three new proofs of theorems regarding variable-coefficient parabolic equations.
In comparison to the results of \cite{Dav18} for constant-coefficient equations, the techniques here are substantially modified to account for the variable-coefficients.

At the heart of the high-dimensional limiting process is a sequence of maps $F_{d,n} : \R^{d \times n} \to \R^d \times \R_+$ which take high-dimensional space points $y \in \R^{d \times n}$ on the elliptic side to space-time points $\pr{x,t} \in \R^d \times \R_+$ on the parabolic side.
These maps can be viewed from three perspectives.
First, they naturally arise through certain random walk processes.
Second, we can use the chain rule to see how these maps connect elliptic and parabolic operators.
That is, suppose we are given some parabolic function $u = u\pr{x,t}$ defined on $\R^d \times \R_+ $ and we define $v_n = v_n(y)$ on $\R^{d \times n}$ by $v_n(y) = u\pr{F_{d,n}(y)}$.
A computation involving only the chain rules shows that if $\di \pr{A \gr u} + \del_t u = 0$, then $\di(\kappa_n \nabla v_n)=\kappa_n \, \ell_n$, where $\kappa_n$ and $\ell_n$ are defined in terms of $u$ and $A$.
In other words, the maps $F_{d,n}$ provide a way of constructing a sequence of ``elliptic" functions from a given ``parabolic" function.
Finally, we can see the power of these maps through their pushforward measures.
In the constant-coefficient case, the pushforward by $F_{d,n}$ of the Lebesgue measure weighted with the fundamental solution of the Laplacian is a space-time Lebesgue measure that is weighted by a function that approximates the Gaussian.
That is, the pushforward connects the elliptic and parabolic fundamental solutions.
In our variable-coefficient setting, we do not have explicit descriptions of the fundamental solutions, so this connection is not as precise, but it mimics the constant-coefficient behavior.

Once we have the transformation maps, our general technique is as follows:
Given a parabolic function $u$, we use $F_{d, n}$ to construct a sequence of functions $\set{v_n}$. 
We apply an elliptic result to each of the $v_n$ functions, and then use the pushforward relationships to reinterpret this result in terms of $u$.
After a limiting process, we arrive at a result for the original parabolic function $u$.

Our first new proof is of an $L^2 \to L^2$ Carleman estimate for parabolic operators of the form $\text{div}(A\nabla)+\partial_t$. 
For elliptic operators, the original Carleman estimates are attributed to Carleman \cite{Car39}, with subsequent advances by Cordes \cite{Cor56}, Aronszajn \cite{Aro57} and Aronszajn et al \cite{AKS62}.
Significant contributions to the theory of elliptic Carleman estimates with applications to strong unique continuation include the work of Jerison-Kenig \cite{JK85}, Sogge \cite{Sog90}, Koch-Tataru \cite{KT01}, and the references therein.
When $A=I$, the parabolic Carleman estimate was proved by Escauriaza in \cite{E}; see also \cite{EV01}, \cite{EF03}, \cite{Ves03}, \cite{Fer03}, \cite{Ngu10} for (variable-coefficient) generalizations that followed.
The article \cite{KT09} of Koch and Tataru provides a nice overview of parabolic Carleman estimates, and the results therein apply to very general parabolic operators. 

Our second novel proof shows that Almgren-type frequency functions \cite{Alm79} associated with parabolic operators of type $\text{div}(A\nabla)+\partial_t$ are monotonically non-decreasing. 
In the variable-coefficient elliptic setting, Almgren-type monotonicity formulas have been used to establish unique continuation results, see \cite{GL86} and \cite{GL87}.
When $A = I$, monotonicity of the parabolic frequency function was originally proved by Poon in \cite{P2} and used to establish strong unique continuation results for caloric functions. 
Since then, the frequency function approach has been used extensively in the study of parabolic unique continuation problems; see for example \cite{EFV06}, \cite{EKPV06}, \cite{GK18}, \cite{Zha18}, and \cite{KQ22}.

Both the Carleman estimates and the monotonicity of the Almgren-type frequency functions, motivated by their elliptic counterparts, allowed the authors of \cite{E} and \cite{P2} (for example) to use the established techniques for elliptic theory to prove strong unique continuation for solutions to the heat equation. 
Since then, a wealth of unique continuation results for variable-coefficient heat operators have been established using both Carleman estimates and frequency functions.
Recently, Carleman techniques have been used to establish space-like quantitative uniqueness of solutions to variable-coefficient parabolic equations that are averaged in time \cite{Zhu18} and at a particular time-slice \cite{AB22}.

Almgren-type monotonicity formulas have also been used extensively in the context of free boundary problems to obtain regularity of solutions and of the free boundary. 
In the case of the parabolic constant-coefficient Signorini problem, a truncated version of Almgren's monotonicity formula was proved in \cite{DGPT}, leading to the optimal regularity of solutions and analysis of the free boundary. 
Elliptic variable-coefficient  Almgren-type monotonicity formulas have also been widely used to study numerous free boundary problems; see for example \cite{Gui}, \cite{GSVG}, \cite{GPSVG}, \cite{GPSVG2}, \cite{JPSVG}, \cite{BBG}, \cite{DT}. 
We refer the reader to \cite{PSU} for a beautiful introduction to the use of Almgren-type monotonicity formulas in free boundary problems.

Our third new proof establishes a result that is, to the best of our knowledge, also new. 
In other words, our technique leads to an original parabolic result: 
We prove an Alt-Caffarelli-Friedman-type (ACF-type) monotonicity formula for variable-coefficient heat operators.
The groundbreaking work of Alt, Caffarelli, and Friedman in \cite{ACF} introduced the use of one such monotonicity formula to study two-phase free boundary elliptic problems. 
Another version of this formula was proved in \cite{C} by Caffarelli and extended by Caffarelli and Kenig in \cite{CK} to establish the regularity of solutions to parabolic equations and their singular perturbations. 
Later on, Caffarelli, Jerison, and Kenig \cite{CJK} considered non-homogeneous elliptic equations in which  the right-hand side of the equation need not vanish at the free boundary. 
Their main result is not a monotonicity result per se, but rather a clever uniform bound on the monotonicity functional, which is just as useful as monotonicity itself. 
Later on, Matevosyan and Petrosyan \cite{MP} further extended that result, proving an 
almost monotonicity estimate for non-homogeneous elliptic and parabolic operators with variable coefficients. 
We refer the interested reader to \cite{PSU} and \cite{Apu18} for an introduction to the use of the ACF monotonicity formula; \cite{PSU} deals with the elliptic setting, while \cite{Apu18} addresses both the parabolic and elliptic settings.  
In the context of almost minimizers of variable-coefficient Bernoulli-type functionals, ACF-type monotonicity formulas have also been used to study the free boundary, see for example \cite{DESVGT}.

Given that our transformation maps relate solutions to homogeneous parabolic equations to a sequence of solutions to non-homogeneous elliptic equations, our proofs of the monotonicity results have added complexity.
More specifically, given that each $v_n$ solves an elliptic equations with a right-hand side, we require elliptic monotonicity results that apply to solutions to non-homogeneous equations.
Therefore, the first step in each parabolic monotonicity proof is to establish the related underlying variable-coefficient elliptic result with a right-hand side.
These new elliptic results, which appear at the beginning of Sections \ref{S:Almgren} and \ref{S:ACF}, may be of independent interest.
We point out that these elliptic estimates generalize both the constant-coefficient and the homogeneous results.

The monotonicity of frequency functions is often proved by showing that the derivative of the frequency function has a fixed sign.
In our proofs of the monotonicity of variable-coefficient Almgren-type and ACF-type frequency functions, we show that each parabolic frequency function may be described as a pointwise limit of a sequence of elliptic frequency functions. 
While we know that each such elliptic frequency function is differentiable and almost monotonic, that is not sufficient to guarantee the differentiability of the corresponding parabolic frequency function. 
Therefore, our proofs rely on a more delicate analysis that allows us to conclude directly that our parabolic frequency functions are monotonic.
These ideas are described at the ends of the proofs of Theorems \ref{T:parabolicALM} and \ref{T:ACFp}.

We work with time-independent variable-coefficient operators for which the coefficient matrix has a specific structure.
That is, we consider symmetric matrices of the form $A = G G^T$, where $G$ is the Jacobian of some invertible map $g : \R^d \to \R^d$.
Clearly, there are bounded, elliptic coefficient matrices $A$ whose structure is not of this form.
However, if $A = I$, then $A$ is associated to the identity map $g(z) = z$ with Jacobian $G = I$, showing that this structural condition on $A$ is a reasonable generalization to constant-coefficient operators.
In subsequent work, we will explore operators with even more general coefficient matrices.

The article is organized as follows. 
In Sections \ref{S:prelim} and \ref{S:measure}, we develop the framework that connects the elliptic and parabolic theory.
That is, we introduce and examine the maps $F_{d,n} : \R^{d \times n} \to \R^d \times \R_+$ that take points in the high-dimensional (elliptic) space to space-time (parabolic) points.
Section \ref{S:prelim} takes the perspective of random walks to introduce these maps, then uses the chain rule to explore how these maps relate elliptic and parabolic operators.
Section \ref{S:measure} examines these maps from a measure theoretic perspective.
In particular, we present the pushforward computations and describe how the integrals on the elliptic side are related to those on the parabolic side.
These two sections contain a collection of calculations and statements that will be referred to throughout the article.

The $L^2 \to L^2$ variable-coefficient parabolic Carleman estimate is proved in Section \ref{S:Carleman}. 
The Almgren-type frequency function theorem for variable-coefficient heat operators is presented in Section \ref{S:Almgren}. 
Finally, Section \ref{S:ACF} contains the monotonicity result for Alt-Caffarelli-Friendman-type energies associated with variable-coefficient heat operators. 

\subsection*{Acknowledgements}
We thank Stefan Steinerberger for interesting discussions.

\section{Elliptic-to-parabolic transformations}
\label{S:prelim}

In this section, we construct the transformations that connect so-called parabolic functions $u = u(x,t)$ defined on $\R^d \times \pr{0, T}$ to a sequence of elliptic functions $v_n = v_n(y)$ defined in $\R^{d \times n}$ for all $n \in \N$.
More specifically, for each $n \in \N$, we construct a mapping of the form
\begin{align*}
F_{d,n} :  \R^{d \times n} &\to \R^d \times \R_+ \\
 y &\mapsto (x,t)
\end{align*}
that takes element $y$ in (high-dimensional) space $\R^{d \times n}$ to elements $(x,t)$ in space-time $\R^d \times \R_+$.
Given a function $u = u(x,t)$ defined on a space-time domain, i.e. a subset of $\R^d \times \R_+$, we use $F_{d,n}$ to define a function $v_n = v_n(y)$ on the space $\R^{d \times n}$ by setting $v_n(y) = u(F_{d,n}(y))$.
As we show below via the chain rule, if $u$ is a solution to a backward parabolic equation, then each $v_n$ is a solution to some non-homogeneous elliptic equation.
This observation explains why we think of $u$ as a parabolic function and of each $v_n$ as an elliptic function.
Moreover, as $n$ becomes large, the function $v_n$ behaves (heuristically) more and more like a solution to a homogeneous elliptic equation.
As such, the transformation $F_{d,n}$ becomes more useful to our purposes as $n \to \iny$, thereby illuminating why the notion of a high-dimensional limit is relevant here.
From another perspective, when we use $F_{d,n}$ to pushforward measures on spheres and balls in $\R^{d \times n}$, we produce measures in space-time that are weighted by approximations to generalized Gaussians.
This perspective is explored in the next section, Section \ref{S:measure}.

We can think of the transformations $F_{d,n} : \R^{d \times n} \to \R^d \times \R_+$ in a number of ways.
On one hand, as we show in this section, these maps are constructed so that each $v_n$ solves an elliptic equation whenever $u$ solves a parabolic equation.
This viewpoint is purely computational as these relationships are illuminated by the chain rule.
This perspective is perhaps the most (easily) checkable, but it is somewhat mysterious.
Random walks are the underlying mechanism in defining these mappings, so we begin with that perspective.

Let $y = \pr{y_{1,1}, y_{1,2}, \ldots, y_{1,n}, \ldots, y_{d,1}, y_{d,2}, \ldots, y_{d,n}} \in \R^{d \times n}$ denote the variables that play the role of the ``random steps" in our random walk.
For some $t > 0$, assume that $y$ satisfies
\begin{equation}
\label{tDef}
2dt = \sum_{i = 1}^d \sum_{j = 1}^n y_{i,j}^2.
\end{equation}
In this setting, we do not fix the step size but simply assume that $y$ is uniformly distributed over the sphere of radius $\sqrt{2dt}$.
Define 
\begin{equation}
\label{zDef}
z_i = y_{i,1} + y_{i,2} + \ldots + y_{i, n} \quad \text{ for } i = 1, \ldots, d
\end{equation}
so that $z = \pr{z_1, \ldots, z_d} \in \R^d$.
In the notation from \cite{Dav18}, we define $f_{d,n} : \R^{d \times n} \to \R^d$ so that 
\begin{equation}
\label{fdnDef}
z = f_{d,n}(y).
\end{equation}
Now let 
$$g : \R^d \to \R^d$$ 
be an invertible function with inverse function 
$$h: \R^d \to \R^d.$$
Set $x = g(z) \in \R^d$ so that
\begin{equation}
\label{xDef}
x_i = g_i(z)  \quad \text{ for } i = 1, \ldots, d
\end{equation}
and then since $z = h(x)$ we have 
\begin{equation*}
z_i = h_i(x)\quad \text{ for } i = 1, \ldots, d.
\end{equation*}
For the moment, assume that $g$ is sufficiently regular for the computations below to hold in a weak sense. 
When further regularity is needed, we specify it. 

The Jacobian of $g = \pr{g_1, \ldots, g_d}$ is a $d \times d$ invertible matrix function whose inverse matrix is the Jacobian of $h = \pr{h_1, \ldots, h_d}$.
Let $G$ and $H$ denote the Jacobian matrices of $g$ and $h$, respectively.
That is,
\begin{equation}
\label{Jacobians}
G(z) = \begin{bmatrix}
\frac{\del g_1}{\del z_1} & \frac{\del g_1}{\del z_2} \ldots & \frac{\del g_1}{\del z_d}  \\
\frac{\del g_2}{\del z_1} & \frac{\del g_2}{\del z_2} \ldots & \frac{\del g_2}{\del z_d}  \\
\vdots & \ddots & \vdots \\
\frac{\del g_d}{\del z_1} & \frac{\del g_d}{\del z_2} \ldots & \frac{\del g_d}{\del z_d}
\end{bmatrix},
\quad
H(x) = \begin{bmatrix}
\frac{\del h_1}{\del x_1} & \frac{\del h_1}{\del x_2}  \ldots & \frac{\del h_1}{\del x_d} \\
\frac{\del h_2}{\del x_1} & \frac{\del h_2}{\del x_2}  \ldots & \frac{\del h_2}{\del x_d} \\
\vdots & \ddots & \vdots \\
\frac{\del h_d}{\del x_1} & \frac{\del h_d}{\del x_2} \ldots & \frac{\del h_d}{\del x_d}
\end{bmatrix}.
\end{equation}
Let $\ga(z) = \det G(z)$ and $\eta(x) = \det H(x)$.
From \eqref{Jacobians}, we obtain
\begin{align*}
    I &= G(z)H(x) = G(z)H(g(z)) = G(h(x))H(x) \\
    I &= H(x)G(z)= H(x) G(h(x)) = H(g(z))G(z) \\
    1 &= \ga(z) \eta(x) = \ga(h(x)) \eta(x) = \ga(z) \eta(g(z)).
\end{align*}

Now we collect some observations that follow from the Chain Rule.
For the first lemma, we describe how these determinants and their derivatives transform through the map $g \circ f_{d,n}$.

\begin{lem}[Determinant Chain Rule Lemma]
\label{ChainRuleLem1}
For $\ga(z) = \det G(z)$ and $\eta(x) = \det H(x)$ as above, define $\kappa_n : \R^{d \times n} \to \R$ to satisfy 
\begin{equation}
    \label{kappaDef}
    \kappa_n(y) = \ga(f_{d,n}(y)) = \ga(z) = \frac 1 {\eta\pr{x}} = \frac 1 {\eta\pr{g(f_{d,n}(y))}}.
\end{equation}
Then
\begin{equation*}
    \begin{aligned}
    \frac{\del \log \kappa_n}{\del y_{i,j}}
    &= \tr\pr{H \pr{ g(z)} \, \frac{\del G(z)}{\del z_i}} 
    = \tr\pr{H \pr{x} \, \frac{\del G}{\del z_i}\pr{h(x)}} \\
    y \cdot \gr_y \log \kappa_n
    &= \tr\pr{H(g(z)) \innp{\gr G(z), z}}
    = \tr\pr{H(x) \innp{\gr_z G(h(x)), h(x)}}.
    \end{aligned}
\end{equation*}
\end{lem}

\begin{pf}
Since $\displaystyle \frac{\del z_k}{\del y_{i,j}} = \de_{k, i}$, then an application of Jacobi's formula shows that
\begin{align}
\label{logExpression}
\frac{\del \log \kappa_n(y)}{\del y_{i,j}} 
&= \frac{\del \log \ga(z)}{\del z_i} 
= \tr \pr{G^{-1}(z) \frac{\del G(z)}{\del z_i}}
= \tr\pr{H \pr{ g(z)} \, \frac{\del G(z)}{\del z_i}} . 
\end{align}
We then get
\begin{align*}
    y \cdot \gr_y \log \kappa_n
    = \sum_{i=1}^d \sum_{j = 1}^n y_{i, j} \frac{\del \log \kappa_n(y)}{\del y_{i,j}} 
    &= \sum_{i=1}^d z_{i} \frac{\del \log \ga(z)}{\del z_{i}} 
    = \innp{\gr \log \ga(z), z}
    = \tr\brac{H(g(z)) \innp{\gr G(z) , z}} .
\end{align*}
\end{pf}

The next set of observations shows how the derivatives of some parabolic function $u : \R^d \times \pr{0, T} \to \R$ can be related to those of its associated elliptic functions $v_n: B_{\sqrt{2dT}} \su \R^{d \times n} \to \R$.
That is, given a parabolic function $u = u(x,t)$, we define the elliptic functions $v_n = v_n(y) = u(F_{d,n}(y))$, where
\begin{equation}
\label{FdnDef}
F_{d,n}(y)=(x,t)=\pr{\pr{g\circ f_{d,n}}(y), \frac{|y|^2}{2d}}.
\end{equation}
The following set of results justifies why we refer to $u$ as parabolic and each $v_n$ as elliptic.

\begin{lem}[Solution Chain Rule Lemma]
\label{ChainRuleLem}
Given $u : \R^d \times \pr{0, T} \to \R$, define $v_n: B_{\sqrt{2dT}} \su \R^{d \times n} \to \R$ to satisfy 
$$v_n(y) = u\pr{F_{d,n}(y)}.$$ 
Define $B = B(z)$ to be a $d \times d$ matrix function with entries 
$$b_{k, \ell} = \gr_z g_k \cdot \gr_z g_\ell = \sum_{i=1}^d \frac{\del g_k}{\del z_i}\frac{\del g_\ell}{\del z_i}.$$
That is, $B =G G^T$.
We then set $A(x) = B\pr{h(x)} = B(z)$ so that $A = H^{-1}\pr{H^{-1}}^T$.
Then
\begin{equation*}
    \begin{aligned}
    \frac{\del v_n}{\del y_{i,j}}
    &= \innp{\gr_x u, \frac{\del g}{\del z_i}}  + \frac{\del u}{\del t} \frac{y_{i, j}}{d} \\
    y \cdot \gr_y v_n
    &= \innp{\gr_x u, G(z) z} + 2 t \frac{\del u}{\del t}
    = \innp{A(x)\gr_x u, {H(x)^T h(x)}} + 2 t \frac{\del u}{\del t} \\
    \abs{\gr_y v_n}^2 
    &= n \abs{G(z)^T \gr_x u}^2 
    + \frac 2 d \frac{\del u}{\del t} \brac{\innp{\gr_x u,  G(z) \, z}
    + t \frac{\del u}{\del t} } \\
    &= n \innp{A(x) \gr_x u, \gr_x u} 
    + \frac 2 d \frac{\del u}{\del t} \brac{\innp{A(x)\gr_x u, H(x)^T h(x)}
    + t \frac{\del u}{\del t} }.
    \end{aligned}
\end{equation*}
Moreover, with $\kappa_n$ as in \eqref{kappaDef},
\begin{align}
    & \frac{\di_y \pr{ \kappa_n(y)\gr_y v_n} }{\kappa_n(y)} 
    = n \brac{\di_x \pr{A \gr_x u} +  \frac{\del u}{\del t}}
+ \frac 2 {d}\brac{ \innp{A\gr_x \frac{\del u}{\del t}, H^T h}  
+ \frac{\del u}{\del t} \frac{\tr\pr{H \innp{\gr_z G(h), h}}}2
+ t \frac{\del^2 u}{\del t^2} } ,
\label{chain}
\end{align}
where the expression on the right depends on $x$ and $t$.
\end{lem}

The proof of this result relies on the chain rule and can be found in Appendix \ref{p:ChainLemProof}.
When $u$ is a solution to a variable-coefficient backwards heat equation, we immediately reach the following consequence.

\begin{cor}[Chain Rule Corollary]
\label{ChainRuleCor}
If $u : \R^d \times \R_{+} \to \R$ is a solution to $\di_x \pr{A \gr_x u} + \del_t u = 0$ and we define $v_n: \R^{d \times n} \to \R$ to satisfy 
$$v_n(y) = u\pr{F_{d,n}(y)},$$ 
then
\begin{equation*}
\di_y \pr{ \kappa_n(y)\gr_y v_n}
= \kappa_n(y) \ell_n(y),
\end{equation*}
where
$$\ell_n = \frac 2 d \brac{\innp{A\gr_x \frac{\del u}{\del t}, H^T h}  
+ \frac 1 2 \frac{\del u}{\del t} \tr\pr{H \innp{\gr_z G(h), h}}
- t \di_x \pr{A \gr_x \frac{\del u}{\del t} }}.$$
\end{cor}

In summary, we have constructed a sequence of maps $F_{d,n}$, given in \eqref{FdnDef}, that serve as the connection between the elliptic and parabolic settings.
The following table describes these relationships and the notation that we use to describe our elliptic and parabolic settings.
Note that $F_{d,n}$ is the connection between the elliptic column and the parabolic column.
$$
\begin{array}{l|ll}
    & \textrm{Elliptic} & \textrm{Parabolic} \\
    \hline \hline 
    \textrm{Space} & \textrm{high-dimensional space} & \textrm{space-time} \\
    \textrm{Elements} & y \in \R^{d \times n} & \pr{x, t} \in \R^d \times \R_+ \\
    \textrm{Functions} & v_n = v_n(y) & u = u\pr{x,t} \\
    \textrm{PDEs} & \di_y \pr{\kappa_n \gr_y v_n} = \kappa_n \, \ell_n & \di_x \pr{A \gr_x u} + \del_t u = 0 \\
\end{array}
$$
Going forward, we write $\nabla u$ to indicate $\nabla_x u$, unless explicitly indicated otherwise.
Similarly, we write $\nabla v_n$ to indicate $\nabla_y v_n$, unless explicitly indicated otherwise.
That is, $u$ is understood to be a function of $x$ and $t$, while $v_n$ is a function of $y$, and all derivatives are interpreted appropriately.

The results above will be used extensively when we prove our parabolic theorems.
Therefore, we work with variable-coefficient operators for which the coefficient matrix has the specific structure that is described by the previous two results.
That is, we consider $A(x) = B(z)$ where $B = G G^T$ and $G$ is the Jacobian of some invertible map $g : \R^d \to \R^d$.
Clearly, there are bounded, elliptic coefficient matrices $A$ whose structure is not of this form.
However, if $A = I$, then $A$ is associated to the identity map $g(z) = z$ with Jacobian $G = I$, showing that this structural condition on $A$ is a reasonable generalization to constant-coefficient operators.

Before concluding this section, we examine some examples of non-trivial coefficient matrices that satisfy our structural condition.
For simplicity, we assume that $d=1$. 

Given a bounded, elliptic function $\la \le C(x)\le \La$, we construct a function $g$ such that the corresponding function $A(x)$ coincides with $C(x)$.
Recall that
\[
A(x)=G(h(x))G^T(h(x))=\brac{g'(h(x))}^2.
\]
For $C(x)=A(x)$ to hold, we need $\sqrt{C(x)}=g'(h(x))$. 
Since $g = h^{-1}$ and $h = g^{-1}$, then $\disp g'(h(x)) = \frac 1 {h'(x)}$, so we need to solve $h'(x)=\frac{1}{\sqrt{C(x)}}$.

\begin{ex}
Let $C(x)=\left(\frac{1+x^2}{2+x^2}\right)^2$.
Notice that $\frac{1}{4}\le C(x)\le 1$ and $\disp \frac 1 {\sqrt{C(x)}} = \frac{2+x^2}{1+x^2}$ so that 
\[
\int \disp \frac {dx} {\sqrt{C(x)}}
=\int\frac{2+x^2}{1+x^2}dx
=x+\arctan(x)+c.
\]
If we want $h(0) = 0$, then define $h(x) = x + \arctan(x)$ and take $g(z)=h^{-1}(z)$. 
\end{ex}

\begin{ex}
Let $C(x)=2+\sin(x)$ and notice that $1\le C(x)\le 3.$ 
We define
\[
h(x)
=\int \frac{1}{\sqrt{2+\sin(x)}}dx
=\frac{-2}{\sqrt{3}} {F\left(\frac{\pi-2x}{4}\Big|\frac{2}{3}\right)},
\]
where $F(x | m)$ is the elliptic integral of the first kind with parameter $m=k^2$. 
Again, we define $g(z)=h^{-1}(z)$.    
\end{ex}

\section{Integral relationships}
\label{S:measure}

In what follows, we examine how integrals of functions $\phi(x,t)$ defined in space-time $\R^d \times \R_+$ can be related to integrals of $\phi(F_{d,n}(y))$ in high-dimensional space $\R^{d \times n}$.
We start from a probabilistic viewpoint.
That is if $y \in \R^{d \times n}$ is uniformly distributed over a fixed sphere, we seek to determine the probability distribution of $x \in \R^d$, where $(x,t) = F_{d,n}(y)$.
As we show below, understanding this probability distribution on the parabolic side reduces to a pushforward computation which we carry out explicitly.
We then collect the consequences that will be used in our elliptic-to-parabolic proofs.

Let $S_t^n$ denote the sphere of radius $\sqrt{2dt}$ in $\R^{d\times n}$.
That is,
\begin{equation}
\label{StnDefn}
S_t^n=\set{y\in\R^{d\times n} \ : \ |y|^2 = 2dt}.
\end{equation}
Let $\sigma_{d n-1}^{t}$ denote the canonical surface measure on this sphere $S_t^n$.
We make the assumption that the vectors $y = (y_{1,1},y_{1,2},\ldots,y_{d,1},\ldots,y_{d,n}) \in \R^{d \times n}$ are uniformly distributed over $S_t^n$ with respect to $\sigma_{d n-1}^{t}$.
Set $\mu_{d,n}^{t}$ to be the normalized surface measure on the sphere $S^n_t$; that is,
\[
\mu_{d, n}^{t}=\frac{1}{|S^{d n-1}|(2dt)^{\frac{d n-1}{2}}} \sigma_{d n-1}^{t}.
\]
Our goal is to find the probability distribution of $x$ in the limit as $n\rightarrow\infty$, where $x=g(z) = g(f_{d,n}(y))$ and $z=f_{d,n}(y)$ is as in \eqref{fdnDef}. 

We first define an intermediate measure on $z \in \R^d$ for each fixed $t$ via the  pushforward as $\omega_{d,n}^t = f_{d, n} \# \mu_{d, n}^t$.
This is the pushforward from \cite{Dav18} and it is shown there that
\begin{align*}
d\omega_{d,n}^t 
&=  \frac{\abs{S^{d n-1 - d}} }{\abs{S^{d n-1}} \pr{2 d n t}^{\frac{d}{2}}} \pr{1 - \frac{z_1^2 + \ldots + z_{d}^2}{2 d n t}}^{\frac{d n - d - 2}{2}} \chi_{B_{nt}}(z) \, dz_1 \ldots dz_{d} \\
&=  \frac{\abs{S^{d n -1 - d}} }{\abs{S^{d n-1}} \pr{2d n t}^{\frac{d}{2}}} \pr{1 - \frac{\abs{z}^2}{2 d n t}}^{\frac{d n - d - 2}{2}} \chi_{B_{nt}}(z) \, dz,
\end{align*}
where we have introduced the notation
\begin{equation}
\label{BntDefn}
B_{nt}=\{z\in\R^d \ : \ |z|^2 < 2dnt\}
\end{equation}
for the ball of radius $\sqrt{2dnt}$ in $\R^d$.

The probability distribution of $x$ is the push-forward of $\mu_{d,n}^{t}$ by $g \circ f_{d, n}$:
 \[
 \nu_{d,n}^{t}:=(g\circ f_{d,n})\#(\mu_{d,n}^{t}).
 \]
 This implies, in particular, that for all $\varphi\in C_0(\R^d)$,
 \[
 \int_{\R^d}\varphi(x) \, d\nu_{d,n}^{t}(x)=\int_{S_t^n}\varphi(g(f_{d,n}(y))) \, d\mu_{d,n}^{t}(y).
 \]
Now observe that 
$$\nu_{d,n}^t = \pr{g \circ f_{d,n}} \# \mu_{d,n}^t =  g \# \pr{f_{d,n} \# \mu_{d,n}^t} = g \# \omega_{d,n}^t.$$
Since $g : \R^d \to \R^d$ is invertible, computing the pushforward here is the same as carrying out a change of variables.
By definition, we have that
\begin{align*}
\frac{\abs{S^{d n-1 - d}} }{\abs{S^{d n-1}} \pr{2 d n t}^{\frac{d}{2}}} \int_{\R^d} \vp\pr{g(z)}  \pr{1 - \frac{\abs{z}^2}{2 d n t}}^{\frac{d n - d - 2}{2}} \chi_{B_{nt}}(z) \, dz 
&= \int_{\R^d} \vp\pr{g(z)} d \om_{d, n}^t
= \int_{\R^d} \vp\pr{x} d\nu_{d, n}^t.
\end{align*}
Since $dx = \ga(z) \, dz$, $\eta(g(z)) \ga(z) =1$, and $z = h(g(z))$, then
\begin{align*}
\int_{\R^d} \vp\pr{g(z)} \pr{1 - \frac{\abs{z}^2}{2 d n t}}^{\frac{d n - d - 2}{2}} \chi_{B_{nt}}(z) \, dz
&= \int_{\R^d} \vp\pr{g\pr{z}} \pr{1 - \frac{\abs{z}^2}{2 d n t}}^{\frac{d n - d - 2}{2}} \eta\pr{g\pr{z}} \chi_{B_{nt}(z)}\, \ga(z)  \, dz \\
&= \int_{\R^d} \vp\pr{x} \eta\pr{x} \pr{1 - \frac{\abs{h\pr{x}}^2}{2 d n t}}^{\frac{d n - d - 2}{2}} \chi_{B_{nt}}(h(x))\, dx.
\end{align*}
It follows that $\disp \nu_{d,n}^t = \eta(x) \, K_{t,n}(h(x)) \, dx$, where we introduce
\begin{equation}
\label{Ktn}
K_{t,n}(x)
= K_{n}(x, t)
=\frac{\abs{S^{d n -1 - d}} }{\abs{S^{d n -1}} \pr{2 d n t}^{\frac{d}{2}}}  \pr{1 - \frac{\abs{x}^2}{2 d n t}}^{\frac{d n - d - 2}{2}}\chi_{B_{nt}}(x).
\end{equation}
As shown in \cite[Lemma 6]{Dav18},
\begin{align}
\label{KtnDefn}
    K_t(x) = K(x,t) := \lim_{n \to \iny} K_{t,n}(x)
    &= \pr{\frac 1 {4 \pi t}}^{\frac d 2}  \exp\pr{-\frac{\abs{x}^2}{4t}},
\end{align}
where the limit is pointwise. 
A much stronger version of convergence holds for this sequence, stated as follows.

\begin{lem}[Uniform convergence of $\set{K_n}$]
\label{L:uniform} 
Given any $t_0 >0$, the sequence $\set{K_n(x,t)}_{n=1}^\iny$ converges uniformly to $K(x,t)$ in $\R^d\times\{t\ge t_0\}$.
\end{lem}

The proof of this result can be found in Appendix \ref{p:uniform}.
In fact, the arguments there, in combination the proof of \cite[Lemma 1]{Dav18}, show that there exists $\mathcal{C}_d$ so that for every $n \in \N$ and every $\pr{x,t} \in \R^d \times \R_+$,
\begin{equation}
\label{KntKtBound}
K_{t,n}(x) \le \mathcal{C}_d K_{t}(x).
\end{equation}
This bound will be used in many of our subsequent arguments.

One might wonder how our transformed heat kernel relates to the standard one. 
The following lemma, proved in Appendix \ref{p:uKEqn}, shows that $K(h(x), t)$ is not a solution to the associated homogeneous variable-coefficient heat equation.

\begin{lem}[Kernel solution]
\label{kernelEquation}
With $u_K(x,t) = K\pr{h(x), t}$, it holds that
\begin{equation*}
    \di \pr{A \gr u_K} - \del_t u_K = - \frac{1}{2t} \sum_{ j, k, \ell=1}^d \frac{\del^2 g_k}{\del {z_j} \del z_\ell }(h(x)) \frac{\del h_\ell}{\del x_k} h_j\, u_K \ne 0.
\end{equation*}
\end{lem}

Summarizing our pushforward computations, we have established the following result:

\begin{lem}[cf. Lemma 2 and Lemma 3 in \cite{Dav18}]
\label{PFTS}
Let $S_{t}^n$, $K_{t,n}$ and $K_t$ be as given in \eqref{StnDefn}, \eqref{Ktn} and \eqref{KtnDefn}, respectively.
If $\vp: \R^d \to \R$ is integrable with respect to $K_{t}\pr{h(x)} d{x}$, then for every $n \in \N$, $\vp$ is integrable with respect to $K_{t,n}\pr{h(x)} d{x}$ and it holds that
\begin{align*}
& \fint_{S^n_t} \kappa_n(y)\vp\pr{g \circ f_{d,n}\pr{y}} d \si_{d n -1}^t
= \int_{S_t^n} \kappa_n(y) \vp\pr{g \circ f_{d,n}\pr{y}}  d\mu_{n,d}^t
= \int_{\R^d} \vp\pr{x} \, K_{t,n}\pr{h(x)} \, d{x},
\end{align*}
where $\kappa_n$ is as defined in \eqref{kappaDef}.
\end{lem}

\begin{rem}
Going forward, we often use the notation $d \si$ or $d \si(y)$ in place of $d \si_{d n -1}^t$.
\end{rem}

It is also helpful to interpret the measures $\nu_{d,n}^t$ as time slices of a space-time object, which comes from the projection of some global measure in the space $y\in\R^{d\times n}$ onto the space-time. 
To do so, we project a measure $\mu_{d,n}$ on $\R^{d\times n}$ to the space-time $\R^d\times \R_+$ by the map $F_{d,n}$ as given in \eqref{FdnDef}. 
With $m$ denoting Lebesgue measure, define
\[
\mu_{d,n}=\frac{d}{|S^{d n-1}||y|^{d n -2}}m.
\]
The ``slices'' of the measure $\mu_{d,n}$ by the spheres $S^t_n$ project by $F_{d,n}$ onto the measures $\nu_{d,n}^t$; that is,
\begin{equation*}
F_{d,n}\# \mu_{d,n}=\int_0^{\infty}(g\circ f_{d,n})\#(\mu_{d,n}^t)dt=\int_0^{\infty}\nu_{d,n}^tdt.
\end{equation*}
Let $B_t^n$ denote the open ball in $\R^{d \times n}$ for which $\del B_n^t = S_n^t$, i.e. 
\begin{equation}
\label{BtnDefn}
B_t^n=\set{y\in\R^{d\times n} \ : \ |y|^2 < 2dt}.
\end{equation}
Integrating the equality from Lemma \ref{PFTS} leads to the following result.

\begin{lem}[cf. Lemma 4 in \cite{Dav18}]
\label{PFT}
Let $B_t^n$, $K_{n}$ and $K$ be as given in \eqref{BtnDefn}, \eqref{Ktn} and \eqref{KtnDefn}, respectively.
If $\phi: \R^d \times \pr{0, T} \to \R$ is integrable with respect to $K\pr{h(x)} d{x} \, dt$, then for every $n \in \N$, $\phi$ is integrable with respect to $K_{n}\pr{h(x)} d{x} \, dt$ and for every $t \in \pr{0, T}$, it holds that
\begin{align*}
& \frac{1}{d \abs{S^{d n -1}}} \int_{B^n_t} \kappa_n(y) \phi\pr{F_{d,n}\pr{y}} \abs{y}^{2 - d n}  d{y} 
=  \int_0^t \int_{\R^d} \phi\pr{x,s} K_n\pr{h(x),s} d{x} \, d{s},
\end{align*}
where $\kappa_n$ is as defined in \eqref{kappaDef}.
In fact, if $T = \iny$, then
\begin{align*}
&\frac{1}{d \abs{S^{d n -1}}} \int_{\R^{d \times n}} \kappa_n(y) \phi\pr{F_{d,n}\pr{y}} \abs{y}^{2 - d n } d{y}
= \int_0^\iny \int_{\R^d} \phi\pr{x,s} \, K_n\pr{h(x),s} \, d{x} \, d{s}.
\end{align*}
\end{lem}

Lemmas \ref{PFTS} and \ref{PFT} will be used many times in our proofs below.
Lemma \ref{PFTS} transforms a $\kappa_n$-weighted integral over a sphere in $\R^{d \times n}$ to a $K_{t,n}$-weighted integral at a particular slice in space-time.
Similarly, Lemma \ref{PFT} transforms a weighted integral over a ball in $\R^{d \times n}$ to a $K_{t,n}$-weighted integral in space-time.
Of note, the elliptic weight in Lemma \ref{PFT} is $\kappa_n$ times the fundamental solutions of the Laplacian, while by \eqref{KtnDefn}, the parabolic weight is an approximation to a transformed heat kernel.
As dimension-limiting arguments will be crucial to our proofs below, we need to understand what happens when $n \to \iny$.
By \eqref{KtnDefn} and \ref{KntKtBound} in combination with the Dominated Convergence Theorem, the parabolic integrals in both of these lemmas converge to $K(h(x), t)$-weighted integrals.

For the remainder of this section, we introduce some definitions and draw some conclusions based on Lemmas \ref{PFTS} and \ref{PFT}.

\begin{defn}[Weighted $L^1$ spaces]
\label{sPDefn}
Given $\phi: \R^d \times \pr{0, T} \to \R$ and $h: \R^d \to \R^d$ as in Section \ref{S:prelim}, define $\mathcal{P}\pr{\cdot; \phi, h} : \pr{0, T} \to \R$ as 
\begin{align*}
\mathcal{P}\pr{t; \phi, h} 
&= \int_{\R^d} \abs{\phi(x,t)} \exp\pr{-\frac{\abs{h(x)}^2}{4t}} dx.
\end{align*}
We say that $\mathcal{P}\pr{\cdot; \phi, h} \in L^1\pr{\brac{0, t}, s^{- \frac d 2}ds}$ if 
$$\int_0^t s^{- \frac d 2} \int_{\R^d} \abs{\phi(x,s)} \exp\pr{-\frac{\abs{h(x)}^2}{4s}} dx \, ds < \iny.$$
\end{defn}

With this definition and Lemma \ref{PFT}, we reach the following result.

\begin{cor}[Weighed integrability on the elliptic side]
\label{functionClassCor}
Let $\phi: \R^d \times \pr{0, T} \to \R$, $h: \R^d \to \R^d$ and assume that $\mathcal{P}\pr{\cdot; \phi, h} \in L^1\pr{\brac{0, t}, s^{- \frac d 2}ds}$.
For each $\psi_n: B_T^n \su \R^{d \times n} \to \R$ defined by
$\psi_n(y) = \phi\pr{F_{d,n}(y)}$, $\psi_n$ is integrable with respect to $\kappa_n(y) \, dy$ in $B_t^n$ for every $n \ge 2$.
\end{cor}

\begin{proof}
An application of Lemma \ref{PFT} shows that 
\begin{align*}
\int_{B_t^n} \kappa_n(y) \abs{\psi_n(y)} dy
&= \int_{B_t^n} \kappa_n(y) \abs{\phi\pr{F_{d,n}(y)}} \abs{y}^{dn-2} \abs{y}^{2-dn}  dy \\
&= d \abs{S^{d n -1}} \int_{0}^t \int_{\R^d} \abs{ \phi\pr{x,s}} \pr{2ds}^{\frac{dn-2}2} K_n\pr{(h(x), s} dx ds \\
&\le d \abs{S^{d n -1}} \mathcal{C}_d \int_{0}^t  \int_{\R^d} \abs{\phi(x,s)} \pr{2ds}^{\frac{dn-2}2} K\pr{(h(x), s}  dx ds \\
&\le \frac{\pr{2dt}^{\frac{dn}2} \mathcal{C}_d \abs{S^{d n -1}}}{2t\pr{4 \pi}^{\frac d 2}} \int_{0}^t  s^{-\frac{d}2} \mathcal{P}\pr{s; \phi, h} ds,
\end{align*}
where we have applied \eqref{KntKtBound} and that $dn \ge 2d \ge d$.
\end{proof}

Now we introduce some function classes for parabolic and elliptic functions that will be used below.

\begin{defn}[Moderate $h$-growth at infinity]
\label{hGrowth}
Let $u : \R^d \times \pr{0, T} \to \R$ be a continuous function with locally integrable weak first order derivatives.
With $h : \R^d \to \R^d$ as in Section \ref{S:prelim} and $A = H^{-1} \pr{H^{-1}}^T$, define
\begin{equation}
\label{HDTDefs}
\begin{aligned}
\mathcal{H}\pr{t} 
= \mathcal{H}\pr{t; u,h} 
&= \int_{\R^d} \abs{u(x,t)}^2 \exp\pr{-\frac{\abs{h(x)}^2}{4t}} dx \\
\mathcal{D}\pr{t}
= \mathcal{D}\pr{t; u,h} &= \int_{\R^d} \innp{A(x) \gr u(x,t), \gr u(x,t)} \exp\pr{-\frac{\abs{h(x)}^2}{4t}} dx \\
\mathcal{T}\pr{t}
= \mathcal{T}\pr{t; u,h} &= \int_{\R^d} \abs{\frac{\del u}{\del t}(x,t)}^2 \exp\pr{- \frac{\abs{h(x)}^2}{4t}} dx.
\end{aligned}
\end{equation}
We say that such a function $u$ has \emph{moderate $h$-growth at infinity} if $\mathcal{H}$, $\mathcal{D}$, and $\mathcal{T}$ belong to $L^1\pr{\brac{0, T}, t^{- \frac d 2}dt}$.
\end{defn}

\begin{defin}[$\kappa$-weighted Sobolev space]
For $B_R \su \R^N$, we say that a function $v : B_R \to \R$ belongs to $L^{p}\pr{B_R, \kappa(y) \, dy}$, the space of \emph{$\kappa$-weighted $p$-integrable functions}, if
\begin{align*}
& \int_{B_R} \kappa(y) \abs{v(y)}^p dy < \iny. 
\end{align*}
Moreover, if both $v$ and $\gr v \in L^2\pr{B_R, \kappa(y) \, dy}$, then we say that $v$ belongs to the \emph{$\kappa$-weighted Sobolev space} and write $v \in W^{1,2}\pr{B_R, \kappa(y) \, dy}$.
\end{defin}

Now we may connect these two function classes.

\begin{lem}[Function class relationship]
\label{functionClassLem}
If $u: \R^d \times \pr{0, T} \to \R$ has moderate $h$-growth at infinity, then for $v_n: B_T^n \su \R^{d \times n} \to \R$ defined by $v_n(y) = u\pr{F_{d,n}(y)}$ and $\kappa_n$ defined by \eqref{kappaDef}, it holds that $v_n \in W^{1,2}\pr{B_T^n, \kappa_n(y) dy}$, whenever $n \ge 2$.
\end{lem}

\begin{proof}
An application of Lemma \ref{PFT} shows that 
\begin{align*}
\int_{B_T^n} \kappa_n(y) \abs{v_n(y)}^2 dy
&= \int_{B_T^n} \kappa_n(y) \abs{u\pr{F_{d,n}(y)}}^2 \abs{y}^{dn-2} \abs{y}^{2-dn}  dy \\
&= d \abs{S^{d n -1}} \int_{0}^T \int_{\R^d}  \abs{u(x,s)}^2 \pr{2ds}^{\frac{dn-2}2} K_n\pr{(h(x), s} dx ds \\
&\le \frac{\pr{2d}^{\frac{dn}2} \mathcal{C}_d \abs{S^{d n -1}}}{2\pr{4 \pi}^{\frac d 2}} \int_{0}^T s^{\frac{dn-d-2}2} \mathcal{H}\pr{s} ds,
\end{align*}
where we have applied \eqref{KntKtBound}.
Since $u$ has moderate $h$-growth at infinity, then
\begin{align*}
\int_{B_T^n} \kappa_n(y) \abs{v_n(y)}^2 dy
\le \frac{\pr{2dT}^{\frac{dn}2} \mathcal{C}_d \abs{S^{d n -1}}}{2T\pr{4 \pi}^{\frac d 2}} \norm{\mathcal{H}}_{L^1\pr{\brac{0,T}, t^{- \frac d 2} dt}}
< \iny.
\end{align*}
For the gradient term, we use Lemma \ref{PFT} in combination with Lemma \ref{ChainRuleLem} to get
\begin{align*}
\int_{B_T^n} \kappa_n(y) \abs{\gr v_n(y)}^2 dy
&= d n \abs{S^{d n -1}} \int_{0}^T \int_{\R^d} \innp{A\gr u, \gr u} \pr{2ds}^{\frac{dn-2}2} K_n\pr{(h(x), s} dx ds \\
&+ \abs{S^{d n -1}} \int_{0}^T \int_{\R^d} 2s \pr{\frac{\del u}{\del s}}^2\pr{2ds}^{\frac{dn-2}2} K_n\pr{(h(x), s} dx ds \\
&+ \abs{S^{d n -1}} \int_{0}^T \int_{\R^d} 2 \frac{\del u}{\del s} \innp{\pr{H^{-1}}^T\gr u, h} \pr{2ds}^{\frac{dn-2}2} K_n\pr{(h(x), s}  dx ds \\
&\le 2 d n \abs{S^{d n -1}} \int_{0}^T \int_{\R^d} \innp{A\gr u, \gr u} \pr{2ds}^{\frac{dn-2}2} K_n\pr{(h(x), s} dx ds \\
&+ 2\abs{S^{d n -1}} \int_{0}^T \int_{\R^d} s \pr{1 + \frac{\abs{h(x)}^2}{2dns}} \pr{\frac{\del u}{\del s}}^2\pr{2ds}^{\frac{dn-2}2} K_n\pr{(h(x), s} dx ds,
\end{align*}
where the last step uses Cauchy-Schwarz.
Since $\abs{h(x)}^2 \le 2 d n s$ on the support of $K_n(h(x))$, then
\begin{align*}
\int_{B_T^n} \kappa_n(y) \abs{\gr v_n(y)}^2 dy
&\le \frac{\pr{2d}^{\frac {dn}2} \mathcal{C}_d n \abs{S^{d n -1}}}{\pr{4\pi}^{\frac d 2}} \pr{\int_{0}^T s^{\frac{dn-d-2}2} \mathcal{D}\pr{s} ds 
+ \frac{2}{dn} \int_{0}^T s^{\frac{dn-d}2} \mathcal{T}\pr{s} ds } \\
&\le \frac{\pr{2dT}^{\frac {dn}2} \mathcal{C}_d n \abs{S^{d n -1}}}{T\pr{4\pi}^{\frac d 2}} \pr{  \norm{\mathcal{D}}_{L^1\pr{\brac{0,T}, t^{- \frac d 2} dt}}
+ \frac{2T}{dn} \norm{\mathcal{T}}_{L^1\pr{\brac{0,T}, t^{- \frac d 2} dt}}}
< \iny,
\end{align*}
where we have again used that $u$ has moderate $h$-growth at infinity.
\end{proof}

\section{Carleman Estimates}
\label{S:Carleman}

Carleman estimates have played a significant role in the development of the theory of unique continuation for both elliptic and parabolic equations.
The original idea (used in the elliptic setting) is attributed to Carleman \cite{Car39}, with subsequent work accomplished by Cordes \cite{Cor56}, Aronszajn \cite{Aro57} and Aronszajn et al \cite{AKS62}.
Since then, a wealth of results have been produced for both elliptic and parabolic operators; see for example \cite{JK85}, \cite{Sog90}, \cite{KT01}, \cite{E}, \cite{EV01}, \cite{EF03}, \cite{Ves03}, \cite{Fer03}, \cite{Ngu10}, and \cite{KT09}.
Additional applications of Carleman estimates include geometry, inverse problems, and control theory.

The following elliptic Carleman estimate is the $L^2 \to L^2$ case of Theorem 1 from \cite{ABG}.
The original theorem was used to establish unique continuation properties of functions that satisfy partial differential inequalities of the form $\abs{\LP v} \le\abs{V} \abs{v}$, for $v \in W^{2,q}_{loc}\pr{\Om}$, $V \in L^w_{loc}\pr{\Om}$, where $w > \frac{N}{2}$, and $\Om \subset \R^N$ is open and connected.

\begin{prop}[\cite{ABG}, Theorem 1]
\label{ECE}
For any $\tau \in \R$ and all $v \in W^{2,2}_0\pr{\R^N \setminus \set{0}}$, the following inequality holds
\begin{equation}
c\pr{\tau, N} \norm{\abs{y}^{-\tau} v}_{L^2\pr{\R^N}} \le \norm{\abs{y}^{-\tau +2} \LP v}_{L^2\pr{\R^N}},
\end{equation}
where
$$c\pr{\tau, N} = \inf_{\ell \in \Z_{\ge 0}} \abs{\pr{\frac{N}{2} + \ell + \tau - 2}\pr{\frac{N}{2} + \ell - \tau}}.$$
\end{prop}

\begin{rem}
In order for this theorem to be meaningful, we choose $\tau \in \R$ so that $\tau - \frac N 2 \not\in \Z_{\ge 0}$.
\end{rem}

\begin{rem}
Other versions of this theorem hold with more general norms.
More specifically, \cite[Theorem 1]{ABG} establishes that for any $1 \le q \le 2 \le p < \iny$ and $\mu := 2 - \frac N w = 2 - N \pr{ \frac 1 q - \frac 1 p} > 0$, 
$$\norm{\abs{y}^{-\tau} v}_{L^p\pr{\R^N}} \lesssim \norm{\abs{y}^{-\tau +\mu} \LP v}_{L^q\pr{\R^N}}.$$
However, since the condition that $\mu > 0$ is equivalent to $N < \frac{2 p q}{p-q}$, we must have that $p$ and $q$ are both very close to $2$ for large $N$.
In particular, when $N \to \iny$, $p, q \to 2$.
Since we use the high-dimensional limit of this elliptic Carleman estimate to establish its parabolic counterpart, this explains why we restrict to the $L^2 \to L^2$ Carleman estimate.
\end{rem}

The following parabolic Carleman estimate is a variable-coefficient $L^2 \to L^2$ version of Theorem 1 from \cite{E},
and it resembles \cite[Theorem 4]{EF03}.
For a much more general result, we refer the reader to \cite[Theorem 3]{KT09}.
The original theorem was used to prove strong unique continuation of solutions to the heat equation.  
As such, this version could be used to establish unique continuation results for solutions to variable-coefficient heat equations.

\begin{thm}[c.f. \cite{E}, Theorem 1]
\label{PCE}
For $H$ is as in \eqref{Jacobians} and $\eta = \det H$, assume that $\pr{H^{-1}}^T \gr \log \eta \in L^\iny\pr{\R^d}$ and $\di \brac{H^{-1} \pr{H^{-1}}^T\gr \log \eta} \in L^\iny\pr{\R^d}$.
Let $d \ge 1$ and take $\al \in \R$ so that $2\al - \frac{d}{2} - 3 \in \pr{0, \iny} \setminus \Z$.
Define $\eps := \dist\pr{2\al - \frac{d}{2} - 1, \Z_{\ge 0}}$ and for some $\de \in \pr{0,1}$, set 
$$T_0 = \eps \sqrt{1-\de} \norm{ 2 \di \brac{H^{-1} \pr{H^{-1}}^T\gr \log \eta} + \abs{\pr{H^{-1}}^T \gr \log \eta}^2}_{L^\iny\pr{\R^d}}^{-1}.$$
Then there is a constant $C$, depending only on $\eps$ and $\de$, such that for every $u \in C^\iny_0\pr{\R^{d} \times \pr{0, T_0} \setminus \set{\pr{0,0}}}$, it holds that
\begin{align*}
\int_0^{T_0} \int _{\R^d} \abs{u\pr{x,t}}^2 t^{-2\al} e^{-\frac{\abs{h(x)}^2}{4t}}  dx \, dt 
\le C\pr{\eps, \de} \int_0^{T_0} \int _{\R^d} \abs{\di \pr{A \gr u} + \del_t u}^2  t^{-2\al + 2} e^{-\frac{\abs{h(x)}^2}{4t}} dx \, dt,
\end{align*}
where $h$ is some invertible function and $A = H^{-1} \pr{H^{-1}}^T$.
\end{thm}

We show that Theorem \ref{PCE} follows from the elliptic result, Proposition \ref{ECE}, Lemma \ref{PFT}, and the results of Section \ref{S:prelim}.
More specifically, given $u$, we define a sequence of functions $\set{v_n}$ and we then apply Proposition \ref{ECE} to each one.
Applications of Lemma \ref{PFT} allow us to transform the integral inequalities for $v_n$ into integral inequalities for $u$.
By taking a limit of this sequence of inequalities and using \eqref{KtnDefn}, we arrive at our conclusion.

\begin{proof}
Let $u \in C^\iny_0\pr{\R^{d} \times \pr{0, T_0} \setminus \set{\pr{0,0}}}$.
For every $n \in \N$, let $v_{n} : \R^{d \times n} \to \R$ satisfy
$$v_{n}\pr{y} = u\pr{F_{d, n}\pr{y}},$$
where $F_{d,n}$ is as defined in \eqref{FdnDef}.
Note that each $v_n$ is compactly supported in $B_{T_0}^n = B_{\sqrt{2 d T_0}}$.
With $\ga(z) = \frac 1 {\eta(x)}$, define $\kappa_n : \R^{d \times n} \to \R$ to satisfy
$$\kappa_{n}\pr{y} = \frac 1 {\eta\pr{g(f_{d, n}(y))}}.$$
For $\al \in \R$ so that $2\al - \frac{d}{2} - 3 \in \pr{0, \iny} \setminus \Z$, set $\tau_n  = 2\al + \frac{d n - d - 2}{2}$.
Since $\tau_n - \frac{d n}{2} = 2\al - \frac{d}{2} - 1 \in \pr{2, \iny} \setminus \Z$, then $\tau_n - \frac{d n}{2} \not\in \Z_{\ge 0}$.

Since 
\begin{align*}
\LP \pr{\sqrt \kappa_n v_n}
&= \di \pr{\sqrt \kappa_n \gr v_n + \frac{\gr \kappa_n}{2 \sqrt \kappa_n} v_n}
= \sqrt \kappa_n \LP v_n  + \frac{\gr \kappa_n}{\sqrt \kappa_n} \cdot \gr v_n + \di\pr{\frac{\gr \kappa_n}{2 \sqrt \kappa_n}} v_n \\
&= \frac{1}{\sqrt \kappa_n} \di \pr{\kappa_n \gr v_n} - \di \pr{\kappa_n \gr\pr{\kappa_n^{-1/2}}} v_n,
\end{align*}
then for $\de > 0$ as given,
\begin{align*}
\abs{\LP \pr{\sqrt \kappa_n v_n}}^2
&\le C_\de \kappa_n \brac{ \frac 1 {\kappa_n} \di \pr{\kappa_n \gr v_n}}^2 
+  \frac {c_\de}{\kappa_n} \brac{\di \pr{\kappa_n \gr\pr{\kappa_n^{-1/2}}} }^2 \abs{\sqrt \kappa_n v_n}^2,
\end{align*}
where $C_\de = {1 + \de^{-1}}$ and $c_\de = 1 + \de$.
An application of Theorem \ref{ECE} with $v = \sqrt \kappa_n v_n$ and $\tau = \tau_n$ shows that
\begin{equation}
\label{EllCE}
\begin{aligned}
c\pr{\tau_n, d n}^2 \int_{B_{T_0}^n}\abs{\sqrt \kappa_n v_n}^2  \abs{y}^{-2\tau_n}  dy
&\le \int_{B_{T_0}^n} \abs{\LP \pr{\sqrt \kappa_n v_n}}^2 \abs{y}^{-2\tau_n +4} dy \\
&\le C_\de \int_{B_{T_0}^n}  \kappa_n \brac{ \frac 1 {\kappa_n} \di \pr{\kappa_n \gr v_n}}^2 \abs{y}^{-2\tau_n +4} dy \\
&+ c_\de \int_{B_{T_0}^n} \brac{\frac{\abs{y}^2}{\sqrt{\kappa_n}} \di \pr{\kappa_n \gr\pr{\kappa_n^{-1/2}}} }^2 \abs{\sqrt \kappa_n v_n}^2 \abs{y}^{-2\tau_n } dy.
\end{aligned}
\end{equation}
Since $u \in C^\iny_0\pr{\R^{d} \times \pr{0, T_0} \setminus \set{\pr{0,0}}}$, then $\abs{u}^2 \pr{2dt}^{\frac{d n}{2} -1 -\tau_n}$ satisfies the hypotheses of Lemma \ref{PFT} and it follows that
\begin{equation}
\label{uEst}
\begin{aligned}
\int_{B_{T_0}^n} \abs{\sqrt{\kappa_n(y)} v_{n}\pr{y}}^2  \abs{y}^{-2 \tau_n} dy 
& =\int_{B_{T_0}^n} \kappa_n(y) \abs{v_{n}\pr{y}}^2  \abs{y}^{-2 \tau_n} dy \\
&= \int_{B_{T_0}^n} \kappa_n(y) \abs{u\pr{g \circ f_{d,n}(y), \frac{\abs{y}^2}{2d}}}^2 \abs{y}^{d n - 2 -2 \tau_n} \abs{y}^{2 - d n} dy \\
&= d \abs{S^{d n-1}} \int_0^{T_0} \int_{\R^d} \abs{u\pr{x, t}}^2 \pr{2dt}^{\frac{d n}{2} -1 -\tau_n} K_n\pr{h(x),t} dx \, dt.
\end{aligned}
\end{equation}
Because $\set{\di \pr{A \gr u} +  \frac{\del u}{\del t}+ \frac 2 {dn}\brac{ \innp{A\gr \frac{\del u}{\del t}, H^T h} + \frac 1 2 \frac{\del u}{\del t} \tr\pr{H \innp{\gr_z G(h), h}} + t \frac{\del^2 u}{\del t^2} }}^2 \pr{2dt}^{\frac{d n}{2} + 1- \tau_n}$ also satisfies the hypotheses of Lemma \ref{PFT}, then another application of Lemma \ref{PFT} shows that
\begin{equation}
\label{uLPEst}
\begin{aligned}
&\frac 1 {\abs{S^{d n-1}}}\int_{B_{T_0}^n}  \kappa_n \brac{ \frac 1 {\kappa_n} \di \pr{\kappa_n \gr v_n}}^2 \abs{y}^{-2\tau_n +4} dy \\
=& \frac 1 {\abs{S^{d n-1}}}\int_{B_{T_0}^n} \kappa_n\pr{y} \brac{ \frac 1 {\kappa_n} \di \pr{\kappa_n \gr v_n}}^2 \abs{y}^{d n + 2-2 \tau_n} \abs{y}^{2 - d n} dy \\
\le& 2 d n ^2  \int_0^{T_0} \int_{\R^d} \abs{\di \pr{A \gr u} +  \frac{\del u}{\del t}}^2  \pr{2dt}^{\frac{d n}{2} + 1- \tau_n} K_n\pr{h(x),t} dx \, dt \\
+& \frac{8}{d} \int_0^{T_0} \int_{\R^d} \abs{ \innp{A\gr \frac{\del u}{\del t}, H^T h} + \frac 1 2 \frac{\del u}{\del t} \tr\pr{H \innp{\gr_z G(h), h}} + t \frac{\del^2 u}{\del t^2} }^2  \pr{2dt}^{\frac{d n}{2} + 1- \tau_n} K_n\pr{h(x),t} dx \, dt,
\end{aligned}
\end{equation}
where we have used \eqref{chain} from Lemma \ref{ChainRuleLem} and the triangle inequality to reach the last line.

\noindent Finally, because $\brac{{2 d n t} \frac{\di \pr{A \gr \sqrt \eta} }{\sqrt{\eta}}}^2 \abs{u}^2 \pr{2dt}^{\frac{d n}{2} -1 -\tau_n}$ satisfies the hypotheses of Lemma \ref{PFT}, we see that
\begin{equation}
\label{extraTerm}
\begin{aligned}
&\int_{B_{T_0}^n} \brac{\frac{\abs{y}^2}{\sqrt{\kappa_n}} \di \pr{\kappa_n \gr\pr{\kappa_n^{-1/2}}} }^2 \abs{\sqrt \kappa_n v_n}^2 \abs{y}^{-2\tau_n } dy \\
=& d \abs{S^{d n-1}} \int_0^{T_0} \int_{\R^d} \brac{{2 d n t} \frac{\di \pr{A \gr \sqrt \eta} }{\sqrt{\eta}}}^2 \abs{u\pr{x, t}}^2 \pr{2dt}^{\frac{d n}{2} -1 -\tau_n} K_n\pr{h(x),t} dx \, dt \\
\le& d n^2 \abs{S^{d n-1}}  \norm{ \frac{\di \pr{A \gr \sqrt \eta} }{\sqrt{\eta}}}_{L^\iny\pr{\R^d}}^2 \int_0^{T_0} \int_{\R^d} \abs{u\pr{x, t}}^2 \pr{2dt}^{\frac{d n}{2} +1 -\tau_n} K_n\pr{h(x),t} dx \, dt,
\end{aligned}
\end{equation}
since \eqref{chain} with $v_n = \kappa_n^{-1/2}$ gives
\begin{align*}
\frac{\abs{y}^2}{\sqrt{\kappa_n(y)}} \di_y \pr{\kappa_n(y) \gr_y \pr{\kappa_n(y)^{-1/2}}}
&= {2 d n t} \frac{\di_x \pr{A(x) \gr_x \sqrt \eta(x)} }{\sqrt{\eta(x)}}. 
\end{align*}
A computation shows that 
\begin{align*}
\frac{\di \pr{A \gr \sqrt \eta} }{\sqrt{\eta}}    
&= \frac{\di \brac{H^{-1} \pr{H^{-1}}^T \sqrt \eta \, \gr \log \eta} }{2 \sqrt{\eta}}
= \frac 1 2 \di \brac{H^{-1} \pr{H^{-1}}^T\gr \log \eta}
+ \frac 1 4 \abs{\pr{H^{-1}}^T \gr \log \eta}^2,
\end{align*}
which belongs to $L^\iny\pr{\R^d}$ by assumption.
Substituting \eqref{uEst}, \eqref{uLPEst}, and \eqref{extraTerm} into \eqref{EllCE} and simplifying shows that
\begin{equation}
\label{EllCE2}
\begin{aligned}
& c\pr{\tau_n, dn}^2 \int_0^{T_0} \int_{\R^d} \abs{u\pr{x, t}}^2 t^{\frac{d n}{2} -1 -\tau_n} K_n\pr{h(x),t} dx \, dt \\
\le& 2 C_\de \pr{2 d n}^{2} \int_0^{T_0} \int_{\R^d} \abs{\di \pr{A \gr u} +  \frac{\del u}{\del t}}^2  t^{\frac{d n}{2} + 1- \tau_n} K_n\pr{h(x),t} dx \, dt \\
+& 32 C_\de \int_0^{T_0} \int_{\R^d} \brac{ \innp{ A\gr \tfrac{\del u}{\del t}, H^T h}  + \tfrac 1 2\tfrac{\del u}{\del t} \text{tr}\pr{H \innp{\gr_z G(h), h}}+ t \tfrac{\del^2 u}{\del t^2} }^2 t^{\frac{d n}{2} + 1- \tau_n} K_n\pr{h(x),t} dx \, dt \\
+& \pr{d n \eps}^2 \pr{1 - \de^2} \int_0^{T_0} \int_{\R^d} \abs{u\pr{x, t}}^2 t^{\frac{d n}{2} -1 -\tau_n} K_n\pr{h(x),t} dx \, dt,
\end{aligned}
\end{equation}
where we have used the definition of $c_\de$ and that $\disp 4 T_0 \norm{ \frac{\di \pr{A \gr \sqrt \eta} }{\sqrt{\eta}}}_{L^\iny\pr{\R^d}} = \eps \sqrt{1 - \de}$ to reach the last line.
Observe that
\begin{align*}
c\pr{\tau_n, dn} 
&= \inf_{\ell \in \Z_{\ge 0}} \abs{\brac{\frac{d n}{2} + \ell + \pr{2\al + \frac{d n - d - 2}{2}} - 2} \brac{\frac{d n}{2} + \ell - \pr{2\al + \frac{d n - d - 2}{2}}}} \\
&= \inf_{\ell \in \Z_{\ge 0}} \abs{\brac{ d  n -2 + \ell + \pr{2\al - \frac{ d }{2} - 1}}\brac{ \ell - \pr{2\al - \frac{ d }{2} - 1}}}
\ge d n \eps,
\end{align*}
since we assumed that $2 \al - \frac d 2 - 3 > 0$.
Returning to \eqref{EllCE2}, the last term on the right may be absorbed into the left to get
\begin{equation*}
\begin{aligned}
&\int_0^{T_0} \int_{\R^d} \abs{u\pr{x, t}}^2 t^{ \frac{d}{2} - 2\al} K_n\pr{h(x),t} dx \, dt \\
\le& 2 C_\de \pr{\frac{2}{\eps \de}}^{2} \int_0^{T_0} \int_{\R^d} \abs{\di \pr{A \gr u} +  \frac{\del u}{\del t}}^2 t^{\frac{d}{2} - 2\al + 2} K_n\pr{h(x),t} dx \, dt \\
+& 2 C_\de \pr{\frac{4}{dn \eps \de}}^2 \int_0^{T_0} \int_{\R^d} \brac{ \innp{ A\gr \tfrac{\del u}{\del t}, H^T h}  + \tfrac 1 2\tfrac{\del u}{\del t} \text{tr}\pr{H \innp{\gr_z G(h), h}}+ t \tfrac{\del^2 u}{\del t^2} }^2  t^{\frac d 2 - 2\al + 2} K_n\pr{h(x),t} dx \, dt,
\end{aligned}
\end{equation*}
with $2\al = \tau_n - \frac{d n - d - 2}{2}$.
We now take the limit as $n \to \iny$. 
An application of Lemma \ref{L:uniform} shows that
\begin{equation*}
\begin{aligned}
&\int_0^{T_0} \int_{\R^d} \abs{u\pr{x, t}}^2 t^{  - 2\al} \exp\pr{- \frac{\abs{h(x)}^2}{4t}}dx \, dt \\
\le& \frac{8 \pr{1 + \de}}{\eps^2 \de^2} \int_0^{T_0} \int_{\R^d} \abs{\di \pr{A \gr u} +  \frac{\del u}{\del t}}^2 t^{- 2\al + 2} \exp\pr{- \frac{\abs{h(x)}^2}{4t}} dx \, dt,
\end{aligned}
\end{equation*}
as required.
\end{proof}

\section{Almgren Monotonicity Formula}
\label{S:Almgren}

In this section, we show that the Almgren-type frequency function associated with the parabolic operator $\text{div}(A\nabla)+\partial_t$ is monotonically non-decreasing. 
When $A = I$, a corresponding result goes back to Poon, \cite{P2}. 
In that paper,  the monotonicity was key in the proof of unique continuation results for caloric functions. 
A version of this result was later proved in the context of the parabolic, constant-coefficient Signorini problem in \cite{DGPT}, where the authors used the monotonicity to establish the optimal regularity of solutions and study the free boundary. 

To establish our parabolic result through the high-dimensional limiting technique, we require a similar result for solutions to non-homogeneous variable-coefficient elliptic equations.
Many similar results for the homogenous setting have been previously established.
For example, Almgren-type monotonicity formulas for variable-coefficient operators have also been extensively used to study a wide variety of free boundary problems, as in \cite{Gui}, \cite{GSVG}, \cite{GPSVG}, \cite{GPSVG2}, \cite{JPSVG}, \cite{BBG}, \cite{DT}. 
The following result is crucial to our upcoming parabolic proof, but it may also be of independent interest.

\begin{prop}
\label{variableAlmgren}
For some $R > 0$, let $B_R \su \R^N$.
Assume that for $\kappa: B_R \to \R_{+}$ it holds that  $\gr \log \kappa \cdot y \in L^{\iny}\pr{B_R}$.
Let $v \in W^{1,2}\pr{B_R, \kappa dy}$ be a weak solution to $\di \pr{\kappa \gr v } = \kappa \ell$ in $B_R$, where $\ell$ is integrable with respect to both $\kappa \, v$ and $\kappa \gr v \cdot y$ on each $B_r$, for $r \in \pr{0, R}$. 
For every $r \in \pr{0, R}$, assuming that each $v \rvert_{\del B_r}$ is non-trivial, define
\begin{align*}
H(r) &= H(r; v, \kappa) = \int_{\del B_r} \kappa\pr{y} \abs{v(y)}^2 d\si(y) \\
D(r) &= D(r; v, \kappa) = \int_{B_r} \kappa(y) \abs{\gr v(y)}^2 dy \\
L(r) &= L(r; v, \kappa) = \frac{r D(r; v, \kappa)}{H(r; v, \kappa)}.
\end{align*}
Set $\widetilde{L}(r) = r^{2\Upsilon} L(r)$, where $\Upsilon \ge \norm{\gr \log \kappa\cdot y}_{L^\iny(B_R)}$.
Then for all $r \in \pr{0, R}$, it holds that
\begin{align*}
\widetilde{L}'(r)
&\ge 2 r^{2\Upsilon }  \brac{\frac{\pr{ \int_{B_r} \kappa \, \ell v \, dy} \pr{ \int_{\del B_r} \kappa \, v \, \gr v \cdot y \, d\si(y)}}{ \pr{ \int_{\del B_r} \kappa \abs{v}^2 \, d\si(y)}^2}
 - \frac{\pr{\int_{B_r} \kappa \, \ell \, \gr v \cdot y \, dy}}{\pr{\int_{\del B_r} \kappa \abs{v}^2\, d\si(y)}}}.
\end{align*}
\end{prop}

\begin{rem}
Notice that if $v$ is a solution to the homogeneous equation $\di \pr{\kappa \gr v } = 0$ in $B_R$, i.e. $\ell = 0$, then $\widetilde{L}(r)$ is non-decreasing in $r$.
Moreover, if $\kappa = 1$, we recover the non-homogenous elliptic result from \cite[Corollary 1]{Dav18}, which is the non-homogeneous version of Poon's result, \cite{P2}.
In particular, we recover the expected monotonicity formula for solutions to elliptic equations.
\end{rem}

\begin{proof}
Observe first that since
\begin{align*}
H(r) &= \int_{\del B_r} \kappa\pr{y} \abs{v(y)}^2 d\sigma(y)
=  r^{N-1} \int_{\del B_1} \kappa\pr{r\zeta} \abs{v(r\zeta)}^2 d\sigma(\zeta),
\end{align*}
then
\begin{align*}
H'(r) 
&= \frac{N-1}{r} H(r) 
+ 2 r^{N-1} \int_{\del B_1} \kappa\pr{r\zeta} v(r\zeta) \gr v(r\zeta) \cdot \zeta d\sigma(\zeta)
+ r^{N-1} \int_{\del B_1} \gr \kappa\pr{r \zeta} \cdot\zeta \abs{v(r\zeta)}^2 d\sigma(\zeta) \\
&= \frac{N-1}{r} H(r) 
+ 2 \int_{\del B_r} \kappa \, v \, \gr v \cdot \hat{n}
+  \int_{\del B_r} \gr \kappa \cdot \hat{n} \, \abs{v}^2 ,
\end{align*}
where $\hat{n}$ indicates the outer unit normal. Moreover, integration by parts and the equation for $v$ shows that
\begin{align}
\label{IExp}
D(r) = \int_{B_r} \kappa \abs{\gr v}^2
= \int_{B_r} \frac 1 2 \di \brac{\kappa \gr \pr{v^2}} - \di\pr{\kappa \gr v} v 
= \int_{\del B_r} \kappa v \gr v \cdot \hat{n} - \int_{B_r} \kappa\ell v.
\end{align}
Now differentiating and integrating by parts shows that
\begin{align*}
D'(r) 
&= \int_{\del B_r} \kappa \abs{\gr v}^2
= \frac 1 r \int_{\del B_r} \kappa \abs{\gr v}^2 y \cdot \hat n \,
= \frac 1 r \int_{B_r} \di\pr{ \kappa \abs{\gr v}^2 y}  \\
&= \frac N r \int_{B_r} \kappa \abs{\gr v}^2
+ \frac 2 r \int_{B_r} \kappa \gr v \cdot D^2 v \, y 
+ \frac 1 r \int_{B_r} \gr \kappa \cdot y \abs{\gr v}^2 \\
&= \frac N r \int_{B_r} \kappa \abs{\gr v}^2
- \frac 2 r \int_{B_r} \kappa \abs{\gr v}^2
- \frac 2 r \int_{B_r} \di\pr{\kappa \gr v } \gr v \cdot y 
+ \frac 2 {r^2} \int_{\del B_r} \kappa \pr{\gr v \cdot y }^2
+ \frac 1 r \int_{B_r} \gr \kappa \cdot y \abs{\gr v}^2 \\
&= \frac {N-2} r D(r)
+ 2 \int_{\del B_r} \kappa \pr{\gr v \cdot \hat{n} }^2
+ \int_{B_r} \gr \kappa \cdot \frac{y}{r} \abs{\gr v}^2
- 2 \int_{B_r} \kappa \, \ell \, \gr v \cdot \frac{y}{r}.
\end{align*}
Therefore, by putting it all together, we see that
\begin{align*}
H(r)^2 L'(r) 
&= D(r)H(r) + r D'(r) H(r) - r D(r) H'(r) \\
&= D(r)H(r)  \\
&+ \brac{\pr{N-2} D(r)+ 2r \int_{\del B_r} \kappa \pr{\gr v \cdot \hat{n} }^2 +r \int_{B_r} \gr \kappa \cdot \frac{y}{r} \abs{\gr v}^2 - 2r \int_{B_r} \kappa \, \ell \, \gr v \cdot \frac{y}{r}}H(r) \\
&- D(r) \brac{\pr{N-1} H(r)  + 2r \int_{\del B_r} \kappa v \gr v \cdot \hat{n} + r \int_{\del B_r} \gr \kappa \cdot \hat{n} \abs{v}^2 } \\
&= r \pr{2 \int_{\del B_r} \kappa \pr{\gr v \cdot \hat{n} }^2 + \int_{B_r} \gr \kappa \cdot \frac{y}{r} \abs{\gr v}^2 - 2 \int_{B_r} \kappa \, \ell \, \gr v \cdot \frac{y}{r}}\pr{\int_{\del B_r} \kappa \abs{v}^2}  \\
&- 2r \pr{\int_{\del B_r} \kappa v \gr v \cdot \hat{n} - \int_{B_r}\kappa \, \ell v} \pr{ \int_{\del B_r} \kappa v \gr v \cdot \hat{n}}
- r \pr{\int_{B_r} \kappa \abs{\gr v}^2} \pr{ \int_{\del B_r} \gr \kappa \cdot \hat{n} \abs{v}^2 } \\
&= -\frac 2 r \brac{ \pr{\int_{\del B_r} \kappa \, v \gr v \cdot y }^2 - \pr{\int_{\del B_r} \kappa \abs{v}^2} \pr{\int_{\del B_r} \kappa \pr{\gr v \cdot y }^2 }} \\
&- \brac{\pr{\int_{B_r} \kappa \abs{\gr v}^2} \pr{ \int_{\del B_r} \frac{\gr \kappa \cdot y}{\kappa} \kappa \abs{v}^2 } - \pr{\int_{B_r} \frac{\gr \kappa \cdot y}{\kappa} \kappa \abs{\gr v}^2}\pr{\int_{\del B_r} \kappa \abs{v}^2} }  \\
&+ 2 \brac{\pr{ \int_{B_r} \kappa \, \ell v} \pr{ \int_{\del B_r} \kappa \, v \, \gr v \cdot y}
 - \pr{\int_{B_r} \kappa \, \ell \, \gr v \cdot y}\pr{\int_{\del B_r} \kappa \abs{v}^2}} .
\end{align*}
An application of Cauchy-Schwartz shows that
\begin{align*}
\pr{\int_{\del B_r} \kappa v \gr v \cdot y }^2
&\le \pr{\int_{\del B_r} \kappa \abs{v}^2}\pr{\int_{\del B_r} \kappa \pr{\gr v \cdot y}^2}
\end{align*}
while with $\Upsilon \ge \norm{\gr \log \kappa \cdot y}_{L^\iny(B_R)}$, we get that
\begin{align*}
&\abs{\pr{\int_{B_r} \kappa \abs{\gr v}^2} \pr{ \int_{\del B_r} \frac{\gr \kappa \cdot y}{\kappa} \kappa \abs{v}^2 } - \pr{\int_{B_r} \frac{\gr \kappa \cdot y}{\kappa} \kappa \abs{\gr v}^2}\pr{\int_{\del B_r} \kappa \abs{v}^2} }
\le 2 \Upsilon D\pr{r} H\pr{r}.
\end{align*}
It follows that
\begin{align*}
L'(r) 
&\ge - \frac {2\Upsilon} r \frac{r D\pr{r}}{ H\pr{r}}
+ 2 \brac{\frac{\pr{ \int_{B_r} \kappa \, \ell v} \pr{ \int_{\del B_r} \kappa v \gr v \cdot y}}{ \pr{ \int_{\del B_r} \kappa \abs{v}^2 }^2}
 - \frac{\pr{\int_{B_r} \kappa \, \ell \, \gr v \cdot y}}{\pr{\int_{\del B_r} \kappa \abs{v}^2}}},
\end{align*}
or rather,
\begin{align*}
\pr{r^{2\Upsilon} L(r)}'
= r^{2\Upsilon }  L'(r) + \frac{2\Upsilon} r r^{2\Upsilon}  L(r)
&\ge 2r^{2\Upsilon}  \brac{\frac{\pr{ \int_{B_r} \kappa \, \ell v} \pr{ \int_{\del B_r} \kappa v \gr v \cdot y}}{ \pr{ \int_{\del B_r} \kappa \abs{v}^2 }^2}
 - \frac{\pr{\int_{B_r} \kappa \, \ell \, \gr v \cdot y}}{\pr{\int_{\del B_r} \kappa \abs{v}^2}}}.
\end{align*}
\end{proof}

Now we use Lemma \ref{PFTS} in combination with Proposition \ref{variableAlmgren} to establish its parabolic counterpart.
Before stating the result, we discuss the kinds of solutions that we work with.

Let $u: \R^d \times \pr{0, T} \to \R$ have moderate $h$-growth at infinity, as described in Definition \ref{hGrowth}.
For every $t \in \pr{0, T}$, assume first that $u$ is sufficiently regular to define the functionals
\begin{equation}
\label{IIntegrals}
\begin{aligned}
\mathcal{I}(t)
&= \mathcal{I}(t; u,h) 
= \int_{\R^d} \abs{u\pr{x,t}} \abs{ \innp{A\gr u, {H^T h}} + 2 t \, \del_t u} \exp\pr{-\frac{\abs{h(x)}^2}{4t}} dx \\
\mathcal{J}(t)
&= \mathcal{J}(t; u,h) 
= \int_{\R^d} \abs{J(x,t)} \abs{u\pr{x,t}} \exp\pr{-\frac{\abs{h(x)}^2}{4t}} dx \\
\mathcal{K}(t)
&= \mathcal{K}(t; u,h) 
= \int_{\R^d} \abs{J(x,t)} \abs{ \innp{A\gr u, {H^T h}} + 2 t \, \del_t u} \exp\pr{-\frac{\abs{h(x)}^2}{4t}} dx,
\end{aligned}
\end{equation} 
where
\begin{equation}
\label{JDefn}
J(x,t) = J(x, t; u,h) = \frac{1}{d} \brac{ 2 \innp{A\gr \frac{\del u}{\del t}, H^T h}  
+ \frac{\del u}{\del t} \tr\pr{H \innp{\gr_z G(h), h}}
+ 2t \frac{\del^2 u}{\del t^2} }
\end{equation}
and all derivatives are interpreted in the weak sense.
Then we say that such a function $u$ belongs to the function class $\mathfrak{A}\pr{\R^d \times \pr{0, T}, h}$ if $u$ has moderate $h$-growth at infinity (so is consequently continuous), and for every $t_0 \in \pr{0, T}$, there exists $\eps \in \pr{0, t}$ so that
\begin{equation}
\label{AlmgrenClass1}
\mathcal{I} \in L^\iny\brac{t_0 - \eps, t_0}
\end{equation}
and there exists $p > 1$ so that
\begin{equation}
\label{AlmgrenClass2}   
\mathcal{J} \in L^{p}\pr{\brac{0, t_0}, t^{- \frac d 2} dt},
\quad 
\mathcal{K} \in L^{p}\pr{\brac{0, t_0}, t^{- \frac d 2} dt}.
\end{equation}
This is the class of functions that we consider in our result.
We remark that this may not be the weakest set of conditions under which our proof holds, but it is far less restrictive than assuming that our solutions are smooth and compactly supported, for example.

\begin{thm}
\label{T:parabolicALM}
Assume that $\tr\pr{H \innp{\gr_z G(h),h}} \in L^\iny(\R^d)$, where $h$, $H$, and $G$ are described by \eqref{xDef}, \eqref{zDef}, and \eqref{Jacobians}.
Define $A = H^{-1}\pr{H^{-1}}^T: \R^d \to \R^{d \times d}$ and let $u \in \mathfrak{A}\pr{\R^d \times \pr{0, T}, h}$ be a non-trivial solution to $\di \pr{A \gr u} + \del_t u = 0$ in $\R^d \times \pr{0, T}$.
For every $t \in \pr{0, T}$, define 
\begin{align*}
\mathcal{H}\pr{t}
= \mathcal{H}\pr{t; u,h} 
&= \int_{\R^d} \abs{u(x,t)}^2 e^{-\frac{\abs{h(x)}^2}{4t}} dx \\
\mathcal{D}\pr{t}
= \mathcal{D}\pr{t; u,h} 
&= \int_{\R^d} \innp{A(x) \gr u(x,t), \gr u(x,t)} e^{-\frac{\abs{h(x)}^2}{4t}} dx \\
\mathcal{L}\pr{t}
= \mathcal{L}\pr{t; u,h} 
&= \frac{t \mathcal{D}\pr{t; u, h}}{\mathcal{H}\pr{t; u, h}}.
\end{align*}
Set $\widetilde{\mathcal{L}}\pr{t} = t^{\Upsilon} \mathcal{L}\pr{t}$, where $\Upsilon \ge \norm{\tr\pr{H \innp{\gr_z G(h), h}}}_{L^\iny(
\R^d)}$.
Then $\widetilde{\mathcal{L}}\pr{t}$ is monotonically non-decreasing in $t$.
\end{thm}

\begin{proof}

We first check that with $\kappa_n$ and $v_n$ defined through the transformation maps $F_{d,n}$, the hypothesis of Lemma \ref{variableAlmgren} are satisfied.
Recall that $B_T^n \su \R^{d \times n}$ is given by \eqref{BtnDefn}.
Let $\kappa_n$ be as in Lemma \ref{ChainRuleLem1}, which shows that
$$\norm{\gr \log \kappa_n \cdot y}_{L^\iny(B_T^n)}
\le \norm{\tr\pr{H \innp{\gr_z G(h),h}}}_{L^\iny({\R^d})} < \iny.$$ 
In particular, $\gr \log \kappa_n \cdot y \in L^\iny(B_T^n)$ for each $n$.

Let $u \in \mathfrak{A}\pr{\R^{d} \times \pr{0,T}, h}$ be a non-trivial solution to $\displaystyle \di \pr{A \gr u} + \del_t u = 0$ in $\R^d \times \pr{0,T}$.
For every $n \in \N_{\ge 2}$, let $v_n : B_{T}^n \to \R$ satisfy
$$v_n\pr{y} = u\pr{F_{d,n}\pr{y}}.$$
Since $u$ has moderate $h$-growth at infinity, then Lemma \ref{functionClassLem} shows that each $v_n$ belongs to $W^{1,2}\pr{B_T^n, \kappa_n(y) \, dy}$.

An application of Corollary \ref{ChainRuleCor} shows that 
\begin{align*}
\frac 1 {\kappa_n(y)} \di \pr{ \kappa_n(y)\gr v_n}
&= J\pr{x,t},
\end{align*}
where $J$ is defined in \eqref{JDefn} and does not depend on $n$. 
For every $n$, define $\ell_{n} : B_{T}^n \to \R$ so that 
$$\ell_n\pr{y} = J\pr{F_{d,n}\pr{y}}$$
and then
$$\di \pr{\kappa_n \gr v_n} = \kappa_n \ell_n.$$

Since $u \in \mathfrak{A}\pr{\R^{d} \times \pr{0,T}, h}$, then \eqref{AlmgrenClass2} implies that for every $t \in \pr{0, T}$, $\mathcal{J}$ belongs to $L^{1}\pr{\brac{0, t}, s^{- \frac d 2} ds}$.
Corollary \ref{functionClassCor} then shows that $\ell_{n} v_n$ is integrable with respect to $\kappa_n \, dy$ in each $B_t^n$.
Similarly, because $\mathcal{K}$ belongs to $L^{1}\pr{\brac{0, t}, s^{- \frac d 2} ds}$, then Lemma \ref{ChainRuleLem} and Corollary \ref{functionClassCor} show that $\ell_{n} \gr v_n \cdot y$ is integrable with respect to $\kappa_n \, dy$ in each $B_t^n$.
Rephrased, this means that $\ell_n$ is integrable with respect to both $\kappa_n \, v_n$ and $\kappa_n \, \gr v_n \cdot y$ on each $B_t^n$.

By backward uniqueness of heat equations, (as in \cite{Ngu10}, for example), $u\pr{\cdot, t}$ is a non-trivial function of $x$ for each $t \in \pr{0, T}$.

Since $v_n \rvert_{\del B_r} = u \pr{\cdot, \frac{r^2}{2d}}$ is non-trivial for each $r \in \pr{0, \sqrt{2dT}}$, then all of the assumptions from Proposition \ref{variableAlmgren} hold. 
Therefore, we may apply Proposition \ref{variableAlmgren} to each $v_n$ on any ball of radius $\sqrt{2 d t}$ for $t < T$.

First we compute the frequency function associated to $v_n$ on the ball of radius $\sqrt{2 d t}$.
By Lemma \ref{PFTS},
\begin{align}
\label{HExpression}
H\pr{\sqrt{2 d t}; v_n, \kappa_n} 
&= \int_{S_t^n} \kappa_n(y) \abs{v_n\pr{y}}^2 d\si(y) 
= \pr{2 d t}^{\frac{d n -1}{2}} \abs{S^{d  n -1}} \int_{\R^d} \abs{u\pr{x,t}}^2 K_{t,n}\pr{h\pr{x}} d{x}.
\end{align}
Lemmas \ref{PFTS} and \ref{ChainRuleLem} imply that
\begin{equation}
\label{D1Term}
\begin{aligned}
\frac{\pr{2 d t}^{ -\frac{d n -1}{2}}}{\abs{S^{d n -1}}}\int_{S_t^n} \kappa_n(y) v_n\pr{y} \pr{y \cdot \gr v_n\pr{y}} d\si(y) 
=& \int_{\R^d} \brac{\innp{A\gr u, {H^T h}} + 2 t \del_t u} u\pr{x,t} K_{t,n}\pr{h\pr{x}} d{x},
\end{aligned}
\end{equation}
while lemma \ref{PFT} implies that
\begin{equation}
\label{D2Term}
\begin{aligned}
\frac{1}{\abs{S^{d n -1}} }\int_{B_t^n} \kappa_n(y) \ell_n\pr{y} v_n\pr{y} \abs{y}^{dn -2} \abs{y}^{2- d n } d{y}
=& d \int_0^t \int_{\R^d} J(x,s) u\pr{x,s} \pr{2 d s }^{\frac{d n -2}{2}} K_{n}\pr{h\pr{x}, s} dx \, ds.
\end{aligned}
\end{equation}
Using the expression \eqref{IExp} along with \eqref{D1Term} and \eqref{D2Term}, we see that
\begin{equation*}
\begin{aligned}
\frac{1}{\abs{S^{d n -1}}} D\pr{\sqrt{2 d t}; v_n, \kappa_n}
&= \pr{2 d t}^{ \frac{d n -2}{2}}\int_{\R^d} \brac{\innp{A\gr u, {H^T h}} + 2 t \del_t u} u\pr{x,t} K_{t,n}\pr{h\pr{x}} d{x} \\
&- d \int_0^t \int_{\R^d} J(x,s) u\pr{x,s} \pr{2 d s }^{\frac{d n -2}{2}} K_{n}\pr{h\pr{x}, s} dx \, ds.
\end{aligned}
\end{equation*}
We remark that since $u \in \mathfrak{A}\pr{\R^d \times \pr{0, T}}$, then Lemmas \ref{PFTS} and \ref{PFT} guarantee that the integrals in \eqref{HExpression}, \eqref{D1Term}, and \eqref{D2Term} are all well-defined.
Therefore,
\begin{align*}
L_n(t)
&:=\frac 1 2 L\pr{\sqrt{2 d t}; v_n, \kappa_n} 
= \frac{\sqrt{2 d t} D\pr{\sqrt{2 d t}; v_n, \kappa_n}}{2 H\pr{\sqrt{2 d t}; v_n, \kappa_n} }  \\
&= \frac{ \int_{\R^d} u\pr{x,t} \brac{\innp{A\gr u, {H^T h}} + 2 t \del_t u} K_{t,n}\pr{h\pr{x}} d{x} 
- d\int_0^t \int_{\R^d} J(x,s) u\pr{x,s} \pr{\frac s t }^{\frac{d  n -2}{2}} K_{n}\pr{h\pr{x}, s} dx \, ds}{ 2\int_{\R^d} \abs{u\pr{x,t}}^2  K_{t,n}\pr{h\pr{x}} d{x}}  \\
&= \frac{ \int_{\R^d} u\pr{x,t} \brac{\innp{A\gr u, {H^T h}} + 2 t \del_t u} K_{t,n}(h(x)) d{x} }{ 2\int_{\R^d} \abs{u\pr{x,t}}^2  K_{t,n}(h(x)) d{x}}  \\
&- \frac{ d\int_0^t \int_{\R^d} J(x,s) u\pr{x,s} \pr{\frac s t - \frac{\abs{h(x)}^2}{2 d n t}}^{\frac{d n - d - 2}{2}}\chi_{B_{ns}}(h(x)) dx \, ds}{2\int_{\R^d} \abs{u\pr{x,t}}^2  \pr{1 - \frac{\abs{h(x)}^2}{2 d n t}}^{\frac{d n - d - 2}{2}}\chi_{B_{nt}}(h(x)) d{x}}. 
\end{align*}
Thus,
\begin{align*}
\lim_{n \to \iny} L_n(t) 
&= \frac{ \int_{\R^d} u\pr{x,t} \brac{\innp{A\gr u, {H^T h}} + 2 t \del_t u} K_{t}\pr{h\pr{x}} d{x} }{ 2\int_{\R^d} \abs{u\pr{x,t}}^2  K_{t}\pr{h\pr{x}} d{x}} ,
\end{align*}
where we use $\displaystyle \lim_{n \to \iny} K_{t,n}(x)= K_t(x)$ from \eqref{KtnDefn}, the bound \eqref{KntKtBound}, and that 
$$\lim_{n \to \iny} \abs{\pr{\frac s t - \frac{\abs{h(x)}^2}{2 d n t}}^{\frac{d n - d - 2}{2}}\chi_{B_{ns}}(h(x))} \le \lim_{n \to \iny }\pr{\frac s t}^{\frac{d n - d - 2}{2}} = 0$$
for every $s \in \pr{0, t}$ along with the Dominated Convergence Theorem.

Since $\del_t u = - \di \pr{A \gr u}$ and $\gr \pr{K_t \circ h} = - \frac 1 {2t} H^T h \, K_t\pr{h}$, then
\begin{align*}
& \frac 1 2 \int_{\R^d} u\pr{x,t} \brac{\innp{A\gr u, {H^T h}} + 2 t \del_t u} K_{t}\pr{h\pr{x}} d{x} \\
&= -t \int_{\R^d} u\pr{x,t} \brac{\innp{A\gr u, \gr K_{t}\pr{h(x)}} + \di\pr{A \gr u} K_{t}\pr{h(x)}}  d{x} \\
&= -t \int_{\R^d} u\pr{x,t} \di \brac{A\gr u \, K_{t}\pr{h(x)}}  d{x}
= t \int_{\R^d} \innp{A\gr u, \gr u} K_{t}\pr{h(x)} d{x}
\end{align*}
and we deduce that
\begin{align}
\label{LtoL}
\lim_{n \to \iny} L_n(t)
&= \mathcal{L}\pr{t; u, h}.
\end{align}
Define $\widetilde{L}_n(t) = t^{\Upsilon_n} L_n(t)$ for some $\Upsilon_n \ge \norm{\gr \log \kappa_n \cdot y}_{L^\iny(B_{T}^n)}$.
Lemma \ref{ChainRuleLem1} shows that
$$\norm{\gr \log \kappa_n \cdot y}_{L^\iny(B_T^n)} \le \norm{\tr\pr{H \innp{\gr_z G(h) , h}}}_{L^\iny({\R^d})} =: \Upsilon,$$ 
so we may choose $\Upsilon_n = \Upsilon$ for all $n \in \N$.
Observe that
\begin{align*}
    \widetilde{L}_n(t)
    &= t^{\Upsilon} L_n(t)
    = \frac 1 2 t^{\Upsilon}  L\pr{\sqrt{2 d t}; v_n, \kappa_n}
    = \frac {1}{2 \pr{2d}^{\Upsilon}} \pr{2dt}^{\Upsilon}  L\pr{\sqrt{2 d t}; v_n, \kappa_n}
    = \frac {1}{2 \pr{2d}^{\Upsilon}} \widetilde{L}\pr{\sqrt{2 d t}; v_n, \kappa_n},
\end{align*}
where $\widetilde L$ is as given in Proposition \ref{variableAlmgren}.
An application of the chain rule shows that
\begin{align*}
    \frac{d}{dt}\widetilde{L}_n(t)
    &= \frac {1}{2 \pr{2d}^{\Upsilon}} \widetilde{L}'\pr{\sqrt{2 d t}; v_n, \kappa_n} \frac{\sqrt{2d}}{2 \sqrt{t}}.
\end{align*}
By Proposition \ref{variableAlmgren}, it follows that
\begin{align*}
\frac{d}{dt}\widetilde{L}_n(t)
&\ge t^{\Upsilon} \sqrt{\frac{d}{2t}} \brac{ \frac{\pr{ \int_{S_t^n} \kappa_n \, v_n \,{ y \cdot \gr v_n} } \pr{ \int_{B_t^n} \kappa_n \, \ell_n \, v_n  }}{ \pr{ \int_{S_t^n} \kappa_n \, \abs{v_n}^2 }^2 } 
- \frac{ \int_{B_t^n} \kappa_n \, \ell_n \, { y \cdot \gr v_n }}{ \int_{S_t^n} \kappa_n \, \abs{v_n}^2}}.
\end{align*}
By Lemmas \ref{PFT} and \ref{ChainRuleLem}
\begin{equation*}
\begin{aligned}
& \frac 1 {d \abs{S^{d n -1}} } \int_{B_t^n} \kappa_n(y) \, \ell_n\pr{y} \pr{ y \cdot \gr v_n\pr{y} }\abs{y}^{d n -2} \abs{y}^{2 - d n} d{y} \\
&= \int_0^t \int_{\R^d} J(x,s) \brac{ \innp{A\gr u, {H^T h}} + 2 s \, \del_s u} \pr{2 d s}^{\frac{d n -2}{2}} K_{n}\pr{h(x),s} d{x} \, d{s}.
\end{aligned}
\end{equation*}
Therefore, substituting this along with the expressions from \eqref{HExpression}, \eqref{D1Term} and \eqref{D2Term} into the previous inequality shows that
\begin{align*}
&\frac{d}{dt}\widetilde{L}_n(t)
\ge - d t^{\Upsilon-1} \frac{ \int_0^t \pr{\frac s t}^{\frac{d n -2}{2}} \int_{\R^d} J(x,s)
\brac{ \innp{A\gr u, {H^T h}} + 2 s \, \del_s u}  K_{n}\pr{h(x),s} d{x} \, d{s} }{2 \int_{\R^d} \abs{u\pr{x,t}}^2 K_{t,n}\pr{h\pr{x}} d{x}}  \\
&+ d t^{\Upsilon-1} \frac{\pr{\int_{\R^d} u\pr{x,t} \brac{\innp{A\gr u, {H^T h}} + 2 t \, \del_t u} K_{t,n}\pr{h\pr{x}} d{x}} \pr{ \int_0^t \pr{\frac s t}^{\frac{d n -2}{2}} \int_{\R^d}J(x,s) u\pr{x,s}  K_{n}\pr{h\pr{x}, s} dx \, ds }}{2 \pr{ \int_{\R^d} \abs{u\pr{x,t}}^2 K_{t,n}\pr{h\pr{x}} d{x}}^2 }.
\end{align*}
Estimate \eqref{KntKtBound} along with \eqref{aldnDefn} and \eqref{aldnBound} show that
 \begin{align}
\frac{d}{dt}\widetilde{L}_n(t)
&\ge- \frac{d \mathcal{C}_{d}t^{\Upsilon-1}}{2 \al_{d} \mathcal{H}_n\pr{t}} \brac{ \int_0^t \pr{\frac s t}^{\frac{d n -2}{2}} \mathcal{K}(s) \, d{s}
+ \frac{\mathcal{C}_d \mathcal{I}(t)}{\al_d \mathcal{H}_n(t)} \int_0^t \pr{\frac s t}^{\frac{d n -2}{2}} \mathcal{J}(s) \, d{s}}
=: \widetilde F_n(t)     
\label{FnBound},
\end{align}
where $\mathcal{I}$, $\mathcal{J}$, and $\mathcal{K}$ are defined in \eqref{IIntegrals} and we have introduced
\begin{align}
\label{HnDefn}
\mathcal{H}_n\pr{t}
= \mathcal{H}_n\pr{t; u, h} 
&:= \int_{\R^d} \abs{u(x,t)}^2 \pr{1 - \frac{\abs{h(x)}^2}{2 d n t}}^{\frac{d n - d - 2}{2}}\chi_{B_{nt}}(h(x)) dx. 
\end{align}

To show that $\widetilde{\mathcal{L}}$ is monotone non-decreasing, it suffices to show that given any $t_0 \in (0, T]$, there exists $\de \in \pr{0, t_0}$ so that $\widetilde{F}_n$ converges uniformly to $0$ on $\brac{t_0 - \de, t_0}$.
Indeed, since $\frac{d}{dt}\widetilde{L}_n(t) \ge \widetilde{F}_n(t)$, then for any $t \in \brac{t_0 - \de, t_0}$, it holds that
\[
\widetilde{L}_n(t_0)-\widetilde{L}_n(t)\ge \int_{t}^{t_0}\widetilde{F}_n(s)ds.
\]
By definition and \eqref{LtoL}, $\widetilde{L}_n(t) = t^\Upsilon L_n(t)$ converges pointwise to $\widetilde{\mathcal{L}}(t)=t^{\Upsilon}\mathcal{L}(t;u, h)$, from which it follows that 
\[
\widetilde{\mathcal{L}}(t_0) - \widetilde{\mathcal{L}}(t)
= \lim_{n \to \iny} \brac{\widetilde{L}_n(t_0)-\widetilde{L}_n(t)} 
\ge \lim_{n \to \iny} \int_{t}^{t_0}\widetilde{F}_n(s)ds.
\]
Assuming the local uniform convergence of $\widetilde F_n$ to 0 on $\brac{t_0 - \de, t_0} \supset \brac{t, t_0}$, we see that
$$\lim_{n \to \iny} \int_{t}^{t_0}\widetilde{F}_n(s)ds 
=  \int_{t}^{t_0} \lim_{n \to \iny} \widetilde{F}_n(s)ds = 0$$
and we may conclude that $\widetilde{\mathcal{L}}(t_0)-\widetilde{\mathcal{L}}(t) \ge 0$, as desired.

It remains to justify the local uniform convergence of $\widetilde{F}_n$ to $0$, as described above. 
Let $t_0 \in (0, T]$ and recall that since $u$ is non-trivial, then backward uniqueness ensures that $\disp \int_{\R^d} \abs{u\pr{x,t_0}}^2 dx > 0$ so that $\mathcal{L}\pr{t_0}$ is well-defined.

We first consider the terms in the denominator of $\widetilde F_n$, defined as $\mathcal{H}_n\pr{t}$ above.
Observe that for any $n \in \N$, we have from \eqref{HnDefn} that 
\begin{equation}
\label{denomFn}
\begin{aligned}
\mathcal{H}_n\pr{t} 
&\ge \int_{\set{\abs{h(x)}^2 \le dnt}} \abs{u\pr{x,t}}^2  \pr{1 - \frac{\abs{h(x)}^2}{2 d n t}}^{\frac{d n - d - 2}{2}} d{x} \\
&\ge \int_{\set{\abs{h(x)}^2 \le dnt}} \abs{u\pr{x,t}}^2  \exp\pr{- \frac{\ln 2\abs{h(x)}^2}{2 t} } d{x},
\end{aligned}
\end{equation}
where have used a Taylor expansion to show that if $\abs{h(x)}^2 \le dnt$, then
\begin{align*}
\log \brac{\pr{1 - \frac{\abs{h(x)}^2}{2 d n t}}^{\frac{d n - d - 2}{2}}}
&= - \frac{d n - d - 2}{dn} \frac{\abs{h(x)}^2}{4 t} \brac{ 1 + \frac 1 2 \pr{\frac{\abs{h(x)}^2}{2 d n t}} + \frac 1 3 \pr{\frac{\abs{h(x)}^2}{2 d n t}}^2 + \ldots } \\
&\ge - \frac{\abs{h(x)}^2}{4 t} \brac{ 1 + \frac 1 2 \pr{\frac 1 2} + \frac 1 3 \pr{\frac 1 2}^2 + \ldots }
= - \frac{\ln 2 \abs{h(x)}^2}{2 t} .
\end{align*}
Set $K_m = \set{x : \abs{h(x)} \le m}$ and note that each $K_m$ is compact, the sets are nested, and $\disp \R^d = \bigcup _{m \in \N} K_m$.
Thus, for any positive real number $\disp 2H \le \int_{\R^d} \abs{u\pr{x,t_0}}^2 dx$, there exists $M \in \N$ so that 
$$\int_{K_M} \abs{u\pr{t_0}}^2  d{x} \ge \frac 3 2 H.$$
Fix some $\mu < \min\set{\frac {t_0} 2, T - t_0}$.
Since $u$ is continuous and $K_M \times \brac{t_1 - \mu, t_1 + \mu}$ is compact, then there exists $\de \in (0, \mu]$ so that whenever $x \in K_M$ and $\abs{t - t_0} \le \de$, it holds that
$$\abs{u(x, t) - u(x, t_0)} < \sqrt{\frac{H}{2\abs{K_M}}}.$$
In particular, if $t \in \brac{t_0 - \de, t_0}$, then
\begin{align*}
 \int_{K_M} \abs{u\pr{x,t}}^2  d{x}
 &\ge \int_{K_M} \abs{u\pr{x,t_0}}^2  d{x}
 - \int_{K_M} \abs{u\pr{x,t} - u\pr{x,t_0}}^2  d{x}
 \ge H.
\end{align*}
If $N \in \N$ is large enough so that $d N \pr{t_0 - \de} \ge M^2$, then $\set{\abs{h(x)}^2 \le d N \pr{t_0 - \de}} \supset K_M$.
It follows that for any $n \ge N$ and any $t \in \brac{t_0 - \de, t_0}$, we have from \eqref{denomFn} that
\begin{equation}
\label{DnBound}
 \begin{aligned}
\mathcal{H}_n\pr{t}
&\ge \int_{\set{\abs{h(x)}^2 \le dNt}} \abs{u\pr{x,t}}^2  \exp\pr{- \frac{\ln 2\abs{h(x)}^2}{2 t} } d{x} 
\ge e^{- \frac{dN \ln 2 }{2} } \int_{K_M} \abs{u\pr{x,t}}^2  d{x} \\
&\ge H e^{- \frac{dN \ln 2 }{2} }
> 0.
\end{aligned}   
\end{equation}
In particular, we have a uniform lower bound on all $\mathcal{H}_n\pr{t}$ for $n \ge N$ and all $t \in \brac{t_0 - \de, t_0}$.

Since $u \in \mathfrak{A}\pr{\R^d \times \pr{0, T}, h}$, then there exists $p > 1$ so that \eqref{AlmgrenClass2} holds.
Returning to the expression \eqref{FnBound}, an application of H\"older's inequality shows that
\begin{equation*}
\begin{aligned}
\int_0^t \pr{\frac s t}^{\frac{d n -2}{2}} \mathcal{J}(s) \, d{s}
&= \int_0^t \pr{\frac s t}^{\frac{d n + d -2}{2}} \pr{\frac s t}^{-\frac{d}{2}} \mathcal{J}(s) \, d{s}
\le \brac{\int_0^t \pr{\frac s t}^{\frac{d n + d -2}{2}\frac{p}{p-1}} d{s}}^{\frac{p-1}{p}} \norm{\mathcal{J}}_{L^{p}\pr{\brac{0, t}, s^{- \frac d 2} ds}} \\
&= t^{\frac{p-1}{p}} \pr{\int_0^1 \tau^{\frac{p\pr{d n + d -2}}{2\pr{p-1}}} d \tau }^{\frac{p-1}{p}} \norm{\mathcal{J}}_{L^{p}\pr{\brac{0, t}, s^{- \frac d 2} ds}}
= \brac{\tfrac{2t\pr{p-1}}{n \pr{d p + \frac{dp-2}{n}}} }^{\frac{p-1}{p}} \norm{\mathcal{J}}_{L^{p}\pr{\brac{0, t}, s^{- \frac d 2} ds}}. 
\end{aligned}
\end{equation*}
In particular, 
\begin{equation}
\label{Ij3IntBound}
\sup_{t \in \brac{t_0 - \de, t_0}} \int_0^t \pr{\frac s t}^{\frac{d n -2}{2}} \mathcal{J}(s) \, d{s}
\le c_{p,d} \pr{\frac{t_0}{n} }^{1 -\frac{1}{p}} \norm{\mathcal{J}}_{L^{p}\pr{\brac{0, t}, s^{- \frac d 2} ds}}
\end{equation}
and an identical argument holds with $\mathcal{K}$ in place of $\mathcal{J}$.
Assuming that $\de \le \eps$ from \eqref{AlmgrenClass1} (which holds by possibly redefining $\de$), we put  \eqref{AlmgrenClass1}, \eqref{DnBound} and \eqref{Ij3IntBound} together in \eqref{FnBound} to see that whenever $n \ge N$,
\begin{align*}
\inf_{t \in \brac{t_0 - \de, t_0}} \widetilde F_n(t)
&\ge - \frac{ \mathcal{C}_{d,p}t_0^{\Upsilon-\frac 1 p}e^{ \frac{dN \ln 2 }{2} } }{2 \al_d H n^{1 - \frac 1 p}} \brac{ \norm{\mathcal{K}}_{L^{p}\pr{\brac{0, t_0}, t^{- \frac d 2} dt}}
+ \frac{\mathcal{C}_d e^{\frac{dN \ln 2}{2}}}{H} \norm{\mathcal{I}}_{L^\iny\pr{\brac{t_0 - \de,t_0}}} \norm{\mathcal{J}}_{L^{p}\pr{\brac{0, t_0}, t^{- \frac d 2} dt}}}.
\end{align*}
The required version of uniform convergence follows from this bound and completes the proof.
\end{proof}

\section{Alt-Caffarelli-Friedman Monotonicity Formula}
\label{S:ACF}

In the groundbreaking work of Alt-Caffarelli-Friedman \cite{ACF}, the authors study two-phase elliptic free boundary problems. 
The monotonicity formula described in Proposition \ref{ACFelliptic}, which we refer to as ACF, is a crucial tool in their work since it is used to establish Lipschitz continuity of minimizers, and study the regularity of the free boundary.  

\begin{prop}[\cite{ACF}, Lemma 5.1]
\label{ACFelliptic}
For some $R > 0$, let $B_R \su \R^N$.
Let $v_1, v_2$ be two non-negative functions that belong to $C^0\pr{B_R}\cap W^{1,2}\pr{B_R}$.
Assume that $\LP v_1 \ge 0$ and $\LP v_2 \ge 0$ in the sense of distributions, $v_1 v_2 \equiv 0$ and $v_1\pr{0} = v_2\pr{0} = 0$.
Then for all $r < R$,
\begin{equation}
\phi\pr{r; v} = \frac{1}{r^4} \pr{ \int_{B_r}  \abs{\gr v_1\pr{y}}^2 \abs{y}^{2-N} d{y} } \pr{ \int_{B_r} \abs{\gr v_2\pr{y}}^2 \abs{y}^{2-N} d{y}}
\label{2PPhiDef}
\end{equation}
is monotonically non-decreasing in $r$.
\label{HFBP}
\end{prop}

Different versions of this formula were proved in \cite{C} by Caffarelli, and by Caffarelli and Kenig in \cite{CK} to show the regularity of solutions to parabolic equations. 
Caffarelli, Jerison, and Kenig in \cite{CJK} further extended these ideas, proving a powerful uniform bound on the monotonicity functional, instead of a monotonicity result. 
Later on, Matevosyan and Petrosyan \cite{MP} proved another such uniform bound for non-homogeneous elliptic and parabolic operators with variable-coefficients. 
ACF-type monotonicity formulas have also been used to study almost minimizers of variable-coefficient Bernoulli-type functionals, see for example \cite{DESVGT}.

In this section, we prove a parabolic version of theorem \ref{ACFelliptic}, given below in Theorem \ref{T:ACFp}.
To the best of our knowledge, this result is also new. 

As in the previous section, to employ the high-dimensional limiting technique to prove this parabolic result, we first need a version of it for solutions to non-homogeneous variable-coefficient elliptic equations.
As similar results for homogeneous variable-coefficient elliptic equations have found numerous applications, this monotonicity result could be interesting in its own right.
Once we have the suitable elliptic result in hand, we employ techniques similar to those in the previous section to establish our parabolic ACF result.

\begin{cor}
\label{C:ACFellipticnonzero}
For some $R > 0$, let $B_R \su \R^N$.
For each $i = 1,2$, we make the following assumptions:
Let $\kappa_i: B_R \to \R_+$ be bounded, elliptic, and regular in the sense that $0 < \la \le \kappa_i \le \La < \iny$ in $B_R$ and $\gr \log\kappa_i \cdot y\in L^\iny(B_R)$.
Define $\Upsilon_i \ge \norm{\gr \log \kappa_i \cdot y}_{L^\iny(B_R)}$.
Let $v_i \in C^0\pr{B_R}\cap W^{1,2}\pr{B_R, \kappa_i \, dy}$ be a non-negative function for which $v_i\pr{0} = 0$, and
$\disp \di(\kappa_i \nabla v_i) \ge \kappa_i \, \ell_i$ in $B_R$ in the sense of distributions, where $\ell_i^-$ is integrable with respect to $\kappa_i \, v_i \abs{y}^{2 + \Upsilon_i - N} dy$ on each $B_r \su B_R$.
Assume further that for every $r < R$, $\Gamma_{i, r} := \text{supp}v_i \cap S_r := \text{supp}v_i \cap \del B_r$ has non-zero measure.
Finally, assume that $N \ge \max\set{\Upsilon_1 , \Upsilon_2} + 2$ and that $v_1 v_2 \equiv 0$.
Then for all $r < R$, we define $\phi\pr{r} = \phi\pr{r; v_1, v_2, \kappa_1, \kappa_2}$ as 
\[
\phi(r)=\frac{1}{r^{4+ \Upsilon_1 + \Upsilon_2}}\pr{\int_{B_r} \kappa_1(y) |\nabla v_1(y)|^2 |y|^{2+\Upsilon_1-N}  dy } \pr{\int_{B_r} \kappa_2(y) |\nabla v_2(y)|^2 |y|^{2+\Upsilon_2-N}  dy}.
\]
If we set $\widetilde \phi\pr{r} = r^{\mu} \phi\pr{r}$, where $\mu = 4\pr{\frac{\La -  \la}{\La}} + \Upsilon_1 + \Upsilon_2$, then it holds that
\begin{equation}
\label{phiTDer}
\begin{aligned}
\frac{d}{dr} \widetilde \phi\pr{r}
&\ge - \frac{2r^{4 \pr{\frac{\La - \la}{\La}} -3}}{ \pr{N-2}} \frac{\pr{\int_{B_r} \kappa_{1} \, \ell_{1}^- v_{1} \abs{y}^{2+ \Upsilon_1-N} dy} \pr{ \int_{B_r} \kappa_{2} \abs{\gr v_{2}}^2 \abs{y}^{2+ \Upsilon_2-N} d{y}} \pr{\int_{S_r} \kappa_{1} \abs{\gr v_{1}}^2 d\si } }{\pr{\int_{S_r}\kappa_{1} \abs{v_{1}}^2 d\si }}  \\
&- \frac{2r^{4 \pr{\frac{\La - \la}{\La}} -3}}{\pr{N-2}}  \frac{ \pr{ \int_{B_r} \kappa_{1} \abs{\gr v_{1}}^2 \abs{y}^{2+ \Upsilon_1-N} d{y} }  \pr{\int_{B_r} \kappa_{2} \, \ell_{2}^- v_{2} \abs{y}^{2+ \Upsilon_2-N} dy} \pr{\int_{S_r} \kappa_{2} \abs{\gr v_{2}}^2 d\si  }}{\pr{\int_{S_r}\kappa_{2} \abs{v_{2}}^2 d\si }}.
\end{aligned}
\end{equation}
\end{cor}

\begin{rem}
Notice that  the main difference between our $\phi(r)$ and the one defined in \eqref{2PPhiDef} is the introduction of $\Upsilon_1$ and $\Upsilon_2$, which depend on $\gr \kappa_1$ and $\gr \kappa_2$, respectively, appearing in the powers on $\abs{y}$.
In fact, if we set $\kappa_i = 1$, we recover a monotonicity formula very similar to \cite[Corollary 2]{Dav18}, which generalizes Proposition \ref{ACFelliptic} to the non-homogeneous setting.
If each $\ell_i = 0$, the right-hand side of \eqref{phiTDer} vanishes and we reach a true monotonicity result.
\end{rem}

\begin{proof}
Observe that for a.e. $r$,
\begin{align*}
\frac{d}{dr} \int_{B_r} \kappa_i \abs{\gr v_i}^2 \abs{y}^{2 + \Upsilon_i - N}  d{y} 
= r^{2 + \Upsilon_i - N} \int_{S_r} \kappa_i \abs{\gr v_i}^2  d\sigma(y).
\end{align*}
Therefore, for a.e. $r$,
\begin{equation}
\label{phiPrimeK}
\begin{aligned}
\phi^\prime\pr{r} 
=& - \frac{4+ \Upsilon_1 + \Upsilon_2}{r^{5+ \Upsilon_1 + \Upsilon_2}} \pr{ \int_{B_r} \kappa_1 \abs{\gr v_1}^2 \abs{y}^{2+ \Upsilon_1-N} d{y} } \pr{ \int_{B_r} \kappa_2 \abs{\gr v_2}^2 \abs{y}^{2+ \Upsilon_2-N} d{y}}   \\
&+ \frac{r^{2-N}}{r^{4 + \Upsilon_2}} \pr{  \int_{S_r} \kappa_1 \abs{\gr v_1}^2  d\sigma(y) } \pr{ \int_{B_r} \kappa_2 \abs{\gr v_2}^2 \abs{y}^{2+ \Upsilon_2-N}  d{y}}  \\
&+ \frac{r^{2-N}}{r^{4+ \Upsilon_1}} \pr{ \int_{B_r} \kappa_1 \abs{\gr v_1}^2 \abs{y}^{2+ \Upsilon_1-N} d{y} } \pr{ \int_{S_r} \kappa_2 \abs{\gr v_2}^2  d\sigma(y)}.
\end{aligned}
\end{equation}
We want to estimate this derivative.

Throughout this part of the proof, we suppress subscripts for $i = 1,2$ on all functions and exponents.
That is, in place of $v_i$, $\kappa_i$, $\ell_i$, $\Upsilon_i$, we write $v$, $\kappa$, $\ell$, and $\Upsilon$.
Since $\disp \di (\kappa \nabla v)\ge \kappa \, \ell$ in $B_R$, then $\Delta v\ge l$, where we define $l:= -\pr{\ell^- + \nabla \log \kappa \cdot \nabla v}$. 
Define $v_{m}$ and $l_{m}$ to be the mollifications of $v$ and $l$, respectively, at scale $m^{-1}$.
Let $A_{r, \eps} = B_r \setminus B_\eps$ for some $\eps \in \pr{0, r}$.
Integration by parts shows that
\begin{align*}
r^{1 + \Upsilon - N} \int_{S_r} \kappa &\abs{v_m}^2 d\si
= \int_{S_r} \kappa \abs{v_m}^2 \abs{y}^{\Upsilon-N} y \cdot \hat n \, d\si(y) \\
&= \int_{A_{r, \eps}} \di \pr{\kappa\abs{v_m}^2 \abs{y}^{\Upsilon-N} y} dy
+ \int_{S_\eps} \kappa\abs{v_m}^2 \abs{y}^{\Upsilon-N} y \cdot \hat n  \, d\si(y)\\
&=  \int_{A_{r, \eps}} \pr{\Upsilon \kappa + \gr \kappa\cdot y} \abs{v_m}^2 \abs{y}^{\Upsilon-N} dy
+ 2 \int_{A_{r, \eps}} \kappa \, v_m \gr v_m \cdot y \abs{y}^{\Upsilon-N} dy 
+ \eps^{1+ \Upsilon -N } \int_{S_\eps} \kappa\abs{v_m}^2  \, d\si(y).
\end{align*}
Moreover,
\begin{align*}
r^{1 + \Upsilon - N } \int_{S_r} \kappa &\, v_m\gr v_m\cdot y  \, d\si(y)
= \int_{S_r} \kappa \, v_m \abs{y}^{2+ \Upsilon-N} \gr v_m\cdot \hat n  \, d\si(y) \\
&= \int_{A_{r, \eps}} \di \pr{\kappa\gr v_m \, v_m \abs{y}^{2+ \Upsilon -N} } dy
+ \int_{S_\eps}  \kappa \, v_m \abs{y}^{2+ \Upsilon-N} \gr v_m\cdot \hat n  \, d\si(y) \\
&=  \int_{A_{r, \eps}} \kappa\abs{\gr v_m}^2 \abs{y}^{2+ \Upsilon-N} dy
+ \int_{A_{r, \eps}} \kappa \brac{ \LP v_m + \frac{\gr \kappa}{\kappa} \cdot \gr v_m} v_m \abs{y}^{2+ \Upsilon-N} dy \\
&+ \pr{2 + \Upsilon - N} \int_{A_{r, \eps}} \kappa \, v_m \gr v_m \cdot y \abs{y}^{\Upsilon-N} dy 
+ \eps^{2+ \Upsilon-N} \int_{S_\eps}  \kappa \, v_m\gr v_m\cdot \hat n  \, d\si(y).
\end{align*}
Since $\disp \LP v_{m} \ge l_{m} = - \pr{\ell^-}_{m} - \pr{\gr \log \kappa  \cdot \gr v}_{m}$ and $\Upsilon + \gr \log \kappa \cdot y \ge 0$ in $B_R$, then
\begin{align*}
& \frac{N - \Upsilon - 2}{2} r^{1 + \Upsilon - N} \int_{S_r} \kappa \abs{v_m}^2  \, d\si(y)
+ r^{1 + \Upsilon - N } \int_{S_r} \kappa \, v_m\gr v_m\cdot y  \, d\si(y) \\
\ge& \int_{A_{r, \eps}} \kappa\abs{\gr v_m}^2 \abs{y}^{2+ \Upsilon-N} dy
- \int_{A_{r, \eps}} \kappa \brac{\pr{\ell^-}_{m} + \pr{\gr \log \kappa  \cdot \gr v}_{m} - \gr \log \kappa \cdot \gr v_m} v_m \abs{y}^{2+ \Upsilon-N} dy \\
+& \frac{N - \Upsilon - 2}{2}\int_{A_{r, \eps}} \pr{\Upsilon + \gr \log \kappa \cdot y} \kappa\abs{v_m}^2 \abs{y}^{\Upsilon-N}  dy
+ I_\eps \\
\ge& \int_{A_{r, \eps}} \kappa\abs{\gr v_m}^2 \abs{y}^{2+ \Upsilon-N} dy
- \int_{A_{r, \eps}} \kappa \, \pr{\ell^-}_{m} \, v_m \abs{y}^{2+ \Upsilon-N} dy \\
-& \int_{A_{r, \eps}} \kappa \abs{\pr{\gr \log \kappa  \cdot \gr v}_{m} - \gr \log \kappa \cdot \gr v_m} v_m \abs{y}^{2+ \Upsilon -N} dy
+ I_\eps,
\end{align*}
where
\begin{align*}
I_{\eps} &= \frac{N - \Upsilon - 2}{2}\eps^{1+\Upsilon-N} \int_{S_\eps} \kappa\abs{v_m}^2  \, d\si(y)
+ \eps^{2+ \Upsilon-N} \int_{S_\eps}  \kappa \, v_m\gr v_m\cdot \hat n  \, d\si(y).
\end{align*}
By rearrangement, that $A_{r, \eps} \su B_r$ and $\Upsilon \ge 0$, it follows that
\begin{align*}
\int_{A_{r,\varepsilon}} &  \kappa\abs{\gr v_m}^2 \abs{y}^{2+ \Upsilon-N} dy
\le \frac{N - 2}{2} r^{1 + \Upsilon - N} \int_{S_r} \kappa \abs{v_m}^2  \, d\si(y)
+ r^{1 + \Upsilon - N } \int_{S_r} \kappa \, v_m\gr v_m\cdot y  \, d\si(y) \\
&+ \int_{B_{r}} \kappa \, \pr{\ell^-}_{m} v_m \abs{y}^{2+ \Upsilon-N} dy
+ \int_{B_{r}} \kappa \abs{\pr{\gr \log \kappa  \cdot \gr v}_{m} - \gr \log \kappa \cdot \gr v_m} v_m \abs{y}^{2+ \Upsilon-N} dy
- I_\eps.
\end{align*}
Since $\gr v_m$ is bounded and $\Upsilon \ge 0$, then $I_{\eps} \to 0$ as $\eps \to 0$.
In particular, the right-hand side of the previous inequality is bounded independent of $\eps$.
Accordingly, we may take $\eps \to 0$, eliminate $I_\eps$ on the right-hand side, and replace $A_{r, \eps}$ with $B_r$ on the left.

Now we integrate with respect to $r$, $r_0 < r < r_0 + \de$, divide through by $\de$, then take $m \to \iny$ to get
\begin{align*}
\fint_{r_0}^{r_0 + \de} \int_{B_r} \kappa\abs{\gr v}^2 & \abs{y}^{2+ \Upsilon-N} dy \, dr
\le \frac{N - 2}{2} \fint_{r_0}^{r_0 + \de} r^{1 + \Upsilon - N} \int_{S_r} \kappa \abs{v}^2  \, d\si(y) \, dr \\
&+ \fint_{r_0}^{r_0 + \de} r^{1 + \Upsilon - N } \int_{S_r} \kappa \, v\gr v\cdot y  \, d\si(y) \, dr 
+ \fint_{r_0}^{r_0 + \de} \int_{B_{r}} \kappa \, \ell^- v \abs{y}^{2+ \Upsilon-N} dy \, dr .
\end{align*}
By taking $\de \to 0$, it follows that for a.e. $r_0 > 0$
\begin{align*}
\int_{B_{r_0}} \kappa \abs{\gr v}^2 \abs{y}^{2+ \Upsilon-N} dy
&\le \frac{N - 2}{2} r_0^{1 + \Upsilon - N} \int_{S_{r_0}} \kappa \, \abs{v}^2 \, d\si(y)
+ r_0^{1 + \Upsilon - N } \int_{S_{r_0}} \kappa \, v\gr v\cdot y \, d\si(y) \\
&+ \int_{B_{r_0}} \kappa \, \ell^- v \abs{y}^{2+ \Upsilon -N} dy.
\end{align*}
After reintroducing the subscripts, we see that for a.e. $r > 0$, 
\begin{equation}
\label{gradKTermEst}
\begin{aligned}
\int_{B_{r}} \kappa_i\abs{\gr v_{i}}^2 \abs{y}^{2+ \Upsilon_i-N} dy
- \int_{B_{r}} \kappa_i \, \ell_{i}^- v_{i} \abs{y}^{2+ \Upsilon_i-N} dy
&\le \frac{N - 2}{2} r^{1 + \Upsilon_i - N} \int_{S_{r}} \kappa_i \abs{v_{i}}^2 \, d\si(y) \\
&+ r^{1 + \Upsilon_i - N } \int_{S_{r}} \kappa_i \, v_{i}\gr v_{i}\cdot y \, d\si(y).
\end{aligned}
\end{equation}
By rescaling, we may assume that $r = 1$.

Let $\gr_\te w$ denote the gradient of  function $w$ on $S^{N-1}$, the unit sphere.
Let $\Ga_i$ denote the support of $v_i$ on $S^{N-1}$ for $i = 1, 2$. 
By assumption, the measures of ${\Ga_1}$ and ${\Ga_2}$ are non-zero.
For $i = 1,2 $, define
\begin{align*}
\frac{1}{\al_i} = \inf \set{ \frac{\int_{\Ga_i} \abs{\gr_\te w}^2 }{\int_{\Ga_i} w^2} : w \in H^{1,2}_0\pr{\Ga_i}}.
\end{align*}
Observe that for any $\be_i \in \pr{0,1}$,
\begin{align*}
\frac{1 - \be_i^2}{\al_i} \int_{S_1} \abs{v_i}^2 d\sigma(y)
&\le \pr{1 - \be_i^2} \int_{S_1} \abs{\gr _\te v_i}^2 d\sigma(y) 
\end{align*}
and 
\begin{align*}
\frac{2 \be_i}{\sqrt{\al_i}} \int_{S_1} \abs{v_i} \abs{ \gr v_i \cdot y} d\sigma(y)
&\le 2 \pr{ \frac{\be_i^2}{ \al_i} \int_{S_1} \abs{v_i}^2 d\sigma(y)}^{\frac 1 2} \pr{ \int_{S_1} \pr{ \gr v_i \cdot y}^2 d\sigma(y)}^{\frac 1 2} \\
&\le \frac{\be_i^2}{\al_i} \int_{S_1} \abs{v_i}^2 d\sigma(y)
+ \int_{S_1} \abs{ \gr_r v_i}^2 d\sigma(y)
\le \int_{S_1} \brac{\be_i^2\abs{\gr _\te v_i}^2 + \abs{ \gr_r v_i}^2}d\sigma(y) .
\end{align*}
If we set 
\begin{equation}
\label{albeRel2}
\frac{1 - \be_i^2}{\al_i} =  \frac{\be_i \pr{N-2}}{\sqrt{\al_i}}
\end{equation}
for $i = 1,2$, then by combining \eqref{gradKTermEst} with the last two inequalities and using that $\la \le \kappa_i(y) \le \La$, we have
\begin{align*}
\int_{B_1} \kappa_i \abs{\gr v_i}^2 \abs{y}^{2+\Upsilon_i-N} d{y} 
&- \int_{B_1} \kappa_i \, \ell^-_{i} v_{i} \abs{y}^{2+ \Upsilon_i-N} dy
\le \frac{N - 2}{2} \int_{S_{1}} \kappa_i \abs{v_{i}}^2 \, d\si(y)
+ \int_{S_1} \kappa_i \, v_{i}\gr v_{i}\cdot y \, d\si(y) \\
&\le  \frac{\La \pr{N - 2}}{2} \int_{S_{1}} \abs{v_{i}}^2\, d\si(y)
+ \La \int_{S_1} \abs{v_{i}} \abs{\gr v_{i}\cdot y} \, d\si(y) \\
&\le \frac{\La \sqrt \al_i}{2\be_i}  \int_{S_{1}} \abs{\gr v_{i}}^2\, d\si(y)
\le \frac{\La \sqrt \al_i}{2 \la \be_i}  \int_{S_{1}} \kappa_i \abs{\gr v_{i}}^2\, d\si(y).
\end{align*}
Substituting these inequalities into \eqref{phiPrimeK} with $r = 1$ gives
\begin{equation*}
\begin{aligned}
\phi^\prime\pr{1} 
&\ge - \pr{4+ \Upsilon_1 + \Upsilon_2} \pr{ \int_{B_1} \kappa_1 \abs{\gr v_1}^2 \abs{y}^{2+ \Upsilon_1-N} d{y} } \pr{ \int_{B_1} \kappa_2 \abs{\gr v_2}^2 \abs{y}^{2+ \Upsilon_2-N} d{y}}   \\
&+ \frac{2 \la \be_1}{\La \sqrt \al_1} \pr{ \int_{B_1} \kappa_1 \abs{\gr v_1}^2 \abs{y}^{2+\Upsilon_1-N} d{y}  - \int_{B_1} \kappa_1 \, \ell^-_{1} v_{1} \abs{y}^{2+ \Upsilon_1-N} dy} \pr{ \int_{B_1} \kappa_2 \abs{\gr v_2}^2 \abs{y}^{2+ \Upsilon_2-N} d{y}}  \\
&+ \frac{2 \la \be_2}{\La \sqrt \al_2} \pr{ \int_{B_1} \kappa_1 \abs{\gr v_1}^2 \abs{y}^{2+ \Upsilon_1-N} d{y} } \pr{\int_{B_1} \kappa_2 \abs{\gr v_2}^2 \abs{y}^{2+\Upsilon_2-N} d{y} 
- \int_{B_1} \kappa_2 \, \ell^-_{2} v_{2} \abs{y}^{2+ \Upsilon_2-N} dy} \\
&= \pr{\frac{2 \la}{\La}\brac{\frac{\be_1}{\sqrt \al_1}+  \frac{\be_2}{\sqrt \al_2}} - 4- \Upsilon_1 - \Upsilon_2} \pr{ \int_{B_1} \kappa_1 \abs{\gr v_1}^2 \abs{y}^{2+ \Upsilon_1-N} d{y} } \pr{ \int_{B_1} \kappa_2 \abs{\gr v_2}^2 \abs{y}^{2+ \Upsilon_2-N} d{y}}   \\
&- \frac{2 \la \be_1}{\La \sqrt \al_1} \pr{\int_{B_1} \kappa_1 \, \ell_{1}^- v_{1} \abs{y}^{2+ \Upsilon_1-N} dy} \pr{ \int_{B_1} \kappa_2 \abs{\gr v_2}^2 \abs{y}^{2+ \Upsilon_2-N} d{y}}  \\
&- \frac{2 \la \be_2}{\La \sqrt \al_2} \pr{ \int_{B_1} \kappa_1 \abs{\gr v_1}^2 \abs{y}^{2+ \Upsilon_1-N} d{y} } \pr{\int_{B_1} \kappa_2 \, \ell_{2}^- v_{2} \abs{y}^{2+ \Upsilon_2-N} dy}.
\end{aligned}
\end{equation*}
The relation \eqref{albeRel2} is satisfied when
$$\frac{\be_i}{\sqrt{\al_i}} = \frac 1 2 \set{\brac{\pr{N-2}^2 + \frac{4}{\al_i}}^{\frac 1 2} - \pr{N-2}}.$$
If we define $\ga_i > 0$ so that $ \ga_i \pr{\ga_i + N - 2} = \frac 1 {\al_i}$ then $ \frac{\be_i}{\sqrt{\al_i}} = \ga_i$ for $i = 1,2$.

As a function that acts on subsets of $S^{N-1}$, $\ga$ was studied in \cite{FH} and it was shown that $ \ga\pr{E} \ge \psi\pr{\frac{\abs{E}}{\abs{S^{N-1}}}}$, where $\psi$ is the decreasing, convex function defined by
\begin{equation*}
\psi\pr{s} = \left\{\begin{array}{ll} \frac 1 2 \log\pr{\frac 1 {4s}} + \frac 3 2 & \text{ if } s < \frac 1 4 \\
2 \pr{1 - s} & \text{ if } \frac 1 4 < s < 1.  \end{array}  \right.
\end{equation*}
We use the notation $\ga_i = \ga\pr{\Ga_i}$ for $i = 1,2$. 
With $ s_i = \frac{\abs{\Ga_i}}{\abs{S^{N-1}}}$, it follows from convexity that
\begin{align*}
\ga_1 + \ga_2
\ge \psi\pr{s_1} + \psi\pr{s_2}
\ge 2\psi\pr{\frac{s_1 + s_2}{2}}
\ge 2\psi\pr{\frac{1}{2}}
= 2.
\end{align*}
Moreover, since each $v_i \in H^{1,2}_0\pr{\Ga_i}$, then
\begin{equation*}
\begin{aligned}
\ga_i 
&= \frac {N-2} 2 \set{\brac{1 + \frac{4}{\al_i\pr{N-2}^2}}^{\frac 1 2} - 1}
\le \frac {N-2} 2 \frac{2}{\al_i\pr{N-2}^2}
= \frac {1}{\al_i \pr{N -2}}
\le \frac{1}{N-2} \frac{\int_{\Ga_i} \abs{\gr_\te v_i}^2 }{\int_{\Ga_i} \abs{v_i}^2} \\
&\le \frac{1}{N-2} \frac{\int_{S_1} \abs{\gr v_i}^2 }{\int_{S_1} \abs{v_i}^2}
\le \frac{\La}{\la \pr{N-2}} \frac{\int_{S_1} \kappa_i \abs{\gr v_i}^2  }{\int_{S_1}\kappa_i \abs{v_i}^2 }.
\end{aligned}
\end{equation*}
Therefore,
\begin{equation*}
\begin{aligned}
\phi^\prime\pr{1} 
&\ge - \brac{4 \pr{\frac{\La - \la}{\La}} + \Upsilon_1 + \Upsilon_2} \phi\pr{1} \\ 
&- \frac{2}{\pr{N-2}}\pr{\int_{B_1} \kappa_1 \, \ell_{1}^- v_{1} \abs{y}^{2+ \Upsilon_1-N} dy} \pr{ \int_{B_1} \kappa_2 \abs{\gr v_2}^2 \abs{y}^{2+ \Upsilon_2-N} d{y}}  \frac{\int_{S_1} \kappa_1 \abs{\gr v_1}^2  }{\int_{S_1}\kappa_1 \abs{v_1}^2 }  \\
&- \frac{2}{\pr{N-2}} \pr{ \int_{B_1} \kappa_1 \abs{\gr v_1}^2 \abs{y}^{2+ \Upsilon_1-N} d{y} } \pr{\int_{B_1} \kappa_2 \, \ell_{2}^- v_{2} \abs{y}^{2+ \Upsilon_2-N} dy}  \frac{\int_{S_1} \kappa_2 \abs{\gr v_2}^2  }{\int_{S_1}\kappa_2 \abs{v_2}^2 }.
\end{aligned}
\end{equation*}

For values of $r \ne 1$, define $v_{i, r}\pr{y} = r^{-1} v_i\pr{r y}$ for $i = 1, 2$ and $\kappa_{i,r}(y)=\kappa_i(ry)$, so that for $0<s,r\le 1,$
\begin{align*}
\phi\pr{s; v_{1, r}, v_{2, r}, \kappa_{1,r}, \kappa_{2,r}}
&=  \frac{1}{s^{4+ \Upsilon_1 + \Upsilon_2}}\pr{ \int_{B_s}  \abs{\gr v_{1,r}}^2 \abs{y}^{2 + \Upsilon_1 -N} \kappa_{1,r} \, d{y} } \pr{ \int_{B_s} \abs{\gr v_{2,r}}^2 \abs{y}^{2+ \Upsilon_2-N} \kappa_{2,r} \, d{y}} \\
&= \frac{1}{(sr)^{4+ \Upsilon_1 + \Upsilon_2}} \pr{ \int_{B_{rs}}  \abs{\gr v_1}^2 \abs{y}^{2 + \Upsilon_1 -N} \kappa_1 \, d{y} } \pr{ \int_{B_{rs}} \abs{\gr v_2}^2 \abs{y}^{2+ \Upsilon_2-N} \kappa_2 \, d{y} } \\
&= \phi\pr{rs; v_1, v_2, \kappa_1, \kappa_2}.
\end{align*}
Thus,
$$\phi'\pr{s; v_{1, r}, v_{2, r}, \kappa_{1,r}, \kappa_{2,r}} 
= \frac{d}{ds}\phi\pr{s; v_{1, r}, v_{2, r}, \kappa_{1,r}, \kappa_{2,r}}
= \frac{d}{ds} \phi\pr{rs; v_1, v_2, \kappa_1, \kappa_2} 
= r \phi'\pr{rs; v_1, v_2, \kappa_1, \kappa_2}$$
and taking $s=1$ gives 
$$\phi'(1; v_{1, r}, v_{2, r}, \kappa_{1,r}, \kappa_{2,r})=r\phi'(r;v_1, v_2, \kappa_1, \kappa_2).$$
Let $\Ga_{i, r}$ denote the support of  $v_{i}$ on $S_r$ and set $ s_{i, r} = \frac{\abs{\Ga_{i, r}}}{r^{N-1} \abs{S^{N-1}}}$ for $i = 1, 2$. 
With $\ell_{i, r}\pr{y} = r \ell_i\pr{r y}$, we have $\disp \text{div}(\kappa_{i,r} \nabla v_{i, r}) \ge \kappa_{i,r} \ell_{i, r}$.
Applying the derivative estimates above to the pair $v_{1, r}$, $v_{2, r}$ then rescaling leads to
\begin{equation*}
\begin{aligned}
\phi^\prime\pr{r} 
&= \frac 1 r \phi'(1; v_{1, r}, v_{2, r}, \kappa_{1,r}, \kappa_{2,r})
\ge - \frac 1 r \brac{4 \pr{\frac{\La - \la}{\La}} + \Upsilon_1 + \Upsilon_2} \phi(1; v_{1, r}, v_{2, r}, \kappa_{1,r}, \kappa_{2,r}) \\ 
&- \frac{2}{r\pr{N-2}}\pr{\int_{B_1} \kappa_{1,r} \, \ell_{1,r}^- v_{1,r} \abs{y}^{2+ \Upsilon_1-N} dy} \pr{ \int_{B_1} \kappa_{2,r} \abs{\gr v_{2,r}}^2 \abs{y}^{2+ \Upsilon_2-N} d{y}}  \frac{\int_{S_1} \kappa_{1,r} \abs{\gr v_{1,r}}^2  }{\int_{S_1}\kappa_{1,r} \abs{v_{1,r}}^2 }  \\
&- \frac{2}{r\pr{N-2}} \pr{ \int_{B_1} \kappa_{1,r} \abs{\gr v_{1,r}}^2 \abs{y}^{2+ \Upsilon_1-N} d{y} } \pr{\int_{B_1} \kappa_{2,r} \, \ell_{2,r}^- v_{2,r} \abs{y}^{2+ \Upsilon_2-N} dy}  \frac{\int_{S_1} \kappa_{2,r} \abs{\gr v_{2,r}}^2  }{\int_{S_1}\kappa_{2,r} \abs{v_{2,r}}^2 } \\
&\ge - \frac 1 r \brac{4 \pr{\frac{\La - \la}{\La}} + \Upsilon_1 + \Upsilon_2} \phi\pr{r; v} \\ 
&- \frac{2}{r^{3 + \Upsilon_1 + \Upsilon_2} \pr{N-2}}\pr{\int_{B_r} \kappa_{1} \, \ell_{1}^- v_{1} \abs{y}^{2+ \Upsilon_1-N} dy} \pr{ \int_{B_r} \kappa_{2} \abs{\gr v_{2}}^2 \abs{y}^{2+ \Upsilon_2-N} d{y}}  \frac{\int_{S_r} \kappa_{1} \abs{\gr v_{1}}^2  }{\int_{S_r}\kappa_{1} \abs{v_{1}}^2 }  \\
&- \frac{2}{r^{3 + \Upsilon_1 + \Upsilon_2}\pr{N-2}} \pr{ \int_{B_r} \kappa_{1} \abs{\gr v_{1}}^2 \abs{y}^{2+ \Upsilon_1-N} d{y} } \pr{\int_{B_r} \kappa_{2} \, \ell_{2}^- v_{2} \abs{y}^{2+ \Upsilon_2-N} dy}  \frac{\int_{S_r} \kappa_{2} \abs{\gr v_{2}}^2  }{\int_{S_r}\kappa_{2} \abs{v_{2}}^2 } .
\end{aligned}
\end{equation*}
That is, with $\mu = 4 \pr{\frac{\La - \la}{\La}} + \Upsilon_1 + \Upsilon_2$, the conclusion described by \eqref{phiTDer} follows.
\end{proof}

In what follows, we introduce our new parabolic ACF-type result for variable-coefficient operators.
We prove this result using only Corollary \ref{C:ACFellipticnonzero} and the tools and ideas that have been developed so far in this paper.
As in the previous section, we first discuss the kinds of solutions that we work with.

Let $u: \R^d \times \pr{0, T} \to \R$ have moderate $h$-growth at infinity, as described in Definition \ref{hGrowth}.
For every $t \in \pr{0, T}$, assume first that $u$ is sufficiently regular to define the functionals $\mathcal{D}\pr{t; u, h}$ and $\mathcal{T}\pr{t; u, h}$ from \eqref{HDTDefs} as well as 
\begin{equation*}
\begin{aligned}
\mathcal{J}^-(t)
&= \mathcal{J}^-(t; u,h) 
= \int_{\R^d} J^-(x,t) \abs{u\pr{x,t}} \exp\pr{-\frac{\abs{h(x)}^2}{4t}} dx,
\end{aligned}
\end{equation*} 
where $J(x,t) = J(x, t; u,h)$ is as defined in \eqref{JDefn} and all derivatives are interpreted in the weak sense.
Then we say that such a function $u$ belongs to the function class $\mathfrak{C}\pr{\R^d \times \pr{0, T}, h}$ if $u$ has moderate $h$-growth at infinity (so is consequently continuous), and for every $t_0 \in \pr{0, T}$, there exists $\eps \in \pr{0, t_0}$ so that
\begin{equation}
\label{ACF1}
\mathcal{D} \in L^\iny\brac{t_0 - \eps, t_0}, 
\quad
\mathcal{T} \in L^\iny\brac{t_0 - \eps, t_0}
\end{equation}
and 
\begin{equation}
\label{ACF2}   
\mathcal{J}^- \in L^{1}\pr{\brac{0, t_0}, t^{- \frac d 2} dt}.
\end{equation}
This is the class of functions that we consider in our result.
As before, this may not be the weakest setting in which our proof holds.

\begin{thm}
\label{T:ACFp}
For each $i = 1, 2$, we make the following assumptions:
Let $\tr\pr{H_i \innp{\gr_z G_i(h_i),h_i}} \in L^\iny(\R^d)$, where $h_i$, $H_i$, and $G_i$ are described by \eqref{xDef}, \eqref{zDef}, and \eqref{Jacobians}.
Define $\Upsilon_i := \norm{ \tr (H_i \innp{\nabla G_i(h_i), h_i}}_{L^\iny\pr{\R^d}}$ and set $A_i = H_i^{-1}\pr{H_i^{-1}}^T: \R^d \to \R^{d \times d}$.
Let $u_i \in \mathfrak{C}\pr{\R^d \times \pr{0, T}, h_i}$ be a non-negative function with $u_i\pr{0,0} = 0$ and $\displaystyle \di \pr{A_i \gr u_i} + \del_t u_i \ge 0$ in $\R^d \times \pr{0,T}$ in the sense of distributions.  
Assume also that $u_1 u_2 \equiv 0$.
For every $t \in \pr{0, T}$, define $\Phi\pr{t} = \Phi\pr{t; u_1, u_2, h_1, h_2}$ as
\begin{equation*}
\Phi\pr{t} 
= \frac{1}{t^{2}} \pr{\int_0^t \int_{\R^d} \pr{\frac s t}^{\frac{\Upsilon_1}2} \langle A_1 \nabla u_1,\nabla u_1\rangle e^{-\frac{\abs{h_1(x)}^2}{4s}} dx \, ds} \pr{\int_0^t \int_{\R^d} \pr{\frac s t}^{\frac{\Upsilon_2}2} \langle A_2 \nabla u_2,\nabla u_2\rangle e^{-\frac{\abs{h_2(x)}^2}{4s}} dx \, ds}.
\end{equation*}
Set $\widetilde \Phi\pr{t} = t^{\frac {\mu} 2} \Phi\pr{t}$, where $\mu = 4\pr{\frac{\La -  \la}{\La}} + \Upsilon_1 + \Upsilon_2$ with $\disp \la^{-1} = \max_{i=1,2}\set{\norm{\det H_i}_{L^\iny}}$ and $\disp \La = \max_{i = 1,2}\set{\norm{\det H_i^{-1}}_{L^\iny}}$. 
Then $\widetilde \Phi\pr{t}$ is monotonically non-decreasing in $t$.
\end{thm}

\begin{proof}
Recall that $B_T^n \su \R^{d \times n}$ is given by \eqref{BtnDefn}.
For each $i = 1,2$, let $\kappa_{i,n}$ be as defined in Lemma \ref{ChainRuleLem1}.
By definition, $\la \le \kappa_{i,n} \le \La$ for every $n \in \N$.
As 
$$\norm{\gr_y \log \kappa_{i, n} \cdot y}_{L^\iny\pr{B_{\sqrt{2dT}}}}
\le \norm{ \text{tr}(H_i \innp{\nabla G_i(h_i), h_i})}_{L^\iny\pr{\R^d}} = \Upsilon_i,$$
then $\gr_y \log \kappa_{i, n} \cdot y \in L^\iny\pr{B_n^T}$.
Define each $\Upsilon_{i,n} := \Upsilon_i$ to be independent of $n$.

Let $u_1$, $u_2$ be as in the statement of the theorem.
For each $i = 1,2$ and $n \in \N_{\ge N}$, define $v_{i, n} : B_T^n \su \R^{d \times n} \to \R$ so that
$$v_{i,n}\pr{y} = u_i\pr{F_{d,n}\pr{y}}.$$
Since each $u_i$ has moderate $h_i$-growth at infinity, then by Lemma \ref{functionClassLem}, $v_{i,n} \in C^0\pr{B_T^n} \cap W^{1,2}\pr{B_T^n, \kappa_{i,n} \, dy}$.
Moreover, each $v_{i,n}$ is non-negative and $v_{1,n} v_{2,n} \equiv 0$.
Assuming without loss of generality that $g(0)=0$, we obtain $v_{i,n}\pr{0} = u_i\pr{0,0} = 0$.

An application of Lemma \ref{ChainRuleLem} shows that
\begin{align*}
\frac{\text{div} ( \kappa_{i,n}\nabla v_{i,n})}{\kappa_{i,n}}
&= n \set{ \text{div} (A_i \nabla u_i) 
+  \frac{\partial u_i}{\partial t}
+ \frac{1}{dn}\brac{2 \innp{A_i \nabla \partial_t u_i , H_i^T h_i}
+ \frac{\partial u_i}{\partial t} \text{tr}(H_i \innp{\nabla_z G_i(h_i), h_i})
+ 2t \frac{\partial^2 u_i}{\partial t^2} }} \\
&\ge \frac{1}{d} \brac{ 2 \innp{A_i \nabla \partial_t u_i , H_i^T h_i}
+ \frac{\partial u_i}{\partial t} \text{tr}(H_i \innp{\nabla_z G_i(h_i), h_i})
+ 2t \frac{\partial^2 u_i}{\partial t^2} }
=:  J_i(x,t),
\end{align*}
where $J_i = J(u_i)$ is defined in \eqref{JDefn} and does not depend on $n$. 
For every $n$, define $\ell_{i,n} : B_{T}^n \to \R$ so that 
$$\ell_{i,n}\pr{y} = J_i\pr{F_{d,n}\pr{y}}$$
and then
$$\di \pr{\kappa_{i,n} \gr v_{i,n}} \ge \kappa_{i,n} \ell_{i,n}.$$

Since each $u_i \in \mathfrak{C}\pr{\R^{d} \times \pr{0,T}, h_i}$, then \eqref{ACF2} in combination with Corollary \ref{functionClassCor} shows that $\ell_{i, n}^- v_{i, n}$ is integrable with respect to $\kappa_n \, dy$ in each $B_t^n$.
In other words, each $\ell_{i,n}^-$ is integrable with respect to $\kappa_{i,n} v_{i,n} \abs{y}^{2 + \Upsilon_i - dn} dy$ on each $B_t^n$.

Let $\Ga_{i, n, t} = \text{supp} \ v_{i, n} \cap S_t^n$. 
For each $t$, the measure of $\text{supp}\ u_i\pr{\cdot, t}$ vanishes if and only if the measure of $\Ga_{i, n, t}$ vanishes for every $n$.
We assume first that for every $t$, the measures of $\text{supp} \ u_1\pr{\cdot, t}$ and $\text{supp} \ u_2\pr{\cdot, t}$ are non-vanishing.
Therefore, for every $i$, $n$ and $t$, $\Ga_{i, n, t}$ has non-zero measure. 
Thus, we may apply Corollary \ref{C:ACFellipticnonzero} to each pair $v_{1, n}$, $v_{2, n}$ on any ball $B_t^n$ for $t \in \pr{0, T}$.

Define $\Phi_n\pr{t} = \frac{4}{\pr{n \abs{S^{d  n - 1}}}^2} \phi\pr{\sqrt{2 d t}; v_{1, n}, v_{2, n}, \kappa_{1, n}, \kappa_{2, n}}$, where $\phi(r)$ is given in Corollary \ref{C:ACFellipticnonzero}.
By Lemmas \ref{ChainRuleLem} and \ref{PFT} 
\begin{equation}
\label{grSqTerm}
\begin{aligned}
\int_{B_t^n} \kappa_{i, n}(y) \abs{\gr v_{i,n}(y) }^2 &  \abs{y}^{2 + \Upsilon_i -d n} dy 
= dn \abs{S^{d n - 1 }} \int_0^t \int_{\R^d} \pr{2 d s}^{\frac{\Upsilon_i}2} \innp{A_i \gr u_i, \gr u_i} K_n(h_i(x),s)dx ds \\
+& 2 \abs{S^{d n - 1 }} \int_0^t \int_{\R^d} \pr{2 d s}^{\frac{\Upsilon_i}2} \frac{\del u_i}{\del s} \brac{\innp{A_i \gr u_{i}, H_i^T h_i} + s \frac{\del u_i}{\del s} } K_n(h_i(x),s)dx ds .
\end{aligned}
\end{equation}
Therefore,
\begin{align*}
\Phi_n\pr{t}
&=  \frac{4 }{\pr{n \abs{S^{d n - 1}}}^2} \frac{1}{\pr{2 d t}^{2 + \frac{\Upsilon_1 + \Upsilon_2}{2}}} \pr{ \int_{B_t^n} \kappa_{1,n}(y) \abs{\gr v_{1, n}\pr{y}}^2 \abs{y}^{2 + \Upsilon_1 - d  n}  dy }  \pr{ \int_{B_t^n} \kappa_{2,n}(y)\abs{\gr v_{2, n}\pr{y}}^2 \abs{y}^{2 + \Upsilon_2 - d  n}  dy} \\
&= \frac{1}{t^{2}} \int_0^t \int_{\R^d} \pr{\frac s t}^{\frac{\Upsilon_1}2}  \brac{ \innp{A_1\, \gr u_1, \gr u_1}
+ \frac{2}{nd} \frac{\del u_1}{\del s} \pr{\innp{A_1 \, \gr u_{1}, H_1^T h_1}
+ s \frac{\del u_1}{\del s} } } K_{n}(h_1(x), s)\, dx ds  \\
&\times \int_0^t \int_{\R^d} \pr{\frac s t}^{\frac{\Upsilon_2}2}  \brac{ \innp{A_2\, \gr u_2, \gr u_2} 
+ \frac{2}{nd} \frac{\del u_2}{\del s} \pr{\innp{A_2\, \gr u_{2}, H_2^T h_2}
+ s \frac{\del u_2}{\del s} } } K_{n}(h_2(x), s) \, dx ds.
\end{align*}
It follows from \eqref{KtnDefn}, \eqref{KntKtBound} and the Dominated Convergence Theorem that
\begin{align*}
\lim_{n \to \iny} \Phi_n\pr{t} 
&= \frac{1}{t^{2}} \prod_{i = 1,2} \pr{ \int_0^t \int_{\R^d} \pr{\frac s t}^{\frac{\Upsilon_i}2} \innp{A_i\, \gr u_i, \gr u_i} K(h_i(x), s)\, dx ds} 
= \Phi\pr{t}.
\end{align*}

Set $\mu = 4\pr{\frac{\La -  \la}{\La}} + \Upsilon_1 + \Upsilon_2$. 
If we then define 
$$\widetilde \Phi_n(t) 
= \frac{4 \pr{2d}^{- \frac {\mu}2}}{\pr{n \abs{S^{d n - 1}}}^2} \widetilde \phi\left(\sqrt{2 d t}\right) 
= \frac{4 \pr{2d}^{- \frac {\mu} 2}}{\pr{n \abs{S^{d n - 1}}}^2} \pr{2 d t}^{\frac {\mu} 2}\phi\left(\sqrt{2 d t}\right),$$
then by analogous arguments, we see that
\begin{equation}
\label{limit}
\lim_{n \to \iny} \widetilde \Phi_n\pr{t} 
= \widetilde \Phi\pr{t} 
:= t^{\frac {\mu} 2} \Phi\pr{t}.
\end{equation}
Moreover, an application of the chain rule shows that
\begin{align*}
\frac{d}{dt} \widetilde \Phi_n(t)
&= \frac{4 \pr{2d}^{- \frac {\mu} 2}}{\pr{n \abs{S^{d n - 1}}}^2} \widetilde \phi^\prime\left(\sqrt{2 d t}\right)  \sqrt{\frac{d}{2t}}
= \frac{2 \pr{2d}^{\frac{1- \mu} 2}}{\pr{n \abs{S^{d  n - 1}}}^2} \frac 1 {\sqrt t} \widetilde \phi^\prime\left(\sqrt{2 d t}\right).
\end{align*}
By Corollary \ref{C:ACFellipticnonzero},
\begin{align*}
\widetilde \phi^\prime\left(\sqrt{2 d t}\right)
&\ge - 2 \frac{\pr{\int_{B_n^t} \kappa_{1,n} \, \ell_{1,n}^- v_{1,n} \abs{y}^{2+ \Upsilon_1-d n} dy} \pr{ \int_{B_n^t} \kappa_{2,n} \abs{\gr v_{2,n}}^2 \abs{y}^{2+ \Upsilon_2-d n} d{y}} \pr{\fint_{S_n^t} \kappa_{1,n} \abs{\gr v_{1,n}}^2 d\si } }{\pr{d n-2} \pr{2 d t}^{ \frac 3 2 - 2 \pr{\frac{\La - \la}{\La}}} \pr{\fint_{S_n^t}\kappa_{1,n} \abs{v_{1,n}}^2 d\si }}  \\
&- 2 \frac{ \pr{ \int_{B_n^t} \kappa_{1,n} \abs{\gr v_{1,n}}^2 \abs{y}^{2+ \Upsilon_1-d n} d{y} }  \pr{\int_{B_n^t} \kappa_{2,n} \, \ell_{2,n}^- v_{2} \abs{y}^{2+ \Upsilon_2-d n} dy} \pr{\fint_{S_n^t} \kappa_{2,n} \abs{\gr v_{2,n}}^2 d\si  }}{\pr{d n-2} \pr{2 d t}^{ \frac 3 2 - 2 \pr{\frac{\La - \la}{\La}}}  \pr{\fint_{S_n^t}\kappa_{2,n} \abs{v_{2,n}}^2 d\si }}.
\end{align*}
Using the functionals that were introduced in \ref{HnDefn} and \ref{HDTDefs}, applications of Lemma \ref{PFTS} show that
\begin{align*}
& \fint_{S_n^t} \kappa_{i,n}(y) \abs{v_{i,n}(y)}^2 \, d\si(y)
= \int_{\R^d} \abs{u_i\pr{x,t}}^2 K_{t, n}\pr{h_i\pr{x}} d{x} \\
&= \frac{\abs{S^{d n -1 - d}} }{\abs{S^{d n -1}} \pr{2 d n t}^{\frac{d}{2}}} \int_{\R^d} \abs{u_i\pr{x,t}}^2   \pr{1 - \frac{\abs{h_i(x)}^2}{2 d n t}}^{\frac{d n - d - 2}{2}}\chi_{B_{nt}}(h_i(x)) d{x}
= \frac{\abs{S^{d n -1 - d}} }{\abs{S^{d n -1}} \pr{2 d n t}^{\frac{d}{2}}} \mathcal{H}_n\pr{t; u_i, \kappa_i}
\end{align*}
and
\begin{align*}
\fint_{S_n^t} \kappa_{i,n}(y)  \abs{\gr v_{i,n}(y)}^2 \, d\si(y) 
&= n \int_{\R^d} \brac{\innp{A_i \gr u_i, \gr u_i} +\frac 2 {dn}\frac{\del u_i}{\del t} \pr{\innp{A_i \gr u_i, H_i^T h_i}  + t \frac{\del u_i}{\del t} }} K_{t, n}\pr{h_i\pr{x}} d{x} \\
&\le 2n \int_{\R^d} \innp{A_i \gr u_i, \gr u_i} K_{t, n}\pr{h_i(x)} d{x}
+ \frac {2t}{d} \int_{\R^d} \pr{1 + \frac{\abs{h_i}^2}{2dnt}} \abs{\frac{\del u_i}{\del t}}^2  K_{t, n}\pr{h_i(x)} d{x} \\
&\le \frac{2n\mathcal{C}_d}{\pr{4 \pi t}^{\frac d 2}} \pr{ \int_{\R^d} \innp{A_i \gr u_i, \gr u_i} e^{-\frac{\abs{h_i(x)}^2}{4t}} d{x}
+ \frac{2t}{dn}  \int_{\R^d} \abs{\frac{\del u_i}{\del t}}^2 e^{-\frac{\abs{h_i(x)}^2}{4t}} d{x}} \\
&= \frac{2n\mathcal{C}_d}{\pr{4 \pi t}^{\frac d 2}} \brac{ \mathcal{D}\pr{t; u_i, \kappa_i}
+ \frac{2t}{dn}  \mathcal{T}\pr{t; u_i, \kappa_i}},
\end{align*}
where we applied Lemma \ref{ChainRuleLem}, Cauchy-Schwarz, then estimate \eqref{KntKtBound}.
If we define 
$$\mathcal{S}_n\pr{t; u, \kappa} = \mathcal{D}\pr{t; u, \kappa}
+ \frac{2t}{dn}  \mathcal{T}\pr{t; u, \kappa},$$ 
then we similarly deduce from \eqref{grSqTerm} that
\begin{equation*}
\begin{aligned}
\int_{B_t^n} \kappa_{i, n}(y) \abs{\gr v_{i,n}(y) }^2 &  \abs{y}^{2 + \Upsilon_i -d n} dy 
\le \frac{\pr{2 d}^{\frac{\Upsilon_i}2 + 1} \mathcal{C}_d n \abs{S^{d n - 1 }}}{\pr{4 \pi}^{\frac d 2}} \int_0^t s^{\frac{\Upsilon_i - d}2} \mathcal{S}_n\pr{s; u_i, \kappa_i} ds.
\end{aligned}
\end{equation*}
Finally, Lemma \ref{PFT} and \eqref{KntKtBound} show that
\begin{align*}
\int_{B_n^t} \kappa_{i,n} \ell_{i,n}^-  v_{i,n} \abs{y}^{2 + \Upsilon_i -n d} \, d{y}
&= d \abs{S^{d n -1}} \int_0^t \int_{\R^d} \pr{2 d s}^{\frac{\Upsilon_i}2} J_i^-(x, s) \, u_i\pr{x, s} \, K_n(h_i(x),s)dx ds \\
&\le \frac{d \pr{2 d}^{\frac{\Upsilon_i}2}\mathcal{C}_d \abs{S^{d n -1}}}{\pr{4 \pi}^{\frac d 2}} \int_0^t \int_{\R^d} s^{\frac{\Upsilon_i - d}2} J_i^-(x, s) \, u_i\pr{x, s} \, e^{-\frac{\abs{h_i(x)}^2}{4s}} dx ds \\
&= \frac{d \pr{2 d}^{\frac{\Upsilon_i}2}\mathcal{C}_d \abs{S^{d n -1}}}{\pr{4 \pi}^{\frac d 2}} \int_0^t  s^{\frac{\Upsilon_i - d}2} \mathcal{J}^-\pr{s; u_i, \kappa_i} ds.
\end{align*}
Then with $\al_{d,n}$ as defined in \eqref{aldnDefn}, we get
\begin{align*}
\widetilde \phi^\prime\left(\sqrt{2 d t}\right)
&\ge -2 t^{- \frac 3 2 + 2 \pr{\frac{\La - \la}{\La}}} \frac{\mathcal{C}_d^3 \pr{2 d}^{\frac{1 + \mu} 2}  \pr{n \abs{S^{d n - 1 }}}^2}{\pr{d n-2} \pr{4 \pi}^{d} \al_{d,n}}
\frac{\mathcal{S}_n\pr{t; u_1, \kappa_1}}{\mathcal{H}_n\pr{t; u_1, \kappa_1}} \\
&\times 
\pr{ \int_0^t  s^{\frac{\Upsilon_1 - d}2}  \mathcal{J}^-\pr{s; u_1, \kappa_1} ds} 
\pr{\int_0^t s^{\frac{\Upsilon_2 - d}2} \mathcal{S}_n\pr{s; u_2, \kappa_2} ds }   \\
&-  2 t^{- \frac 3 2 + 2 \pr{\frac{\La - \la}{\La}}} \frac{\mathcal{C}_d^3 \pr{2 d}^{\frac{1+\mu} 2}  \pr{n \abs{S^{d n - 1 }}}^2}{\pr{d n-2} \pr{4 \pi}^{d} \al_{d,n}} 
\frac{\mathcal{S}_n\pr{t; u_2, \kappa_2}}{\mathcal{H}_n\pr{t; u_2, \kappa_2}} \\
&\times 
\pr{ \int_0^t  s^{\frac{\Upsilon_2 - d}2}  \mathcal{J}^-\pr{s; u_2, \kappa_2} ds} 
\pr{\int_0^t s^{\frac{\Upsilon_1 - d}2} \mathcal{S}_n\pr{s; u_1, \kappa_1} ds } .
\end{align*}
Using \eqref{aldnBound} then shows that
\begin{align*}
\frac{d}{dt} &\widetilde \Phi_n(t)
= \frac{2 \pr{2d}^{\frac{1- \mu} 2}}{\pr{n \abs{S^{d  n - 1}}}^2} \frac 1 {\sqrt t} \widetilde \phi^\prime\left(\sqrt{2 d t}\right)  \\
\ge& - \frac{8 \mathcal{C}_d^3 t^{-\frac{2 \la}{\La}}}{n\pr{1-\frac 2{dn}} \pr{4 \pi}^{d}\al_{d}} 
\frac{\mathcal{S}_n\pr{t; u_1, \kappa_1}}{\mathcal{H}_n\pr{t; u_1, \kappa_1}}
\pr{ \int_0^t  s^{\frac{\Upsilon_1 - d}2}  \mathcal{J}^-\pr{s; u_1, \kappa_1} ds} 
\pr{\int_0^t s^{\frac{\Upsilon_2 - d}2} \mathcal{S}_n\pr{s; u_2, \kappa_2} ds }   \\
&- \frac{8 \mathcal{C}_d^3 t^{-\frac{2 \la}{\La}}}{n\pr{1-\frac 2{dn}} \pr{4 \pi}^{d} \al_{d}} 
\frac{\mathcal{S}_n\pr{t; u_2, \kappa_2}}{\mathcal{H}_n\pr{t; u_2, \kappa_2}} 
\pr{ \int_0^t  s^{\frac{\Upsilon_2 - d}2}  \mathcal{J}^-\pr{s; u_2, \kappa_2} ds} 
\pr{\int_0^t s^{\frac{\Upsilon_1 - d}2} \mathcal{S}_n\pr{s; u_1, \kappa_1} ds } \\
&=: \widetilde{F}_n(t),
\end{align*}

As in the proof of Theorem \ref{T:parabolicALM}, to show that $\widetilde{\Phi}$ is monotone non-decreasing, it suffices to show that given any $t_0 \in (0, T]$, there exists $\de \in \pr{0, t_0}$ so that $\widetilde{F}_n$ converges uniformly to $0$ on $\brac{t_0 - \de, t_0}$.

We first consider the terms in the denominator of $\widetilde F_n$. 
For brevity, we set $\mathcal{H}_{n}^{(i)}(t) = \mathcal{H}_{n}\pr{t; u_i, \kappa_i}$.
By repeating the arguments from \eqref{denomFn} through \eqref{DnBound}, we deduce that there exists $N_i \in \N$ and $\de_i \in \pr{0, t_0}$ so that whenever $n \ge N_i$ and $t \in \brac{t_0 - \de_i, t_0}$, it holds that
\begin{equation*}
\mathcal{H}^{(i)}_n\pr{t}
\ge H_i e^{- \frac{dN_i \ln 2 }{2} }.
\end{equation*}
Assumption \eqref{ACF2} shows that
\begin{align*}
\int_0^t  s^{\frac{\Upsilon_i - d}2}  \mathcal{J}^-\pr{s; u_i, \kappa_i} ds
\le t^{\frac{\Upsilon_i}2} \norm{\mathcal{J}^-\pr{\cdot; u_i, \kappa_i}}_{L^1\pr{\brac{0, t}, s^{- \frac d 2} ds}}.
\end{align*}
Similarly, since $u_i \in \mathfrak{C}\pr{\R^d \times \pr{0, T}, h_i}$ implies that $u_i$ has moderate $h_i$-growth at infinity, then both $\mathcal{D}\pr{\cdot; u_i, h_i}$ and $\mathcal{T}\pr{\cdot; u_i, h_i}$ belong to weighted $L^1\brac{0, t}$ and then
\begin{align*}
\int_0^t  s^{\frac{\Upsilon_i - d}2}  \mathcal{S}_n\pr{s; u_i, \kappa_i} ds
\le t^{\frac{\Upsilon_i}2}\pr{ \norm{\mathcal{D}\pr{\cdot; u_i, \kappa_i}}_{L^1\pr{\brac{0, t}, s^{- \frac d 2} ds}} 
+ \frac{2t}{dn} \norm{\mathcal{T}\pr{\cdot; u_i, \kappa_i}}_{L^1\pr{\brac{0, t}, s^{- \frac d 2} ds}}}.
\end{align*}
Set $\de = \min\set{\de_1, \de_2, \frac{t_0}2}$, $N = \max\set{N_1, N_2}$, and $H = \min\set{H_1, H_2}$, then observe that whenever $t \in \brac{t_0 - \de, t_0}$ and $n \ge N$,
\begin{align*}
\widetilde{F}_n(t)
&\ge - \frac{9 \mathcal{C}_d^3 e^{ \frac{dN \ln 2 }{2}} t^{\frac{\mu}2-2}}{n H \pr{4 \pi}^{d}\al_{d}}
\norm{\mathcal{J}^-\pr{\cdot; u_1, \kappa_1}}_{L^1\pr{\brac{0, t}, s^{- \frac d 2} ds}}
\mathcal{S}_n\pr{t; u_1, \kappa_1} \\
&\times \brac{ \norm{\mathcal{D}\pr{\cdot; u_2, \kappa_2}}_{L^1\pr{\brac{0, t}, s^{- \frac d 2} ds}} 
+ \frac{2t}{dn} \norm{\mathcal{T}\pr{\cdot; u_2, \kappa_2}}_{L^1\pr{\brac{0, t}, s^{- \frac d 2} ds}} } \\
& - \frac{9 \mathcal{C}_d^3 e^{ \frac{dN \ln 2 }{2}} t^{\frac{\mu}2-2}}{n H \pr{4 \pi}^{d}\al_{d}} 
\norm{\mathcal{J}^-\pr{\cdot; u_2, \kappa_2}}_{L^1\pr{\brac{0, t}, s^{- \frac d 2} ds}}
  \mathcal{S}_n\pr{t; u_2, \kappa_2} \\
&\times \brac{ \norm{\mathcal{D}\pr{\cdot; u_1, \kappa_1}}_{L^1\pr{\brac{0, t}, s^{- \frac d 2} ds}} 
+ \frac{2t}{dn} \norm{\mathcal{T}\pr{\cdot; u_1, \kappa_1}}_{L^1\pr{\brac{0, t}, s^{- \frac d 2} ds}} }.
\end{align*}
Assuming that $\de \le \eps$, where $\eps$ is from assumption \eqref{ACF1}, and using that $\disp t \in \brac{\frac{t_0}2, t_0}$, it follows that 
\begin{align*}
\inf_{t \in \brac{t_0 - \de, t_0}} \widetilde{F}_n(t)
&\ge - \frac{36 \mathcal{C}_d^3 e^{ \frac{dN \ln 2 }{2}} t_0^{\frac{\mu}2 -2}}{n H \pr{4 \pi}^{d}\al_{d}}
\brac{ \norm{\mathcal{D}\pr{\cdot; u_2, \kappa_2}}_{L^1\pr{\brac{0, t_0}, t^{- \frac d 2} dt}} 
+ \frac{2t_0}{dn} \norm{\mathcal{T}\pr{\cdot; u_2, \kappa_2}}_{L^1\pr{\brac{0, t_0}, t^{- \frac d 2} dt}} } \\
&\times \norm{\mathcal{J}^-\pr{\cdot; u_1, \kappa_1}}_{L^1\pr{\brac{0, t_0}, t^{- \frac d 2} dt}}
\brac{\norm{\mathcal{D}\pr{\cdot; u_1, \kappa_1}}_{L^\iny\brac{t_0 - \de, t_0}} 
+ \frac{2t_0}{dn} \norm{\mathcal{T}\pr{\cdot; u_1, \kappa_1}}_{L^\iny\brac{t_0 - \de, t_0}}} \\
& - \frac{36 \mathcal{C}_d^3 e^{ \frac{dN \ln 2 }{2}} t_0^{\frac{\mu}2-2}}{n H \pr{4 \pi}^{d}\al_{d}} 
\brac{ \norm{\mathcal{D}\pr{\cdot; u_1, \kappa_1}}_{\pr{\brac{0, t_0}, t^{- \frac d 2} dt}} 
+ \frac{2t_0}{dn} \norm{\mathcal{T}\pr{\cdot; u_1, \kappa_1}}_{\pr{\brac{0, t_0}, t^{- \frac d 2} dt}} } \\
&\times \norm{\mathcal{J}^-\pr{\cdot; u_2, \kappa_2}}_{\pr{\brac{0, t_0}, t^{- \frac d 2} dt}}
\brac{\norm{\mathcal{D}\pr{\cdot; u_2, \kappa_2}}_{L^\iny\brac{t_0 - \de, t_0}} 
+ \frac{2t_0}{dn} \norm{\mathcal{T}\pr{\cdot; u_2, \kappa_2}}_{L^\iny\brac{t_0 - \de, t_0}}}.
\end{align*}
In particular, this shows that $\widetilde F_n$ converges uniformly to $0$, as required. 

We have shown that the proof is complete under the assumption that the measures of $\text{supp} \, u_1\pr{\cdot, t}$ and $\text{supp} \, u_2\pr{\cdot, t}$ are non-vanishing for every $t$.

Now assume that there exists some values of $t$ such that the measure of $\text{supp} \, u_1(\cdot, t)$ or the measure of $\text{supp} \, u_2(\cdot, t)$ vanishes.
Let $\tau$ be the largest such $t$-value.
Without loss of generality, we may assume that $|\text{supp} \, u_1(\cdot, \tau)| = 0$.
Since $\Ga_{1, n, \tau} = \text{supp} \ v_{1, n} \cap S_\tau^n$, it follows that $|\Ga_{1, n, \tau}| = 0$ for every $n$ as well.
Therefore, for every $n$, $\di\pr{\kappa_{1,n} \gr v_{1,n}} \ge \kappa_{1,n} \ell_{1, n}$ in $D_{1, n, \tau} := \text{supp} \, v_{1,n} \cap B_\tau^n$ with $v_{1, n} = 0$ along $\del D_{1, n, \tau}$.
By the arguments used to reach estimate \eqref{gradKTermEst} applied to $v_{1,n}$ on $D_{1,n,\tau}$, we see that
\begin{align*}
\int_{D_{1, n, \tau}} \kappa_{1, n} |\gr v_{1,n}|^2 |y|^{2 + \Upsilon_1 - d n} d{y} 
\le - \int_{D_{1, n, \tau}} \kappa_{1, n} \ell_{1,n} v_{1,n}  |y|^{2 + \Upsilon_1 - d  n} dy
\end{align*}
so that
\begin{align*}
\int_{B_\tau^n}  \kappa_{1, n} |\gr v_{1,n}|^2 |y|^{2 + \Upsilon_1 - d n} d{y} 
\le - \int_{B_\tau^n}  \kappa_{1, n} \ell_{1,n} v_{1,n}  |y|^{2 + \Upsilon_1 - d n} dy.
\end{align*}
As shown above,
\begin{align*}
\frac{\pr{2 d}^{-\frac{\Upsilon_1}2}}{d n\abs{S^{d n - 1 }}} & \int_{B_\tau^n}  \kappa_{1, n} |\gr v_{1,n}|^2 |y|^{2 + \Upsilon_1 - d n} d{y} \\
&= \int_0^t \int_{\R^d} s^{\frac{\Upsilon_1}2} \brac{\innp{A_1 \gr u_1, \gr u_1} + \frac{2}{dn} \frac{\del u_1}{\del s} \pr{\innp{A_1 \gr u_{1}, H_1^T h_1} + s \frac{\del u_1}{\del s} }} K_n(h_1(x),s)dx ds
\end{align*}
and 
\begin{align*}
\frac{\pr{2 d}^{-\frac{\Upsilon_1}2}}{dn \abs{S^{d n - 1 }}} \int_{B_\tau^n}  \kappa_{1, n} \ell_{1,n} v_{1,n}  |y|^{2 + \Upsilon_1 - d n} dy
&= \frac 1 n \int_0^t \int_{\R^d} s^{\frac{\Upsilon_1}2} J_1^-(x, s) \, u_1\pr{x, s} \, K_n(h_1(x),s)dx ds
\end{align*}
from which it follows that for every $n \in \N$,
\begin{align*}
 &\int_0^t \int_{\R^d} s^{\frac{\Upsilon_1}2} \innp{A_1 \gr u_1, \gr u_1} K_n(h_1(x),s)dx ds \\
 &= \int_0^t \int_{\R^d} s^{\frac{\Upsilon_1}2} \brac{\innp{A_1 \gr u_1, \gr u_1} + \frac{2}{dn} \frac{\del u_1}{\del s} \pr{\innp{A_1 \gr u_{1}, H_1^T h_1} + s \frac{\del u_1}{\del s} }} K_n(h_1(x),s)dx ds \\
&- \frac{2}{dn} \int_0^t \int_{\R^d} s^{\frac{\Upsilon_1}2} \frac{\del u_1}{\del s} \pr{\innp{A_1 \gr u_{1}, H_1^T h_1} + s \frac{\del u_1}{\del s} } K_n(h_1(x),s)dx ds \\
&\le \frac 1 n \int_0^t \int_{\R^d} s^{\frac{\Upsilon_1}2} \brac{- J_1^-(x, s) \, u_1\pr{x, s} - \frac 2 d \frac{\del u_1}{\del s} \pr{\innp{A_1 \gr u_{1}, H_1^T h_1} + s \frac{\del u_1}{\del s} }}\, K_n(h_1(x),s)dx ds \\
&\le - \frac 2 {dn} \int_0^t \int_{\R^d} s^{\frac{\Upsilon_1}2} \brac{\frac d 2 J_1(x, s) \, u_1\pr{x, s} +  \frac{\del u_1}{\del s} \pr{\innp{A_1 \gr u_{1}, H_1^T h_1} + s \frac{\del u_1}{\del s} }}\, K_n(h_1(x),s)dx ds,
\end{align*}
since $u_1 \ge 0$.
With $J$ as given in \eqref{JDefn}, we see that
\begin{align*}
& \frac d 2 J_1(x, s) \, u_1\pr{x, s} 
+ \frac{\del u_1}{\del s} \pr{\innp{A_1 \gr u_{1}, H_1^T h_1} + s \frac{\del u_1}{\del s} } \\
&= \brac{ \innp{A_1 \gr\frac{\del u_1}{\del s}, H_1^T h_1}  
+ \frac 1 2 \frac{\del u_1}{\del s} \tr\pr{H_1 \innp{\gr_z G_1(h_1), h_1}}
+ s \frac{\del^2 u_1}{\del s^2} } \, u_1\pr{x, s} 
+ \frac{\del u_1}{\del s} \brac{\innp{A_1 \gr u_{1}, H_1^T h_1} + s \frac{\del u_1}{\del s} } \\
&= \innp{A_1 \gr w_1, H_1^T h_1}
+ s \frac{\del w_1}{\del s}
+ \frac {w_1}2 \tr\pr{H_1 \innp{\gr_z G_1(h_1), h_1}},
\end{align*}
where we have introduced the notation $w_1 = u_1 \frac{\del u_1}{\del s}$.
Therefore,
\begin{align*}
 &\int_0^t \int_{\R^d} s^{\frac{\Upsilon_1}2} \innp{A_1 \gr u_1, \gr u_1} K_n(h_1(x),s)dx ds \\
&\le - \frac {2 \mathcal{C}_d} {dn} \int_0^t \int_{\R^d} s^{\frac{\Upsilon_1}2} \brac{\innp{A_1 \gr w_1, H_1^T h_1}
+ s \frac{\del w_1}{\del s}
+ \frac{w_1} 2  \tr\pr{H_1 \innp{\gr_z G_1(h_1), h_1}}}\, K(h_1(x),s)dx ds.
\end{align*}
In particular,
\begin{align*}
&\int_0^t \int_{\R^d} s^{\frac{\Upsilon_1}2} \innp{A_1 \gr u_1, \gr u_1} K_n(h_1(x),s)dx ds \\
&\le - \lim_{n \to \iny} \frac {2 \mathcal{C}_d} {dn} \int_0^t \int_{\R^d} s^{\frac{\Upsilon_1}2} \brac{\innp{A_1 \gr w_1, H_1^T h_1}
+ s \frac{\del w_1}{\del s}
+ \frac{w_1} 2  \tr\pr{H_1 \innp{\gr_z G_1(h_1), h_1}}}\, K(h_1(x),s)dx ds 
=0.
\end{align*}
Then $\Phi(t) = 0$ for every $t \le \tau$ showing that $\Phi(t)$ is monotonically non-decreasing on $(0, \tau]$.
Since $|\Ga_{1, n, t}| \ne 0$ and $|\Ga_{2, n, t}| \ne 0$ for every $n$ and every $t > \tau$, then by the arguments from the first case, $\Phi(t,u)$ is monotonically non-decreasing whenever $t > \tau$, completing the proof.
\end{proof}

\begin{appendix}

\section{Computational and Technical Proofs}
\label{AppA}

Within this appendix, we have collected the proofs that are either purely computational or technical in nature.
We begin with the proof of Lemma \ref{ChainRuleLem}, which describes a collection of relationships between elliptic functions $v_n : B_{\sqrt{2 d T}} \su \R^{d \times n} \to \R$ and parabolic functions $u : \R^d \times \pr{0, T} \to \R$ that are related through the transformation $F_{d,n}$.
That is, Lemma \ref{ChainRuleLem} describes the relationships between derivatives of $u$ and $v_n$ whenever $v_n(y) = u \pr{F_{d,n}(y)} = u (x, t)$.
The proof of this result relies on the chain rule.

\begin{pf}[Proof of Lemma \ref{ChainRuleLem}]
\label{p:ChainLemProof}
Since $\displaystyle \frac{\del z_k}{\del y_{i,j}} = \de_{k, i}$, then
\begin{align*}
\frac{\del v_n}{\del y_{i,j}}
&= \frac{\del u}{\del x_1} \frac{\del g_1}{\del z_i} + \ldots + \frac{\del u}{\del x_d} \frac{\del g_d}{\del z_i}
+ \frac{\del u}{\del t} \frac{y_{i, j}}{d}
= \innp{\gr_x u, \frac{\del g}{\del z_i}}  + \frac{\del u}{\del t} \frac{y_{i, j}}{d},
\end{align*}
where $\displaystyle \frac{\del g}{\del z_i} $ is the $i^{th}$ column of the Jacobian matrix $G$.
Then we have
\begin{align*}
    y \cdot \gr_y v_n
    &= \sum_{i=1}^d \sum_{j = 1}^n y_{i, j}\brac{\sum_{k=1}^d \frac{\del u}{\del x_k} \frac{\del g_k}{\del z_i}  + \frac{\del u}{\del t} \frac{y_{i, j}}{d}}
    = \sum_{i, k=1}^d \frac{\del u}{\del x_k} \frac{\del g_k}{\del z_i} z_{i}
    + \frac{\del u}{\del t} \sum_{i=1}^d \sum_{j = 1}^n \frac{y_{i, j}^2}{d} \\
    &= \innp{\gr_x u, G(z) z} + 2 t \frac{\del u}{\del t}
    = \innp{A(x)\gr_x u, H(x)^T h(x)} + 2 t \frac{\del u}{\del t}
\end{align*}
and
\begin{align*}
    \abs{\gr_y v_n}^2 
    &= \sum_{i=1}^d \sum_{j = 1}^n \pr{\frac{\del v_n}{\del y_{i,j}}}^2
    = \sum_{i=1}^d \sum_{j = 1}^n\brac{\sum_{k=1}^d \frac{\del u}{\del x_k} \frac{\del g_k}{\del z_i}  + \frac{\del u}{\del t} \frac{y_{i, j}}{d}} \brac{\sum_{\ell=1}^d \frac{\del u}{\del x_\ell} \frac{\del g_\ell}{\del z_i}  + \frac{\del u}{\del t} \frac{y_{i, j}}{d}} \\
    &= n \sum_{i, k, \ell=1}^d \pr{ \frac{\del u}{\del x_k} \frac{\del g_k}{\del z_i} } \pr{ \frac{\del u}{\del x_\ell} \frac{\del g_\ell}{\del z_i} }
    + \frac 2 d \frac{\del u}{\del t} \sum_{i, k=1}^d \pr{\frac{\del u}{\del x_k} \frac{\del g_k}{\del z_i}  } z_{i}    
    + \pr{\frac{\del u}{\del t}}^2 \sum_{i=1}^d \sum_{j = 1}^n \pr{ \frac{y_{i, j}}{d}}^2 \\
    &= n \abs{G(z)^T \gr_x u}^2 
    + \frac 2 d \frac{\del u}{\del t} \brac{\innp{\gr_x u, G(z) \, z}
    + t \frac{\del u}{\del t} }.
\end{align*}

Now we look at the second-order derivatives.
Since $\displaystyle \kappa_n(y) \frac{\del v_n}{\del y_{i,j}} = \sum_{k=1}^d \frac{\del u}{\del x_k} \ga(z) \frac{\del g_k}{\del z_i}  + \frac{\del u}{\del t} \ga(z) \frac{y_{i, j}}{d}$, then
\begin{align*}
\frac{\del}{\del y_{i,j}}\pr{ \kappa_n(y)\frac{\del v_n}{\del y_{i,j}}}
&= \frac{\del}{\del y_{i,j}}\pr{\sum_{k=1}^d \frac{\del u}{\del x_k} \ga(z) \frac{\del g_k}{\del z_i}} + \frac{\del}{\del y_{i,j}}\pr{\frac{\del u}{\del t} \gamma(z)  \frac{y_{i, j}}{d}} \\
&= \sum_{k, \ell=1}^d \frac{\del^2 u}{\del x_\ell \del x_k} \ga \frac{\del g_k}{\del z_i} \frac{\del g_\ell}{\del z_i} 
+ 2\sum_{k=1}^d \frac{\del^2 u}{\del t \del x_k} \ga \frac{\del g_k}{\del z_i} \frac{y_{ij}}{d}
+ \sum_{k=1}^d \frac{\del u}{\del x_k} \frac{\del \ga}{\del z_{i}} \frac{\del g_k}{\del z_i}
+ \sum_{k=1}^d \frac{\del u}{\del x_k} \ga \frac{\del^2 g_k}{\del z_i^2} \\
&+\frac{\del^2 u}{\del t^2} \gamma \pr{\frac{y_{i, j}}{d}}^2
+ \frac{\del u}{\del t} \frac{\del \ga}{\del z_i}  \frac{y_{i, j}}{d}
+ \frac{\del u}{\del t} \ga \frac{1}{d},
\end{align*}
from which it follows that
\begin{equation}
\label{chainRule1}
\begin{aligned}
\frac{1}{\kappa_n(y)}\frac{\del}{\del y_{i,j}}\pr{ \kappa_n(y)\frac{\del v_n}{\del y_{i,j}}}
&= \innp{D^2_x u \, \frac{\del g}{\del z_i}, \frac{\del g}{\del z_i}}
+ \innp{\gr_x u, \frac{\del \log \ga}{\del z_{i}} \frac{\del g}{\del z_i} + \frac{\del^2 g}{\del z_{i}^2}}  
+  \frac 1 d \frac{\del u}{\del t} \\
&+ 2 \innp{\gr_x \pr{\frac{\del u}{\del t}}, \frac{\del g}{\del z_i} \frac{y_{i, j}}{d}}
+ \frac{\del \log \ga}{\del z_{i}} \frac{\del u}{\del t}\frac{y_{i, j}}{d}
+ \frac{\del^2 u}{\del t^2} \pr{\frac{y_{i, j}}{d}}^2.
\end{aligned}
\end{equation}
Because $B(z) = G(z) G^T(z) = A(x)$, then
\begin{align}
\label{2ndOrderTerm}
\sum_{i =1}^d \innp{D^2_x u \, \frac{\del g}{\del z_i}, \frac{\del g}{\del z_i}}
&= \sum_{i, k, \ell=1}^d \frac{\del^2 u}{\del x_k \del x_\ell} \frac{\del g_k}{\del z_i} \frac{\del g_\ell }{\del z_{i}}
= \sum_{k, \ell=1}^d b_{k, \ell}(z) \frac{\del^2 u}{\del x_k \del x_\ell}
= \sum_{k, \ell=1}^d a_{k, \ell}(x) \frac{\del^2 u}{\del x_k \del x_\ell}.
\end{align}
Since $\displaystyle a_{k, \ell}(x) = \sum_{i=1}^d \frac{\del g_k}{\del z_i}\pr{h(x)} \frac{\del g_\ell}{\del z_i}\pr{h(x)}$, then
\begin{equation}
\label{coefGrad}
\begin{aligned}
\frac{\del a_{k, \ell}}{\del x_k}
&= \sum_{i=1}^d \frac{\del}{\del x_k}\pr{\frac{\del g_k}{\del z_i}\pr{h(x)}} \frac{\del g_\ell}{\del z_i}\pr{h(x)}
+ \sum_{i=1}^d \frac{\del g_k}{\del z_i}\pr{h(x)} \frac{\del}{\del x_k}\pr{\frac{\del g_\ell}{\del z_i}\pr{h(x)}} \\
&= \sum_{i, m=1}^d \frac{\del^2 g_k\pr{z}}{\del z_i \del z_m} \frac{\del g_\ell\pr{z}}{\del z_i} \frac{\del h_m}{\del x_k}\pr{g(z)}
+ \sum_{i, m=1}^d \frac{\del^2 g_\ell\pr{z}}{\del z_i \del z_m} \frac{\del h_m}{\del x_k}\pr{g(z)} \frac{\del g_k\pr{z}}{\del z_i}.
\end{aligned}
\end{equation}
Summing \eqref{coefGrad} over $k$, substituting \eqref{logExpression} and using that $\displaystyle \sum_{k =1}^d \frac{\del h_m}{\del x_k} \frac{\del g_k}{\del z_i} = \de_{m, i}$ since $HG = I$, shows that
\begin{equation}
\label{1stOrderTerm}
\begin{aligned}
\sum_{k =1}^d \frac{\del a_{k, \ell}}{\del x_k}
&= \sum_{i, m, k=1}^d \frac{\del h_m}{\del x_k}\pr{g(z)} \frac{\del^2 g_k\pr{z}}{\del z_i \del z_m} \frac{\del g_\ell\pr{z}}{\del z_i} 
+ \sum_{i, m, k=1}^d \frac{\del^2 g_\ell\pr{z}}{\del z_i \del z_m} \frac{\del h_m}{\del x_k}\pr{g(z)} \frac{\del g_k\pr{z}}{\del z_i} \\
&= \sum_{i=1}^d  \frac{\del \log \ga(z)}{\del z_i} \frac{\del g_\ell\pr{z}}{\del z_i}
+ \sum_{i =1}^d \frac{\del^2 g_\ell\pr{z}}{\del z_i^2}.
\end{aligned}
\end{equation}
Therefore,
\begin{equation}
\label{divTerm}
\begin{aligned}
\di_x \pr{A \gr_x u} 
&= \sum_{k = 1}^d \frac{\del }{\del x_k} \pr{\sum_{\ell = 1}^d a_{k, \ell} \frac{\del u}{\del x_\ell} }
= \sum_{k, \ell = 1}^d a_{k, \ell} \frac{\del^2 u}{\del x_\ell \del x_k}
+ \sum_{\ell = 1}^d \pr{\sum_{k = 1}^d\frac{\del a_{k, \ell}}{\del x_k} } \frac{\del u}{\del x_\ell}  \\
&=\sum_{i =1}^d \innp{D^2_x u \, \frac{\del g}{\del z_i}, \frac{\del g}{\del z_i}}
+ \innp{\gr_x u, \frac{\del \log \ga}{\del z_{i}} \frac{\del g}{\del z_i} + \frac{\del^2 g}{\del z_i ^2} },
\end{aligned}
\end{equation}
where we have used \eqref{2ndOrderTerm} and \eqref{1stOrderTerm}.
Summing \eqref{chainRule1} over $i$ and $j$ and substituting \eqref{divTerm} then shows that
\begin{equation*}
\begin{aligned}
\frac{1}{\kappa_n(y)} \di_y \pr{ \kappa_n(y)\gr_y v_n}
&= \frac{1}{\kappa_n(y)} \sum_{i=1}^d \sum_{j=1}^n \frac{\del}{\del y_{i,j}}\pr{ \kappa_n(y)\frac{\del v_n}{\del y_{i,j}}} \\
&= \sum_{j=1}^n \sum_{i=1}^d \brac{\innp{D^2_x u \, \frac{\del g}{\del z_i}, \frac{\del g}{\del z_i}}
+ \innp{\gr_x u, \frac{\del \log \ga}{\del z_{i}} \frac{\del g}{\del z_i} + \frac{\del^2 g}{\del z_{i}^2}}  
+  \frac 1 d \frac{\del u}{\del t}} \\
&+  \sum_{i=1}^d \sum_{j=1}^n \brac{2 \innp{\gr_x \pr{\frac{\del u}{\del t}}, \frac{\del g}{\del z_i} \frac{y_{i, j}}{d}}
+ \frac{\del \log \ga}{\del z_{i}} \frac{\del u}{\del t}\frac{y_{i, j}}{d}
+ \frac{\del^2 u}{\del t^2} \pr{\frac{y_{i, j}}{d}}^2
} \\
&= \sum_{j=1}^n \brac{\di_x \pr{A \gr_x u} +  \frac{\del u}{\del t}} 
+  \frac 2 d \sum_{i, k = 1}^d \frac{\del^2 u}{\del x_k \del t} \frac{\del g_k}{\del z_i}(z) z_i
+ \frac 1 d \frac{\del u}{\del t} \sum_{i =1}^d \frac{\del \log \ga}{\del z_i}(z) z_i
+ \frac{2t}{d} \frac{\del^2 u}{\del t^2} \\
&= n \set{\di_x \pr{A \gr_x u} +  \frac{\del u}{\del t}
+ \frac 1 {dn}\brac{ 2 \innp{\gr_x \frac{\del u}{\del t}, G(z) z}  
+ \frac{\del u}{\del t} \innp{\gr \log \ga(z), z}
+ 2t \frac{\del^2 u}{\del t^2} }} ,
\end{aligned}
\end{equation*}
and the result follows from changing the last set of terms to depend on $x$. 
\end{pf}

Here we prove Lemma \ref{L:uniform} regarding the uniform convergence of our sequence of approximations to the Gaussian.

\begin{proof}[Proof of Lemma \ref{L:uniform}]
\label{p:uniform}
Let 
\begin{align}
\label{aldnDefn}
\alpha_{d,n}
&=\frac{|S^{d n -1-d}|}{|S^{d  n-1}|}\left(\frac{2\pi}{dn}\right)^{\frac{d}{2}}
= \pr{\frac{2 \pi}{d n}}^{\frac{d}{2}} \frac{2 \pi^{\frac{d n- d}{2}}}{\Ga\pr{\frac{d  n- d}{2}}} \frac{\Ga\pr{\frac{d n}{2}}}{2 \pi^{\frac{d  n}{2}}}
=  \pr{\frac{2}{d n}}^{\frac{d}{2}} \frac{\Ga\pr{\frac{ d  n}{2}}}{\Ga\pr{\frac{d  n}{2} - \frac d 2}} .
\end{align}
If $\frac d 2 \in \Z$, then
\begin{align*}
\Ga\pr{\frac{d  n}{2}}
&= \pr{\frac {d  n}{2} - 1} \pr{\frac {d  n}{2} - 2} \ldots \pr{\frac {d  n}{2} - \frac d 2} \Ga\pr{\frac{d  n}{2} - \frac d 2} \\
&= \pr{\frac{d n}{2}}^{\frac d 2} \pr{1 - \frac{2}{d  n}} \pr{1 - \frac{4}{d  n}} \ldots \pr{1 - \frac{d}{d  n}} \Ga\pr{\frac{d  n}{2} - \frac d 2}
\end{align*}
and we see that
$$\al_{d,n} = \prod_{j=1}^{\frac d 2} \pr{1 - \frac{2j}{d n}}.$$
Stirling's formula shows that
$$\Ga(z) = \sqrt{\frac{2\pi}{z}} \pr{\frac z e}^z \pr{1 + \mathcal{O}\pr{\frac 1 z}}$$
so we see that
\begin{align*}
\al_{d,n} 
&=  \pr{\frac{2}{n d}}^{\frac{d}{2}} \frac{\sqrt{\frac{2\pi}{\frac{d n}{2}}} \pr{\frac {\frac{d  n}{2}} e}^{\frac{d n}{2}} \pr{1 + \mathcal{O}\pr{\frac 2 {d n}}}}{\sqrt{\frac{2\pi}{{\frac{d  n}{2} - \frac d 2}}} \pr{\frac {\frac{d n}{2} - \frac d 2} e}^{\frac{d n}{2} - \frac d 2}\pr{1 + \mathcal{O}\pr{\frac 2 {d  n- d}}} } 
=  \brac{e\pr{1 - \frac 1 n}^{n}}^{-\frac d 2}  \pr{1 - \frac 1 n}^{\frac{d+1}{2}}  \pr{1 + \mathcal{O}\pr{\frac 2 {d n}}}.
\end{align*}
In both cases, it holds that $\al_{d,n} \to 1$ as $n \to \infty$.
Therefore, there exists $\al_d$ so that for every $n \in \N$,
\begin{equation}
\label{aldnBound}
\al_{d,n} \le \al_d.
\end{equation}

To analyze the remaining piece of $K_n(x,t)$, let $m=\frac{d n}{2}$, $\delta=\frac{d+2}{2}$, $z=\frac{|x|^2}{4t}$. 
Notice that $|x|^2<2dnt$ if and only if $0\le z<m$, and $\left(1-\frac{|x|^2}{2dnt}\right)^{\frac{d n-d-2}{2}}=\left(1-\frac{z}{m}\right)^{m-\delta}.$ 
Define 
\[
f_m(z)=\left(1-\frac{z}{m}\right)^{m-\delta}\chi_{\{0\le z< m\}}(z).
\]
We show $f_m$ converges uniformly to $f(z)=e^{-z}\chi_{\{z\ge 0\}}(z)$.  
First, notice that if $z>m$, then $|f(z)-f_m(z)|=e^{-z}<e^{-m}$.

If $0\le z\le m$, then
\begin{equation}
\label{f}
\begin{aligned}
    |f(z)-f_m(z)|
    &=\left|e^{-z}-\left(1-\frac{z}{m}\right)^m+\left(1-\frac{z}{m}\right)^m-\left(1-\frac{z}{m}\right)^{m-\delta}\right|\\
    &\le \left|e^{-z}-\left(1-\frac{z}{m}\right)^m\right|+\left|\left(1-\frac{z}{m}\right)^{m-\delta}-\left(1-\frac{z}{m}\right)^m\right|=: \textrm{(I)} +\textrm{(II)}.
\end{aligned}
\end{equation}
We address (I) by showing there exists a constant $c$ such that for all $m \geq 1$
\[ \max_{0 < z < m} \left| e^{-z} - \left(1 - \frac{z}{m}\right)^m \right| \leq c \frac{(\log{m})^2}{m}.\]
Indeed, we have
$$ \log\left[ \left(1 - \frac{z}{m}\right)^m\right] = m \log \left(1 - \frac{z}{m} \right).$$
For $0 < z < m$, this means that $\log\left(1-\frac{z}{m}\right)\in (0,1)$. 
Since $\log{x} < x-1$ for all $x > 0$, then
$$ \log\left[ \left(1 - \frac{z}{m}\right)^m\right] \leq - z.$$ Consequently, $\left(1-\frac{z}{m}\right)^m$ is smaller than $e^{-z}$. 
It remains to obtain an upper bound on 
$$ \max_{0 \leq z \leq m} \brac{e^{-z} -  \left(1 - \frac{z}{m}\right)^m}.$$
Using a Taylor expansion, we have
$$ \log{(1-x)} = -x - \frac{x^2}{2} + \mathcal{O}(x^3).$$
Thus
$$ \log\left[ \left(1 - \frac{z}{m}\right)^m\right] = -z - \frac{z^2}{2m} + \mathcal{O}\left(\frac{z^3}{m^2}\right),$$
and
$$  \left(1 - \frac{z}{m}\right)^m = e^{-z} e^{-\frac{z^2}{2m}} e^{\mathcal{O}\left(z^3/m^2\right)}.$$
We split the problem into two parts, by bounding
$$ \max_{0 \leq z \leq \alpha} \brac{e^{-z} -  \left(1 - \frac{z}{m}\right)^m}$$
and
$$  \max_{\alpha \leq z \leq m} \brac{e^{-z} -  \left(1 - \frac{z}{m}\right)^m}$$
for some $\al$ to be determined.
For the second part, observe that
$$  \max_{\alpha \leq z \leq m} \brac{e^{-z} -  \left(1 - \frac{z}{m}\right)^m} \leq   \max_{\alpha \leq z \leq m} e^{-z} = e^{-\alpha}.$$
As for the first bound, we note that
$$ e^{-z} -  \left(1 - \frac{z}{m}\right)^m = e^{-z}  \left(1 -e^{-\frac{z^2}{2m}}  e^{\mathcal{O}\left(\frac{z^3}{m^2}\right)}\right).$$
If $\frac{z^2}{2m} \rightarrow 0$, then $\frac{z^3}{m^2} \rightarrow 0$ as well and we can bound this term using another Taylor approximation as
 $$  e^{-z}  \left(1 -e^{-\frac{z^2}{2m}}  e^{\mathcal{O}\left(\frac{z^3}{m^2}\right)}\right) = \mathcal{O}\left(\frac{z^2}{m}\right).$$
 By choosing $\alpha = \log{m}$, we obtain
 $$  \max_{\alpha \leq z \leq m} \brac{e^{-z} -  \left(1 - \frac{z}{m}\right)^m } \leq   \max_{\alpha \leq z \leq m} e^{-z} = e^{-\alpha} \leq \frac{1}{m}$$
 and
 $$ \max_{0 \leq z \leq \alpha} \brac{e^{-z} -  \left(1 - \frac{z}{m}\right)^m} \leq \mathcal{O}\left( \frac{(\log{m})^2}{m} \right).$$
This proves that $\disp \textrm{(I)} \le c \frac{(\log{m})^2}{m}$.

To estimate (II), let $g(x)=(1-x)^{m-\delta}-(1-x)^m$, with $0\le x\le 1$. 
Notice  that $g(0)=g(1)=0$ and $g(x)>0$ for $x\in (0,1)$. 
Moreover,
\[
g'(x)=-(m-\delta)(1-x)^{m-\delta-1}+m(1-x)^{m-1}=(1-x)^{m-1-\delta}\brac{m(1-x)^{\delta}-(m-\delta)}.
\]
Therefore $g'(x)\ge 0$ if and only if $x\le 1-\left(1-\frac{\delta}{m}\right)^{1/\delta}.$ We conclude that $g$ attains its maximum at $x_0=1-\left(1-\frac{\delta}{m}\right)^{1/\delta}$, and
\[
g(x_0)=\frac{\delta}{m}\left(1-\frac{\delta}{m}\right)^{\frac{m}{\delta}-1}.
\]
Consequently, $g(x)\in\left[0,\frac{\delta}{m}\right]$ for all $x\in[0,1]$. 

By combining the estimates obtained for (I) and (II) and \eqref{f}, we conclude that for all $z \ge 0$,
\[
|f(z)-f_m(z)|\le c\frac{\left(\log m\right)^2}{m} +e^{-m}+\frac{\delta}{m}.
\]
In particular, this shows that $f_m \to f$ uniformly, and the conclusion of the lemma follows.
\end{proof}

Finally, we give the computational proof of Lemma \ref{kernelEquation}.

\begin{proof}[Proof of Lemma \ref{kernelEquation}]
Let $u_K(x, t) = K(h(x), t) = \frac{1}{\pr{4 \pi t}^{d/2}} \exp\pr{- \frac{\abs{h(x)}^2}{4t}}$.
Computations show that
\begin{align*}
\frac{\del u_K}{\del t }
&= - \frac d{2t} u_K
+ \frac{\abs{h(x)}^2}{\pr{2t}^2} u_K 
\end{align*}
and
\begin{align*}
\frac{\del u_K}{\del {x_i}}
&= - \frac{1}{2t} \sum_{j=1}^d \frac{\del h_j}{\del {x_i} }\, h_j \, u_K.
\end{align*}
Since $\disp a_{k,i} = \sum_{\ell=1}^d \frac{\del g_k}{\del {z_\ell} }(h(x)) \frac{\del g_i}{\del {z_\ell} }(h\pr{x})$ and $\disp \frac{\del g_i}{\del {z_\ell} }(h\pr{x}) \frac{\del h_j}{ \del{x_i} }(x) = \de_{\ell, j}$, then
\begin{align*}
\sum_{i = 1}^d a_{k,i} \frac{\del u_K}{\del {x_i}}
&= - \frac{1}{2t} \sum_{i, j, \ell=1}^d \frac{\del g_k}{\del {z_\ell} }(h(x)) \frac{\del g_i}{\del {z_\ell} }(h\pr{x}) \frac{\del h_j}{ \del{x_i} }(x) h_j(x) u_K
= - \frac{1}{2t} \sum_{ j=1}^d \frac{\del g_k}{\del {z_j} }(h(x)) h_j(x) u_K.
\end{align*}
Therefore
\begin{align*}
\di \pr{A \gr u_K} 
&= \sum_{k = 1}^d \frac{\del}{\del x_k} \brac{\sum_{i = 1}^d a_{k,i} \frac{\del u_K}{\del {x_i}} }
= - \frac{1}{2t} \sum_{k = 1}^d \frac{\del}{\del x_k} \brac{ \sum_{ j=1}^d \frac{\del g_k}{\del {z_j} }(h(x)) h_j(x) u_K} \\
&= - \frac{1}{2t} \sum_{ j,k=1}^d \frac{\del g_k}{\del {z_j} }(h(x)) h_j(x) \frac{\del u_K}{\del x_k}
- \frac{1}{2t} \sum_{ j,k=1}^d \pr{\frac{\del g_k}{\del {z_j} }(h(x)) \frac{\del h_j}{\del x_k} + \frac{\del^2 g_k}{\del {z_j} \del z_\ell }(h(x)) \frac{\del h_\ell}{\del x_k} h_j}u_K  \\
&= \frac{1}{\pr{2t}^2} \sum_{ j, k, \ell =1}^d \frac{\del g_k}{\del {z_j} }(h(x)) h_j \frac{\del h_\ell}{\del {x_k} } h_\ell u_K
- \frac{1}{2t} \sum_{ j,k=1}^d \pr{\frac{\del g_k}{\del {z_j} }(h(x)) \frac{\del h_j}{\del x_k} + \frac{\del^2 g_k}{\del {z_j} \del z_\ell }(h(x)) \frac{\del h_\ell}{\del x_k} h_j}u_K  \\
&= \frac{\abs{h(x)}^2}{\pr{2t}^2} u_K
- \frac{d}{2t} u_K
- \frac{1}{2t} \sum_{ j, k, \ell=1}^d \frac{\del^2 g_k}{\del {z_j} \del z_\ell }(h(x)) \frac{\del h_\ell}{\del x_k} h_j \, u_K \\
&= \frac{\del u_K}{\del t }
- \frac{1}{2t} \sum_{ j, k, \ell=1}^d \frac{\del^2 g_k}{\del {z_j} \del z_\ell }(h(x)) \frac{\del h_\ell}{\del x_k} h_j \, u_K,
\end{align*}
as claimed.
\label{p:uKEqn}
\end{proof}

\end{appendix}

\bibliographystyle{plain}
\bibliography{Bibliography.bib}
\end{document}